\def\todaysdate{2\textsuperscript{nd} November 2023}
\documentclass[reqno,a4paper]{article}
\usepackage{etex}
\usepackage[T1]{fontenc}
\usepackage[utf8]{inputenc}
\usepackage{lmodern}

\usepackage[style=alphabetic,giveninits=true,backref,backend=biber,isbn=false,url=false,maxbibnames=99,maxalphanames=4]{biblatex}

\renewbibmacro{in:}{}
\newbibmacro{string+doi}[1]{\iffieldundef{doi}{\iffieldundef{url}{#1}{\href{\thefield{url}}{#1}}}{\href{http://dx.doi.org/\thefield{doi}}{#1}}}
\DeclareFieldFormat*{title}{\usebibmacro{string+doi}{\emph{#1}}}
\DefineBibliographyStrings{english}{%
  backrefpage = {$\uparrow$},
  backrefpages = {$\uparrow$},
}

\usepackage{amssymb,amsmath,amsthm}
\usepackage{thmtools,xcolor}
\usepackage{units}
\usepackage{graphicx}
\usepackage[all]{xy}
\usepackage{tikz}
\usetikzlibrary{arrows,calc,positioning,decorations.pathreplacing}
\usepackage{tikz-cd}
\usepackage{relsize}
\usepackage{mhsetup}
\usepackage{mathtools}
\usepackage{stmaryrd}
\usepackage{tocloft}
\usepackage{paralist} 
\usepackage[left=3cm,right=3cm,top=3cm,bottom=3cm]{geometry} 
\usepackage[bottom]{footmisc}
\usepackage[colorlinks,citecolor=blue,linkcolor=blue,urlcolor=blue,filecolor=blue,breaklinks]{hyperref}
\usepackage{footnotebackref}
\usepackage[format=plain,indention=1em,labelsep=quad,labelfont={up},textfont={footnotesize},margin=2em]{caption}
\usepackage[mathscr]{euscript}
\usepackage{accents}
\usepackage{xparse}

\definecolor{lightblue}{rgb}{0.8,0.8,1}

\numberwithin{equation}{section}
\numberwithin{figure}{section}

\definecolor{vdarkred}{rgb}{0.7,0,0}
\declaretheoremstyle[
  spaceabove=\topsep,
  spacebelow=\topsep,
  headpunct=,
  numbered=no,
  postheadspace=1ex,
  headfont=\color{vdarkred}\normalfont\bfseries,
  bodyfont=\normalfont\itshape,
]{colored}
\declaretheoremstyle[
  spaceabove=\topsep,
  spacebelow=\topsep,
  headpunct=,
  numbered=no,
  postheadspace=1ex,
  headfont=\normalfont\bfseries,
  bodyfont=\normalfont\itshape,
]{italic}
\declaretheoremstyle[
  spaceabove=\topsep,
  spacebelow=\topsep,
  headpunct=,
  numbered=no,
  postheadspace=1ex,
  headfont=\normalfont\bfseries,
  bodyfont=\normalfont\upshape,
]{upright}

\declaretheorem[style=italic,name=Theorem,numbered=yes,numberwithin=section]{thm}
\declaretheorem[style=italic,name=Lemma,numbered=yes,numberlike=thm]{lem}
\declaretheorem[style=italic,name=Proposition,numbered=yes,numberlike=thm]{prop}
\declaretheorem[style=italic,name=Corollary,numbered=yes,numberlike=thm]{coro}

\declaretheorem[style=italic,name=Theorem,numbered=yes,numberwithin=section]{athm}

\declaretheorem[style=upright,name=Definition,numbered=yes,numberlike=thm]{defn}
\declaretheorem[style=upright,name=Remark,numbered=yes,numberlike=thm]{rmk}
\declaretheorem[style=upright,name=Example,numbered=yes,numberlike=thm]{eg}

\declaretheorem[style=upright,name=Notation,numbered=yes,numberlike=thm]{notation}
\declaretheorem[style=upright,name=Convention,numbered=yes,numberlike=thm]{convention}

\declaretheorem[style=upright,name=Question,,numbered=yes,numberlike=thm]{question}
\makeatletter
\renewcommand*{\@seccntformat}[1]{\upshape\csname the#1\endcsname.\hspace{1ex}}
\renewcommand*{\section}{\@startsection{section}{1}{\z@}%
	{2.5ex \@plus 1ex \@minus 0.2ex}%
	{1.5ex \@plus 0.2ex}%
	{\normalfont\Large\bfseries}}
\renewcommand*{\subsection}{\@startsection{subsection}{2}{\z@}%
	{2.5ex \@plus 1ex \@minus 0.2ex}%
	{1.5ex \@plus 0.2ex}%
	{\normalfont\large\bfseries}}
\renewcommand*{\subsubsection}{\@startsection{subsubsection}{3}{\z@}%
	{2.5ex \@plus 1ex \@minus 0.2ex}%
	{1.5ex \@plus 0.2ex}%
	{\normalfont\normalsize\bfseries}}
\newcommand*{\subsubsubsection}{\@startsection{paragraph}{4}{\z@}%
	{2.5ex \@plus 1ex \@minus 0.2ex}%
	{1.5ex \@plus 0.2ex}%
	{\normalfont\normalsize\bfseries}}
\makeatother

\newcommand{\incl}[3][right]%
{%
\draw[<-,>=#1 hook] #2 to ($ #2!0.5!#3 $);
\draw[->] ($ #2!0.5!#3 $) to #3;%
}
\newcommand{\inclusion}[5][right]%
{%
\draw[<-,>=#1 hook] #4 to ($ #4!0.5!#5 $) node[#2,font=\small]{#3};
\draw[->] ($ #4!0.5!#5 $) to #5;%
}

\newenvironment{itemizeb}%
{\begin{compactitem}

}%
{\end{compactitem}}
{\begin{compactitem}[#1]

}%
{\end{compactitem}}
{\begin{compactdesc}

}%
{\end{compactdesc}}

\newcommand{\cE}{\mathcal{E}}

\newcommand{\cL}{\mathcal{L}}
\newcommand{\cM}{\mathcal{M}}

\newcommand{\cQ}{\mathcal{Q}}
\newcommand{\cR}{\mathcal{R}}
\newcommand{\cS}{\mathcal{S}}

\newcommand{\cU}{\mathcal{U}}

\newcommand{\bC}{\mathbb{C}}
\newcommand{\bD}{\mathbb{D}}

\newcommand{\bF}{\mathbb{F}}

\newcommand{\bK}{\mathbb{K}}

\newcommand{\bN}{\mathbb{N}}

\newcommand{\bQ}{\mathbb{Q}}

\newcommand{\bS}{\mathbb{S}}

\newcommand{\bZ}{\mathbb{Z}}
\newcommand{\ddd}{\mathbf{D}}
\newcommand{\fL}{\mathfrak{L}}
\newcommand{\fK}{\mathfrak{K}}

\renewcommand{\geq}{\geqslant}
\renewcommand{\leq}{\leqslant}
\renewcommand{\footnoterule}{%
  \kern -3pt
  \hrule width \textwidth height 0.4pt
  \kern 2.6pt
}

\newcommand{\opp}{\ensuremath{\mathrm{op}}}
\newcommand{\too}{\ensuremath{\longrightarrow}}
\newcommand{\rot}{\mathrm{rot}}

\newcommand{\Sym}{\mathrm{Sym}}
\newcommand{\ith}{\mathrm{th}}
\newcommand{\mm}{{\bar m}}

\newcommand{\U}{\ensuremath{\cU}}
\newcommand{\Ab}{\mathbf{Ab}}
\newcommand{\Mod}{\textrm{-}\mathrm{Mod}}
\newcommand{\Tor}{\mathrm{Tor}}
\newcommand{\wt}{\mathrm{wt}}
\newcommand{\pa}{\partial}

\newcommand{\aug}{\mathrm{aug}}
\newcommand{\Cok}{\mathrm{Coker}}
\newcommand{\Ker}{\mathrm{Ker}}
\newcommand{\Image}{\mathrm{Im}}
\newcommand{\Ext}{\mathrm{Ext}}

\newcommand{\odd}{\mathrm{odd}}
\newcommand{\even}{\mathrm{even}}
\newcommand{\st}{\mathrm{st}}
\newcommand{\bj}{\mathbf{j}}

\newcommand{\id}{\mathrm{id}}

\newcommand{\Modulo}[1]{\ (\mathrm{mod}\ #1)}

\definecolor{dred}{rgb}{0.7,0,0}

\definecolor{dblue}{rgb}{0,0,0.8}

\definecolor{dgreen}{rgb}{0,0.5,0}

\begin{document}
\title{\LARGE\bfseries Stable twisted cohomology of the mapping class groups in the exterior powers of the unit tangent bundle homology}
\author{\normalsize Nariya Kawazumi and Arthur Souli{\'e}}
\date{\normalsize\todaysdate}
\maketitle
{\makeatletter
\renewcommand*{\BHFN@OldMakefntext}{}
\makeatother
}

\begin{abstract}
In the wake of \cite{KawazumiSoulieI}, we study the stable cohomology groups of the mapping class groups of surfaces with twisted coefficients given by the $d^{\ith}$ exterior powers of the first rational homology of the unit tangent bundles of the surfaces $\tilde{H}_{\bQ}$. These coefficients are outside of the traditional framework of cohomological stability. They form a module $H_{\st}^{*}(\Lambda^{d}\tilde{H}_{\bQ})$ over the stable cohomology algebra of the mapping class groups with trivial coefficients denoted by $\Sym_{\bQ}(\cE)$. If $d\neq 2$, the $\Tor$-group in each degree of $H_{\st}^{*}(\Lambda^{d}\tilde{H}_{\bQ})$ does not vanish, and we compute all these $\Tor$-groups explicitly for $d \leq 5$. In particular, for each $d\neq 2$, the module $H_{\st}^{*}(\Lambda^{d}\tilde{H}_{\bQ})$ is not free over $\Sym_{\bQ}(\cE)$, while it is free for $d=2$.
For comparison, we also compute the stable cohomology group with coefficients in the $d^{\ith}$ exterior powers of the first rational cohomology of the unit tangent bundle of the surface, which fit into the classical framework of cohomological stability.
\end{abstract}

\section*{Introduction}

Considering a smooth compact connected orientable surface of genus $g\geq0$ and with one boundary component $\Sigma_{g,1}$, we denote by $\Gamma_{g,1}$ its mapping class group, i.e.~the isotopy classes of its diffeomorphisms restricting to the identity on the boundary.
The study of the (co)homology with twisted coefficients of these mapping class groups of surfaces has been a very rich and active research topic over the last decades; see \cite{Harer,HarerH_2,HarerH_3}, \cite{Ivanov}, \cite{Looijenga}, \cite{KawazumitwistedMMM,stablecohomologyKawazumi}, \cite[\S5.5]{GalatiusKupersRW}, \cite[\S4]{Hain} and \cite[\S11]{PetersenTavakolYin} for instance.
Let $\mathrm{UT}\Sigma_{g,1}$ denote the unit tangent bundle of the surface $\Sigma_{g,1}$.
In \cite{KawazumiSoulieI}, we compute the stable cohomology groups of $\{\Gamma_{g,1},g\in\bN\}$ with twisted coefficients defined by the first rational homology group $H_{1}(\mathrm{UT}\Sigma_{g,1};\bQ)$. In this paper, we study the stable cohomology groups of the mapping class groups with twisted coefficients given by the exterior powers of this representation and its dual; see Theorems~\ref{thm:main_Thm_exterior_powers_contravariant}--\ref{thm:main_Thm_Tor} below.

\paragraph*{Cohomological stability.}
Computing cohomology of groups may be a difficult task in general. However, cohomological stability phenomena happen for mapping class groups and they give a clue for these computations.
Following the works of Harer \cite{Harer}, Ivanov \cite{Ivanov} and Randal-Williams and Wahl \cite{WahlRandal-Williams}, the current classical framework for cohomological stability with twisted coefficients for mapping class groups is defined as follows. 
We view the mapping class groups as a set of groups $\{\Gamma_{g,1}\}_{g\in\bN}$. For each $g\geq0$, we consider the canonical embedding $\mathfrak{i}_{g}\colon \Sigma_{g,1}\hookrightarrow \Sigma_{g+1,1}$ given by viewing $\Sigma_{g,1}$ as a subsurface of $\Sigma_{g+1,1}$ thanks to the boundary connected sum $\Sigma_{g+1,1}\cong \Sigma_{1,1}\natural \Sigma_{g,1}$, and it induces an injection $\Gamma_{g,1}\hookrightarrow \Gamma_{g+1,1}$. There is a category $\U\cM_2$ whose objects are the surfaces $\{\Sigma_{g,1}\}_{g\in\bN}$, the mapping class groups $\{\Gamma_{g,1}\}_{g\in\bN}$ as automorphisms and the embeddings $\{\mathfrak{i}_{g}\}_{g\in\bN}$; see \S\ref{ss:twisted_coefficient_systems} for the precise definition. For simplicity, we identify the surface $\Sigma_{g,1}$ with its indexing integer $g$, especially when applying a functor on that object.
Denoting by $\Ab$ the category of abelian groups, a functor $F \colon \U\cM_2 \to \Ab$ induces a $\Gamma_{g,1}$-equivariant morphism $F(\mathfrak{i}_{g})\colon F(g) \to F(g+1)$ for each $g\geq0$. Applying the duality functor $-^{\vee} : \Ab \to \Ab^{\opp}$, this data defines maps in cohomology for all $i,g\geq0$:
$$H^{i}(\Gamma_{g+1,1}; F^{\vee}(g+1)) \to H^{i}(\Gamma_{g,1}; F^{\vee}(g)).$$
If this map is an isomorphism for $B(i,F)\leq g$ with some $B(i,F)\in\bN$ depending on $i$ and $F$, then the mapping class groups are said to satisfy \emph{(classical) cohomological stability} with (twisted) coefficient in $F^{\vee}$. Also, the \emph{stable cohomology group} $H^{i}(\Gamma_{\infty,1}; F^{\vee})$ is the inverse limit $\underleftarrow{\mathrm{Lim}}_{g\geq0}H^{*}(\Gamma_{g,1}; F^{\vee}(g))$ induced by the stability maps.
Ivanov \cite[Th.~4.1]{Ivanov} and Randal-Williams and Wahl \cite[Th.~5.26]{WahlRandal-Williams} prove such cohomological stability property with twisted coefficients in $F^{\vee}$ if $F\colon\U\cM_{2}\to\Ab$ satisfies some \emph{polynomiality} conditions; see \S\ref{ss:homological_stability} for further details.

The representation theory of the mapping class group $\Gamma_{g,1}$ is still poorly known; see Margalit's expository paper \cite{Margalit} for instance.
The first rational homology group of the surface $\Sigma_{g,1}$, that we denote by $H_{\bQ}(g)$, gives a well-known self-dual $2g$-dimensional representation of $\Gamma_{g,1}$. Also, the first rational homology group $H_{1}(\mathrm{UT}\Sigma_{g,1};\bQ)$, that we denote by $\tilde{H}_{\bQ}(g)$, naturally defines a $(2g+1)$-dimensional representation of $\Gamma_{g,1}$ which has been studied by Trapp \cite{Trapp}; see \S\ref{ss:MCG_representations} for further details.
Taking the complex numbers as ground ring, thanks to the works of Franks and Handel \cite{FranksHandel}, Korkmaz \cite{KorkmazII,KorkmazIII} and Kasahara \cite{Kasahara}, there are classification results for the $\Gamma_{g,1}$-representations of dimension less than or equal to $2g+1$, also known as the \emph{low-dimensional representations}. Namely, up to conjugation and dual, $H_{1}(\Sigma_{g,1};\bC)$ and $H_{1}(\mathrm{UT}\Sigma_{g,1};\bC)$ are the only indecomposable complex low-dimensional $\Gamma_{g,1}$-representations if $g\geq 7$; see \cite[Th.~1.1, Rem.~1.2.(2)]{Kasahara}. Therefore, $H_{\bQ}(g)$ and $\tilde{H}_{\bQ}(g)$ are fundamental blocks of the representation theory of the mapping class group $\Gamma_{g,1}$. Moreover, the families of representations $\{\Lambda^{d}H_{\bQ}(g)\}_{g\in\bN}$ and $\{\Lambda^{d}\tilde{H}_{\bQ}(g)\}_{g\in\bN}$ for each $d\geq1$ define polynomial functors $\U\cM_{2}\to\Ab$ in the sense of  \S\ref{ss:homological_stability}, denoted by $\Lambda^{d}H_{\bQ}$ and $\Lambda^{d}\tilde{H}_{\bQ}$ respectively; see Example~\ref{eg:ext_powers_H_tildeH}.

In contrast, for each $d\geq1$, the family of mapping class groups representations $\{\Lambda^{d}\tilde{H}^{\vee}_{\bQ}(g)\}_{g\in\bN}$ defines a functor $\Lambda^{d}\tilde{H}^{\vee}_{\bQ}\colon \U\cM_2^{\opp} \to \Ab$. This functor does not fit into the above classical setting for twisted cohomological stability, so we cannot deal with the cohomology with coefficients in $\Lambda^{d}\tilde{H}_{\bQ}(g)$ with that method. However, we can handle this type of coefficients with an exotic approach detailed in \S\ref{sss:exotic_framework}.
For example, the functor $\Lambda^{d}\tilde{H}_{\bQ}\colon \U\cM_2 \to \Ab$ has the property that, for each $g\geq0$, the $\Gamma_{g,1}$-equivariant morphism $\Lambda^{d}\tilde{H}_{\bQ}(\mathfrak{i}_{g})$ has a canonical $\Gamma_{g,1}$-equivariant splitting; see Example~\ref{eg:cohomological_stability_contra}. The injections $\Gamma_{g,1}\hookrightarrow\Gamma_{g+1,1}$ along with that splitting for each $g$ induce maps in cohomology $H^{i}(\Gamma_{g+1,1}; \Lambda^{d}\tilde{H}_{\bQ}(g+1)) \to H^{i}(\Gamma_{g,1}; \Lambda^{d}\tilde{H}_{\bQ}(g))$ for all $i,g\geq0$, which are isomorphisms for $g$ big enough with respect to $i$ and $d$. The \emph{stable cohomology group} $H^{i}(\Gamma_{\infty,1}; \Lambda^{d}\tilde{H}_{\bQ})$ is then the inverse limit $\underleftarrow{\mathrm{Lim}}_{g\geq0}H^{*}(\Gamma_{g,1}; \Lambda^{d}\tilde{H}_{\bQ}(g))$ with respect to these stabilisation maps.

\paragraph*{Mumford-Morita-Miller classes.}
We recall well-known cohomology classes which compute some stable cohomology groups of the mapping class groups.
Madsen and Weiss \cite{MadsenWeiss} prove Mumford's conjecture \cite{Mumford} and thus compute the stable rational cohomology $H^{*}(\Gamma_{\infty,1};\bQ)$. Namely, they use the cohomology classes $\{e_{i}\in H^{2i}(\Gamma_{\infty,1};\bQ);i\geq1\}$ introduced by  Mumford \cite{Mumford}, Morita \cite{Morita_char_class_bull,Morita_char_class} and Miller \cite{Miller}, called the \emph{classical Mumford-Morita-Miller classes}. We denote by $\cE$ the $\bQ$-vector space $\bigoplus_{i=1}^{\infty}\bQ e_{i}$ and by $\Sym_{\bQ}(\cE)$ its symmetric algebra. Madsen and Weiss prove that there is an algebra isomorphism:
\begin{equation}\label{eq:Madsen_weiss_thm}
    H^{*}(\Gamma_{\infty,1};\bQ)\cong\Sym_{\bQ}(\cE).
\end{equation}
We note that, both for the classical and exotic framework for cohomological stability, the graded module of the stable twisted cohomology groups has a canonical $\Sym_{\bQ}(\cE)$-module structure induced by the cup product with elements of $H^{*}(\Gamma_{\infty,1};\bQ)$.
For $\{M_{g}\}_{g\in\bN}$ a family of $\{\Gamma_{g,1}\}_{g\in\bN}$ for which there is cohomological stability, $\Sym_{\bQ}(\cE)$ being concentrated in even degrees, we denote by $H^{\odd}(\Gamma_{\infty,1};M)$
and $H^{\even}(\Gamma_{\infty,1};M_{\infty})$ the $\bN$-graded $\Sym_{\bQ}(\cE)$-submodules of $H^{*}(\Gamma_{\infty,1};M_{\infty})$ defined by $\{H^{\mathrm{2i+1}}(\Gamma_{\infty,1};M_{\infty}),i\in\bN\}$ and $\{H^{\mathrm{2i}}(\Gamma_{\infty,1};M_{\infty}),i\in\bN\}$ respectively.
Furthermore, the first author \cite{KawazumitwistedMMM} introduced a series of twisted cohomology classes on the mapping class group $\Gamma_{g,1}$, called the \emph{twisted Mumford-Morita-Miller classes} $\{m_{i,j};i\geq0,j\geq1\}$; see \S\ref{s:stable_cohomology_framework}. He also proves in \cite{stablecohomologyKawazumi} that some algebraic combinations of the twisted Mumford-Morita-Miller classes provide a free basis of the stable cohomology group of the mapping class groups with twisted coefficient in $\Lambda^{d} H_{\bQ}$ for all $d\geq1$; see Theorem~\ref{stablecohomologyKawazumi}.

\paragraph*{Stable cohomology for $\Lambda^{d}\tilde{H}^{\vee}_{\bQ}$.}
The computations of the stable cohomology with twisted coefficients in the representations $\Lambda^{d}\tilde{H}_{\bQ}^{\vee}(g)$ is the least difficult. We prove that:
\begin{athm}[Theorem~\ref{thm:stable_cohomology_contravariant_tildeH_ext}]\label{thm:main_Thm_exterior_powers_contravariant}
The stable twisted cohomology groups $\bigoplus_{d\in\bN} H^{*}(\Gamma_{\infty,1};\Lambda^{d}\tilde{H}^{\vee}_{\bQ})$ form a commutative bigraded $\Sym_{\bQ}(\cE)$-algebra $H^{*}(\Gamma_{\infty,1};\Lambda^{*}\tilde{H}^{\vee}_{\bQ})$ isomorphic to the polynomial $\Sym_{\bQ}(\cE)$-algebra
$$\Sym_{\bQ}(\cE)[\{m_{i,j};i\geq0,j\geq1,(i,j)\neq(1,1)\}].$$
\end{athm}
In particular, for $d=1$, Theorem~\ref{thm:main_Thm_exterior_powers_contravariant} recovers \cite[Th.~A]{KawazumiSoulieI}. This computation is based on the study of the cohomology sequence of the short exact sequence \eqref{exteriorSES_cov} viewing $\Lambda^{d}\tilde{H}^{\vee}_{\bQ}(g)$ as a non-trivial extension of $\Gamma_{g,1}$-representations; see Lemma~\ref{lem:connecting_hom_tilde_H_vee}. We stress that, since the functors $\Lambda^{*}\tilde{H}^{\vee}_{\bQ}$ are \emph{contravariant}, we are in the \emph{classical} framework for cohomological stability in that case; see \S\ref{s:stable_cohomology_framework}.

\paragraph*{Stable cohomology for $\Lambda^{d}\tilde{H}_{\bQ}$: general results.}
The functor $\tilde{H}_{\bQ}$ being part of the \emph{exotic} framework for cohomological stability, the stable cohomology groups with twisted coefficients in $\Lambda^{d}\tilde{H}_{\bQ}$ are much more difficult to compute. For $d=1$, we recall that we make the full calculation of $H^{*}(\Gamma_{\infty,1};\tilde{H}_{\bQ})$ in \cite[Th.~B]{KawazumiSoulieI}. Here, we start by proving general qualitative results for all $d\geq1$.

First, we consider the exterior powers all together. Namely, we study the bigraded $\Sym_{\bQ}(\cE)$-algebra of the stable twisted cohomology groups $H^{*}(\Gamma_{\infty,1};\Lambda^{*}\tilde{H}_{\bQ}) := \bigoplus_{d\in\bN} H^{*}(\Gamma_{\infty,1};\Lambda^{d}\tilde{H}_{\bQ})$. The main idea consists in studying this algebra over the \emph{localisation} $\Sym_{\bQ}(\cE^{\pm})$ of $\Sym_{\bQ}(\cE)$; see \S\ref{ss:first_localisation}.
\begin{athm}[Theorem~\ref{thm:full_stable_cohomology_extension}, Corollaries~\ref{coro:cok_torsion}, \ref{coro:Sym_SymE}, \ref{coro:full_stable_cohomology}]\label{thm:main_Thm_exterior_powers_general_covariant_graded_algebra}
The $\Sym_{\bQ}(\cE^{\pm})$-algebra $\Sym_{\bQ}(\cE^{\pm})\otimes_{\Sym_{\bQ}(\cE)}H^{*}(\Gamma_{\infty,1};\Lambda^{*}\tilde{H}_{\bQ})$ is free whose set of generators is detailed in \eqref{eq:generators_localised_stable_cohomology_exterior_tildeH}.

The bigraded $\Sym_{\bQ}(\cE)$-algebra $H^{\even}(\Gamma_{\infty,1};\Lambda^{\even}\tilde{H}_{\bQ})\oplus H^{\odd}(\Gamma_{\infty,1};\Lambda^{\odd}\tilde{H}_{\bQ})$ is computed as the pullback square \eqref{eq:pullback}, while the $\Sym_{\bQ}(\cE)$-module $H^{\even}(\Gamma_{\infty,1};\Lambda^{\odd}\tilde{H}_{\bQ})\oplus H^{\even}(\Gamma_{\infty,1};\Lambda^{\odd}\tilde{H}_{\bQ})$ is torsion.

Finally, the $\Sym_{\bQ}(\cE)$-algebra $H^{*}(\Gamma_{\infty,1};\Lambda^{*}\tilde{H}_{\bQ})$ is not generated from the free commutative algebra of the $\Sym_{\bQ}(\cE)$-module $H^{*}(\Gamma_{\infty,1};\tilde{H}_{\bQ})$.
\end{athm}

Now, fixing the exterior power $d\geq1$, we also have general results on the $\Sym_{\bQ}(\cE)$-module $H^{*}(\Gamma_{\infty,1};\Lambda^{d}\tilde{H}_{\bQ})$ in \S\ref{ss:key_tools}.
We consider the $\Tor$-groups $\Tor_{*}^{\Sym_{\bQ}(\cE)}(\bQ,H^{*}(\Gamma_{\infty,1};\Lambda^{d}\tilde{H}_{\bQ}))$, that we write $H_{*}(\Sym_{\bQ}(\cE);H_{\st}^{*}(\Lambda^{d}\tilde{H}_{\bQ}))$ for the sake of simplicity. They provide key information on the stable twisted cohomology groups  as a $\Sym_{\bQ}(\cE)$-module. For instance, $H_{0}(\Sym_{\bQ}(\cE);H_{\st}^{*}(\Lambda^{d}\tilde{H}_{\bQ}))$ gives a lower bound of the number of generators of $H^{*}(\Gamma_{\infty,1};\Lambda^{d}\tilde{H}_{\bQ})$, and $H_{j}(\Sym_{\bQ}(\cE);H_{\st}^{*}(\Lambda^{d}\tilde{H}_{\bQ}))$ for each $j\geq1$ characterises the complexity of the relations of $H^{*}(\Gamma_{\infty,1};\Lambda^{d}\tilde{H}_{\bQ})$. In order to label the $\Sym_{\bQ}(\cE)$-submodules induced by the parity of the cohomological degrees, we use the notations $\dagger=``\odd"$ and $\ddagger=``\even"$ if $d$ is odd while $\dagger=``\even"$ and $\ddagger=``\odd"$ if $d$ is even.
\begin{athm}[Corollaries~\ref{coro:stable_cohomo_algebra_not_free} and \ref{coro:low_dim_paired}, Theorem~\ref{thm:Tor_j_non-trivial}]\label{thm:main_Thm_exterior_powers_general_covariant_module_d_fixed}
For $d=1$ or $d\geq3$, the $\Sym_{\bQ}(\cE)$-module $H^{*}(\Gamma_{\infty,1};\Lambda^{d}\tilde{H}_{\bQ}(g))$ is not free. More precisely, the $\Sym_{\bQ}(\cE)$-module $H^{\ddagger}(\Gamma_{\infty,1};\Lambda^{d}\tilde{H}_{\bQ}(g))$ is torsion, and the $\Sym_{\bQ}(\cE)$-module $H^{\dagger}(\Gamma_{\infty,1};\Lambda^{d}\tilde{H}_{\bQ}(g))$ is not free.
Moreover, the $\Tor$-groups of these $\Sym_{\bQ}(\cE)$-modules are always non-null.

Furthermore, for $d\geq3$, the stable twisted cohomology group $H^{i}(\Gamma_{\infty,1};\Lambda^{d}\tilde{H}_{\bQ})$ is null if $i<d$ and $i\equiv d\Modulo{2}$.
\end{athm}
Theorem~\ref{thm:main_Thm_exterior_powers_general_covariant_module_d_fixed} highlights the contrast with the stable twisted cohomology computations of Theorem~\ref{thm:main_Thm_exterior_powers_contravariant}
.
All these results and computations are based on the cohomology sequence of the short exact sequence \eqref{prop:ses_exterior_power} expressing $\Lambda^{d}\tilde{H}_{\bQ}(g)$ as a non-trivial extension of $\Gamma_{g,1}$-representations. Contrary to the case of $\Lambda^{d}\tilde{H}^{\vee}_{\bQ}(g)$, for which the connecting homomorphism is simply a \emph{multiplication} (see Lemma~\ref{lem:connecting_hom_tilde_H_vee}), Proposition~\ref{lem:connecting_hom_exterior_contra} proves that the connecting homomorphism of \eqref{prop:ses_exterior_power} is determined by a \emph{derivation}, whose kernel is much more complicated to study.

\paragraph*{Stable cohomology for $\Lambda^{d}\tilde{H}_{\bQ}$: computations for $d\leq5$.}
For each $d$ small, we compute explicitly the cokernel of the derivation giving the connecting homomorphism of \eqref{prop:ses_exterior_power} and the $\Tor$-groups of the $\Sym_{\bQ}(\cE)$-module $H^{*}(\Gamma_{\infty,1};\Lambda^{d}\tilde{H}_{\bQ})$.
From now on, we will also use the modified version of the twisted Mumford-Morita-Miller classes $\mm_{i,j}:=((-1)^j/j!)\cdot m_{i,j}$, which are more convenient to make our computations.

First, the case of $d=2$ stands out from the crowd, since we can fully compute $H^{*}(\Gamma_{\infty,1};\Lambda^{2}\tilde{H}_{\bQ})$. In particular, it turns out to be a free $\Sym_{\bQ}(\cE)$-module. Indeed, we prove:
\begin{athm}[Theorem~\ref{thm:result_second_exterior_power_contra}]\label{thm:main_Thm_2nd_exterior_powers_general_covariant}
The $\Sym_{\bQ}(\cE)$-module $H^{*}(\Gamma_{\infty,1};\Lambda^{2}\tilde{H}_{\bQ})$ is isomorphic to the free $\Sym_{\bQ}(\cE)$-module with basis
$$\left\{\mm_{j,1} \mm_{l,1}-e_{l}\mm_{j,2}-e_{j}\mm_{l,2};\, j\geq l\geq1 \right\}.$$
\end{athm}

In contrast, the computation of the $\Sym_{\bQ}(\cE)$-module $H^{*}(\Gamma_{\infty,1};\Lambda^{d}\tilde{H}_{\bQ}(g))$ for each $d\geq3$ is very complicated and seems out of reach with our current techniques. However, for $3\leq d\leq 5$, we fully compute the $\Sym_{\bQ}(\cE)$-submodule $H^{\ddagger}(\Gamma_{\infty,1};\Lambda^{d}\tilde{H}_{\bQ})$, where $\ddagger=``\even"$ if $d$ is odd and $\ddagger=``\odd"$ if $d$ is even.

\begin{athm}[Theorems~\ref{prop:3rd-ext-even}, \ref{prop:4th-ext_odd}, \ref{thm:5th-ext-even}]\label{thm:main_Thm_ddagger_part}
There are $\Sym_{\bQ}(\cE)$-module isomorphisms
$$H^{\even}(\Gamma_{\infty,1};\Lambda^{3}\tilde{H}_{\bQ})\cong (\Sym_{\bQ}(\cE)/(e_{1}^{2},e_{\alpha},\alpha\geq2))\left\{ \mm_{0,2}\right\},$$
$$H^{\odd}(\Gamma_{\infty,1};\Lambda^{4}\tilde{H}_{\bQ})\cong (\Sym_{\bQ}(\cE)/(e_{1}^{2},e_{\alpha},\alpha\geq2)) \left\{ \mm_{0,3}\right\},$$
and the $\Sym_{\bQ}(\cE)$-module $H^{\even}(\Gamma_{\infty,1};\Lambda^{5}\tilde{H}_{\bQ})$ is isomorphic to the quotient of the torsion $\Sym_{\bQ}(\cE)$-module
$$\begin{array}{c}
(\Sym_{\bQ}(\cE)/(e_{1}^{3}, e_{\alpha}e_{\beta},e_{\gamma};\alpha,\beta\geq1\,\, \text{except $\alpha = \beta = 1$}, \gamma\geq4))\{\mm_{0,4};\mm_{0,2}\mm_{0,2}\}\\
\oplus(\Sym_{\bQ}(\cE)/(
e_{\alpha}, e_{\beta} e_{\gamma} ; \alpha\geq3, \beta, \gamma \geq 1))\{\mm_{0,3}\mm_{1,1}\}
\\\oplus(\Sym_{\bQ}(\cE)/(e_{\alpha}, e_{\beta} e_{\gamma} ;\alpha\geq2, \beta, \gamma \geq 1))\{\mm_{0,3}\mm_{2,1}\}.
\end{array}$$
by the direct sum of the trivial $\Sym_{\bQ}(\cE)$-modules $\{2e_{2}\mm_{0,3}\mm_{1,1}+3e_{1}\mm_{0,3}\mm_{2,1}\}$, $\{e_{2}\mm_{0,2}\mm_{0,2}+6e_{1}\mm_{0,3}\mm_{1,1}\}$ and $\{e_{3}\mm_{0,2}\mm_{0,2}-6e_{1}^{2}\mm_{0,4}\}$.
\end{athm}

On the other hand, the difficult part to compute is the $\Sym_{\bQ}(\cE)$-submodule $H^{\dagger}(\Gamma_{\infty,1};\Lambda^{d}\tilde{H}_{\bQ})$, where $\ddagger=``\odd"$ if $d$ is odd and $\dagger=``\even"$ if $d$ is even. Instead, we fully compute the torsion groups $H_{*}(\Sym_{\bQ}(\cE);H_{\st}^{*}(\Lambda^{d}\tilde{H}_{\bQ}))$, because they reflect the complexity of the stable twisted cohomology groups.
For each $\ell\geq1$, the $\bQ$-vector spaces $\bigoplus_{1\leq i \leq \ell -1}\bQ e_{i}$ and $\bigoplus_{i \geq \ell}\bQ e_{i}$ are denoted by $\cE_{\leq \ell-1}$ and $\cE_{\geq \ell}$ respectively.
\begin{athm}[Propositions~\ref{prop:3rd-ext_Tor} and \ref{prop:4th-ext_Tor}, Theorem~\ref{thm:5th-ext_Tor}]\label{thm:main_Thm_Tor}
The full computations of the $\Tor$-groups $H_{j}(\Sym_{\bQ}(\cE);H_{\st}^{*}(\Lambda^{d}\tilde{H}_{\bQ}))$ are done as follows.
\begin{itemizeb}
    \item For $j\geq1$ and $d\in\{3,4\}$:
    $$H_{j}(\Sym_{\bQ}(\cE);H_{\st}^{*}(\Lambda^{d}\tilde{H}_{\bQ}))\cong\Lambda^{j-1}\cE_{\geq 2}\oplus\Lambda^{j}\cE_{\geq 2}\oplus\Lambda^{j+1}\cE_{\geq 2}\oplus\Lambda^{j+2}\cE_{\geq 2}.$$
    \item For $d\in\{3,4\}$, $H_{0}(\Sym_{\bQ}(\cE);H_{\st}^{*}(\Lambda^{d}\tilde{H}_{\bQ}))\cong\bQ\oplus \cE_{\geq 2}\oplus\Lambda^{2}\cE_{\geq 2} \oplus \mathscr{S}_{d}$, where
    $$\mathscr{S}_{3}=\bQ\{[\mm_{\alpha-1,2}\mm_{\beta,1} - \mm_{\beta-1,2}\mm_{\alpha,1}];2 \leq \alpha < \beta \} \oplus \bQ\{[\mm_{\alpha,1}\mm_{\beta,1}\mm_{\gamma,1}];1 \leq \alpha \leq \beta \leq \gamma \geq 2 \},$$
    \begin{align*}
\mathscr{S}_{4}  \;=\;& \bQ\{[\mm_{a,2}\mm_{b,2}-\mm_{a-1,3}\mm_{b+1,1}-\mm_{b-1,3}\mm_{a+1,1}];2 \leq a \leq b\} \\  
& \oplus\bQ\{[\mm_{\gamma-1,2}\mm_{\alpha,1}\mm_{\beta,1} - \mm_{\alpha-1,2}\mm_{\beta,1}\mm_{\gamma,1}]
; 1\leq \alpha < \beta < \gamma\}
\\ 
& \oplus\bQ\{
[\mm_{\gamma-1,2}\mm_{\alpha,1}\mm_{\beta,1} - \mm_{\beta-1,2}\mm_{\alpha,1}\mm_{\gamma,1}]
; 1\leq \alpha < \beta < \gamma\}
\\ 
& \oplus\bQ\{[\mm_{\gamma-1,2}\mm_{\alpha,1}\mm_{\alpha,1} - \mm_{\alpha-1,2}\mm_{\alpha,1}\mm_{\gamma,1}]
; 1\leq \alpha < \gamma\geq 3\}\\
& \oplus\bQ\{ 
[\mm_{\gamma-1,2}\mm_{\alpha,1}\mm_{\gamma,1} - \mm_{\alpha-1,2}\mm_{\gamma,1}\mm_{\gamma,1}]
; 1\leq \alpha < \gamma\}\\
&\oplus\bQ\{[\mm_{\alpha,1}\mm_{\beta,1}\mm_{\gamma,1}\mm_{\delta,1}];1\leq \alpha\leq \beta\leq \gamma\leq \delta \}.
\end{align*}
    \item For $d=5$ and $j\geq1$:
\begin{eqnarray*}
   H_{j}(\Sym_{\bQ}(\cE);H_{\st}^{*}(\Lambda^{5}\tilde{H}_{\bQ}))&\cong& (\Lambda^{j-3}\cE_{\geq 4})^{\oplus 6}\oplus(\Lambda^{j-2}\cE_{\geq 4})^{\oplus 17}\oplus (\Lambda^{j-1}\cE_{\geq 4})^{\oplus 21}\\
   &&\oplus(\Lambda^{j}\cE_{\geq 4})^{\oplus 21}\oplus (\Lambda^{j+1}\cE_{\geq 4})^{\oplus 15}\oplus(\Lambda^{j+2}\cE_{\geq 4})^{\oplus 4}.
\end{eqnarray*}
Moreover, $H_{0}(\Sym_{\bQ}(\cE);H_{\st}^{*}(\Lambda^{5}\tilde{H}_{\bQ}))\cong\bQ^{\oplus 21}\oplus \cE_{\geq 4}^{\oplus 15}\oplus(\Lambda^{2}\cE_{\geq 4})^{\oplus 4} \oplus\mathscr{S}_{5}$ where $\mathscr{S}_{5}$ is a complicated $\bQ$-module introduced in Theorem~\ref{thm:5th-ext_Tor}.
\end{itemizeb}
\end{athm}

The result of Theorem~\ref{thm:main_Thm_Tor} highlights that the complexity of the description of the $H^{*}(\Gamma_{\infty,1};\Lambda^{d}\tilde{H}_{\bQ}))$ as a $\Sym_{\bQ}(\cE)$-module grows with $d$. The length and technicality of our proofs also increase a lot with $d$; see \S\ref{ss:third_exterior}, \S\ref{ss:fourth_exterior} and \S\ref{ss:fifth_exterior}. This reflects the difficulties we face in our work as $d$ grows and is the reason we stop our computations at $d=5$.

\paragraph*{Surfaces with more boundaries.}
All our results straightforwardly generalise to stable twisted cohomology for mapping class groups of surfaces with more boundary components. Namely, we denote by $\Sigma_{g,n}$ the smooth compact connected orientable surface of genus $g\geq0$ and with $n\geq1$ boundary components, and by $\Gamma_{g,n}$ its mapping class group. We recall that gluing a disc on each boundary component (i.e.~\emph{capping} these boundaries) except one induces a surjection $\kappa_{n}\colon \Gamma_{g,n}\twoheadrightarrow \Gamma_{g,1}$; see \cite[\S3.6.2, \S4.2.1]{farbmargalit}.
Then, the $\Gamma_{g,1}$-representations $H_{\bQ}(g)$ and $\tilde{H}_{\bQ}(g)$ are also $\Gamma_{g,n}$-representations via $\kappa_{n}$. We also consider the $\Gamma_{g,n}$-representations $H_{1}(\Sigma_{g,n};\bQ)$ and $H_{1}(UT\Sigma_{g,n};\bQ)$, that we denote by $H_{\bQ}(g,n)$ and $\tilde{H}_{\bQ}(g,n)$. We note that $H_{\bQ}(g,n)$ is self-dual as a $\Gamma_{g,n}$-representation, similarly to $H_{\bQ}(g)=H_{\bQ}(g,1)$.

We fix $n\geq1$.
For the representations $\{H_{\bQ}(g,n)\}_{g\in\bN}$, it follows from \cite[Th.~1.A]{stablecohomologyKawazumi} that, for any $d\geq1$, $H^{i}(\Gamma_{g,n}; \Lambda^{d}H(g,n))\cong H^{i}(\Gamma_{g,1}; \Lambda^{d}H(g))$ for $g\geq 3i+d+2$.
Also, the representations $\{H_{\bQ}(g)\}_{g\in\bN}$ provide \emph{finite degree coefficient systems} in the sense of Boldsen \cite[\S4.1]{Boldsen}. So, for any $d\geq1$, it follows from \cite[Th.~4.15]{Boldsen} that $H^{i}(\Gamma_{g,n}; \Lambda^{d}H(g))\cong H^{i}(\Gamma_{g,1}; \Lambda^{d}H(g))$ for $g\geq 3i+d+2$.
Hence, we may view $\Lambda^{d} \tilde{H}_{\bQ}(g)$ and $\Lambda^{d} \tilde{H}_{\bQ}(g,n)$ as non-trivial extensions of $\Gamma_{g,n}$-representations whose associated cohomology classes are the same as \eqref{prop:ses_exterior_power}.
This also holds for the duals $\Lambda^{d} \tilde{H}^{\vee}_{\bQ}(g)$ and $\Lambda^{d} \tilde{H}^{\vee}_{\bQ}(g,n)$ whose cohomology classes as extensions are equal to \eqref{exteriorSES_cov}.
Therefore, there are twisted cohomological stability for these coefficient systems, the stable twisted cohomology groups $H^{*}(\Gamma_{\infty,n}; \Lambda^{d}\tilde{H}_{\bQ}(n))$ and $ H^{*}(\Gamma_{\infty,n}; \Lambda ^{d}\tilde{H}_{\bQ})$ are both isomorphic to $H^{*}(\Gamma_{\infty,1}; \Lambda^{d}\tilde{H}_{\bQ})$, while $H^{*}(\Gamma_{\infty,n}; \Lambda^{d}\tilde{H}^{\vee}_{\bQ}(n))$ and $ H^{*}(\Gamma_{\infty,n}; \Lambda ^{d}\tilde{H}^{\vee}_{\bQ})$ are both isomorphic to $H^{*}(\Gamma_{\infty,1}; \Lambda^{d}\tilde{H}^{\vee}_{\bQ})$. A fortiori, Theorems~\ref{thm:main_Thm_exterior_powers_contravariant}--\ref{thm:main_Thm_Tor} repeat verbatim for the stable twisted cohomology groups of $\Gamma_{g,n}$ with coefficients in the exterior powers of $\tilde{H}_{\bQ}(n)$, $\Lambda ^{d}\tilde{H}_{\bQ}$, $\tilde{H}^{\vee}_{\bQ}(n)$ and $\tilde{H}^{\vee}_{\bQ}$.

\paragraph*{Perspectives.}
For a \emph{fixed} cohomological degree $i$ and power $d\geq3$, our work gives all the necessary tools and methods to compute explicitly the stable twisted cohomology group $H^{i}(\Gamma_{\infty,1};\Lambda^{d}\tilde{H}_{\bQ})$ as a $\bQ$-module.
Furthermore, a lot of key steps of the results of Theorems~\ref{thm:main_Thm_exterior_powers_contravariant}--\ref{thm:main_Thm_Tor} could be proved with integral coefficients.
So we might in principle be able to do the computations with $\bZ$ as ground ring. However, the stable cohomology $H^{*}(\Gamma_{\infty,1};\bZ)$ is still poorly known. Nevertheless, one could potentially make stable twisted cohomology computations with the finite field $\bF_{p}$ as ground ring, by using the computations of $H^{*}(\Gamma_{\infty,1};\bF_{p})$ of Galatius \cite{Galatius}.

\paragraph*{Outline.} In \S\ref{s:general_recollections}, we make recollections on the representation theory of mapping class groups and we recall the framework and key properties for studying the stable twisted cohomology of mapping class groups. We first   compute the stable twisted cohomology of Theorem~\ref{thm:main_Thm_exterior_powers_contravariant}
in \S\ref{s:contravariant_coeff_syst_exterior}. Then, we carry on the general theory for the study of the stable twisted cohomology with coefficients in $\Lambda^{*}\tilde{H}_{\bQ}$ in \S\ref{s:covariant_coeff_syst_exterior_general_theory} to prove Theorems~\ref{thm:main_Thm_exterior_powers_general_covariant_graded_algebra} and \ref{thm:main_Thm_exterior_powers_general_covariant_module_d_fixed}. We finally make computations of these stable twisted cohomology and associated $\Tor$-groups for $d\leq 5$ in \S\ref{s:covariant_coeff_syst_exterior_computations}, in particular proving Theorems~\ref{thm:main_Thm_2nd_exterior_powers_general_covariant}--\ref{thm:main_Thm_Tor}.

\paragraph*{Conventions and notations.}
We standardly denote by $\bN$ the set of non-negative integers and by $\bj$ the set $\{1,\ldots,j\}$.
We denote by $\mathfrak{S}_{n}$ the symmetric group on a set of $n$ elements.
For a ring $R$, we denote by $R\Mod$ the category of left $R$-modules. Non-specified tensor products are taken over the clear ground ring.
For $R=\bZ$, the category of $\bZ$-modules is also denoted by $\Ab$.
We denote by $\Lambda^{d}:R\Mod \to R\Mod$ the $d^{\ith}$ exterior product functor. For $\bK$ a commutative ring and $V$ a $\bK$-vector space, we denote by $\Sym_{\bK}(V)$ the symmetric algebra on $V$ over $\bK$.
For a map $f$, we generically (when everything is clear from the context)
denote by $f_{*}$ the induced map in homology and by $f^{*}$ the induced map in cohomology.

The duality functor, denoted by $-^{\vee} : R\Mod \to R\Mod^{\opp}$, is defined by $\mathrm{Hom}_{R\Mod}(-,R)$. In particular, for $G$ a group and $V$ a $R[G]$-module, we denote by $V^{\vee}$ the dual $R[G]$-module $\mathrm{Hom}_{R}(V,R)$.
Then, for a functor $F\colon \U\cM_{2} \to \Ab$, the post-composition by the duality functor defines a functor $F^{\vee}$ that we view as a functor of the form $\U\cM_{2} \to \Ab^{\opp}$ (rather than $\U\cM_{2} \to \Ab^{\opp}$, these two points of view being equivalent).

Considering a functor $M\colon \U\cM_{2} \to \Ab$, if there is no risk of confusion, we generally denote the stable twisted cohomology groups $H^{*}(\Gamma_{\infty,1};M)$ and $H^{*}(\Gamma_{\infty,1};M^{\vee})$ by $H_{\st}^{*}(M)$ and $H_{\st}^{*}(M^{\vee})$ for the sake of simplicity.
We denote the cup product by $\cup$, but generally abuse the notation denoting it by an empty space for simplicity when there is no risk of confusion.

\paragraph*{Acknowledgements.}
The authors would like to thank Geoffroy Powell, Oscar Randal-Williams, Antoine Touz{\'e}, Christine Vespa and Shun Wakatsuki for illuminating discussions and questions.
The authors were supported by the PRC CNRS-JSPS French-Japanese Project ``Cohomological study of MCG and related topics''. The first author was supported in part by the grants
JSPS KAKENHI 15H03617, 18K03283, 18KK0071, 19H01784, 20H00115 and 22H01120.
The second author was supported by a Rankin-Sneddon Research Fellowship of the University of Glasgow, by the Institute for Basic Science IBS-R003-D1 and by the ANR Project AlMaRe ANR-19-CE40-0001-01.
\tableofcontents

\section{Background and recollections}\label{s:general_recollections}

In this section, we first review some facts about the representations of the mapping class groups in \S\ref{ss:MCG_representations}. Then we recall the functorial framework to encode these representations in \S\ref{ss:twisted_coefficient_systems} and its applications for stable twisted cohomology computations in \S\ref{ss:homological_stability}. We finally recollect methods to compute the homology of modules over the commutative algebra $\Sym_{\bQ}(\cE)$ in \S\ref{ss:homology_of_algebras}.

\subsection{Representations of the mapping class groups}\label{ss:MCG_representations}

The first integral homology group $H_{1}(\Sigma_{g,1}; \bZ)$ will generally be denoted by $H(g)$ all along the paper (and by $H_{\bQ}(g)$ for the rational version).
We recall that it is naturally equipped with a $\Gamma_{g,1}$-action which factors through the symplectic group $\mathrm{Sp}_{2g}(\bZ)$, and so is called the \emph{symplectic} representation of $\Gamma_{g,1}$.
We note that there is a $\Gamma_{g,1}$-module isomorphism between $H^{1}(\Sigma_{g,1}; \bZ)$ and $H(g)$ because the Poincar{\'e}-Lefschetz duality is the cap product with the fundamental class.
Also, since $H^{1}(\Sigma_{g,1}; \bZ)$ is isomorphic to $H^{\vee}(g)$ as $\Gamma_{g,1}$-modules by the universal coefficient theorem for cohomology of spaces (see \cite[Ex.~3.6.7]{Weibel} for instance), we have a $\Gamma_{g,1}$-module isomorphism
\begin{equation}\label{eq:dual_H}
H^{\vee}(g) \cong H(g).
\end{equation}

\paragraph*{Tangent bundle and framings.}
We denote by $T\Sigma_{g,1}$ the tangent bundle of the surface $\Sigma_{g,1}$ and fix a Riemannnian metric $\Vert \cdot \Vert$ on it.
By definition, the unit tangent bundle $\mathrm{UT}\Sigma_{g,1}$ is the set of 
elements of $T\Sigma_{g,1}$ whose length is $1$ with respect to $\Vert\cdot\Vert$. 
For any diffeomorphism $\varphi$ of $\Sigma_{g,1}$,
its differential $d\varphi$ acts on the unit tangent bundle $\mathrm{UT}\Sigma_{g,1}$.
The canonical projection of the unit tangent bundle $\mathrm{UT}\Sigma_{g,1}$ onto the surface defines the locally trivial fibration $\bS^{1}\overset{\iota}{\hookrightarrow}\mathrm{UT}\Sigma_{g,1}\overset{\varpi}{\to}\Sigma_{g,1}$.

We recall that a \emph{framing} of $\mathrm{UT}\Sigma_{g,1}$ is a continuous map $\xi: \mathrm{UT}\Sigma_{g,1} \to \bS^{1}$
whose restriction to each fiber is an orientation-preserving homeomorphism.
We denote by $\mathscr{F}(\Sigma_{g,1})$ the set of homotopy classes (without fixing the boundary) of framings of $\mathrm{UT}\Sigma_{g,1}$. It is an affine set modeled by the cohomology group 
$H^1(\Sigma_{g,1}; \bZ)$.
The mapping class group $\Gamma_{g,1}$ acts on the set $\mathscr{F}(\Sigma_{g,1})$ by
$$\varphi\cdot \xi = \xi\circ d\varphi^{-1}: \mathrm{UT}\Sigma_{g,1} 
\overset{d\varphi^{-1}}\too \mathrm{UT}\Sigma_{g,1} \overset{\xi}\too \bS^{1}$$
for $\varphi \in \Gamma_{g,1}$ and $\xi \in \mathscr{F}(\Sigma_{g,1})$. For $\alpha: \bS^{1}\to \Sigma_{g,1}$ an immersed loop, its \emph{rotation number} $\rot_\xi(\alpha) \in \bZ$ with respect to the framing $\xi$ is the mapping degree of the composite $\xi\circ({\overset\cdot\alpha}/{\Vert \overset\cdot\alpha\Vert})\colon\bS^{1}\to \mathrm{UT}\Sigma_{g,1}\to \bS^{1}$. The difference of  two framings $\xi$ and $\xi'$ is given by a cohomology class $u \in H^1(\Sigma_{g,1}; \bZ)$ if and only if $\rot_{\xi'}(\alpha) -  \rot_\xi(\alpha) = u([\alpha])$, where $[\alpha] \in H_{1}(\Sigma_{g,1}; \bZ)$ is the homology class of  the immersed loop $\alpha$. For an immersed loop $\alpha$ on $\Sigma_{g,1}$, we have $\rot_{\varphi\cdot\xi}(\alpha) = \rot_\xi(\varphi^{-1}\circ\alpha)$.

\paragraph*{The unit tangent bundle homology representations.}
We now consider the first integral homology group of the unit tangent bundle $\mathrm{UT}\Sigma_{g,1}$, denoted by $\tilde{H}(g)$. It is naturally equipped with a $\Gamma_{g,1}$-action.
Since $\tilde{H}(g)\cong \bZ^{2g+1}$ as an abelian group, the dual $\tilde{H}^{\vee}(g)$ is isomorphic to the first integral cohomology group $H^{1}(\mathrm{UT}\Sigma_{g,1}; \bZ)$ by the universal 
coefficient theorem for cohomology of spaces (see \cite[Example 3.6.7]{Weibel}). Then the Gysin sequence of the locally trivial fibration $\bS^{1}\overset{\iota}{\hookrightarrow}\mathrm{UT}\Sigma_{g,1}\overset{\varpi}{\to}\Sigma_{g,1}$ provides the following $\Gamma_{g,1}$-equivariant short exact sequences:
\begin{equation}\label{SEShutb}
\xymatrix{0\ar@{->}[r] & \bZ\ar@{->}[r]^-{\iota_{*}} & \tilde{H}(g)\ar@{->}[r]^-{\varpi_{*}} & H(g)\ar@{->}[r] & 0,}
\end{equation}
\begin{equation}\label{SESchutb}
\xymatrix{0\ar@{->}[r] & H(g)\ar@{->}[r]^-{\varpi^{*}} & \tilde{H}^{\vee}(g)\ar@{->}[r]^-{\iota^{*}} & \bZ\ar@{->}[r] & 0.}
\end{equation}
We also have the analogue short exact sequences to \eqref{SEShutb} and \eqref{SESchutb} with the rational versions $H_{\bQ}(g)$, $\tilde{H}_{\bQ}(g)$ and $\tilde{H}^{\vee}_{\bQ}(g)$ of the homology groups $H(g)$, $\tilde{H}(g)$ and $\tilde{H}^{\vee}(g)$.
We denote by $\hbar \in \tilde{H}_{\bQ}(g)$ the homology class of the fiber of the unit tangent bundle.

Trapp \cite{Trapp} describes explicitly the $\Gamma_{g,1}$-module structures of $\tilde{H}(g)$ and $\tilde{H}^{\vee}(g)$ as follows.
Fixing a framing $\xi$ of $\mathrm{UT}\Sigma_{g,1}$ and $g\geq2$, the map $k_{\xi}(g,-): \Gamma_{g,1} \to H^1(\Sigma_{g,1}; \bZ)$ defined by
\begin{equation}\label{eq:formula_k}
    k_{\xi}(g,\varphi) = \varphi\cdot \xi - \xi \in H^1(\Sigma_{g,1}; \bZ)
\end{equation}
is a $1$-cocycle of $\Gamma_{g,1}$. Following Earle \cite{Earle} and Furuta \cite[\S4]{Morita_casson_invariant}, the \emph{Earle cohomology class} $k(g):= [k_{\xi}(g)]$ does not depend on the choice of $\xi$ and generates 
the infinite cyclic group $H^1(\Gamma_{g,1}; H(g))$; see also \cite[Th.~1.2]{KawazumiSoulieI}. It is called the \emph{the Earle class} following \cite{Earle}.
Since the cohomology groups $H^1(\Gamma_{0,1}, H_{1}(\Sigma_{0,1}; 
\bZ))$ and $H^1(\Gamma_{1,1}; H(1))$ are trivial (see \cite[Lem.~1.1]{KawazumiSoulieI} for the latter), we define $k(0): \Gamma_{0,1} \to H^1(\Sigma_{0,1}; \bZ)$ and $k(1): \Gamma_{1,1} \to H^1(\Sigma_{1,1}; \bZ)$ to be the trivial cohomology classes.
Trapp \cite[Th.~2.2]{Trapp} shows that, for an element $\varphi\in\Gamma_{g,1}$, the action of $\varphi$ on $\tilde{H}(g)$ is given by the matrix
\begin{equation}\label{eq:matrix_action}
\left[\begin{array}{cc}
\id_{\bZ} & k_{\xi}(g,\varphi) \\
(0) & H(\varphi)
\end{array}\right]    
\end{equation}
where $H(\varphi)$ denotes the action of $\varphi$ on $H(g)$.
In particular, the class $k(g)$ in $\mathrm{Ext}_{\bZ[\Gamma_{g,1}]}^{1}(H(g),\bZ)$ is by definition the extension class of the short exact sequence \eqref{SEShutb}; this may alternatively be shown from the defining formula \eqref{eq:formula_k}.

Also, the formal dual $k^{\vee}(g)$ in $\mathrm{Ext}_{\bZ[\Gamma_{g,1}]}^{1}(\bZ,H^{\vee}(g))$ is the extension class of the short exact sequence \eqref{SESchutb}. However, it straightforwardly follows from the formal definitions of the $1$-cocycles that $k_{\xi}^{\vee}(g,\varphi)(\xi)=\varphi k_{\xi}^{\vee}(g,\varphi^{-1})=-k_{\xi}^{\vee}(g,\varphi^{-1})$ for all $\varphi\in \Gamma_{g,1}$, so:
\begin{lem}\label{lem:cocycles_relations}
For each $g\geq0$, we have $k^{\vee}(g)=-k(g)$ in $H^1(\Gamma_{g,1}; H^{\vee}(g))$.
\end{lem}
This equality does not depend on the self-duality \eqref{eq:dual_H}.
Furthermore, although the sequences \eqref{SEShutb} and \eqref{SESchutb} split as abelian groups short exact sequences, they clearly do not as $\Gamma_{g,1}$-modules. A fortiori, contrary to the homology and cohomology of the surface (where Poincar{\'e}-Lefschetz duality applies), the dual $\tilde{H}^{\vee}(g)$ is not isomorphic to $\tilde{H}(g)$ as a $\Gamma_{g,1}$-module.

\subsection{Twisted coefficient systems}\label{ss:twisted_coefficient_systems}

We recall here the suitable categories to encode representations of the mapping class groups. It is a mild variation of the one introduced in \cite[\S5.6.1]{WahlRandal-Williams}; see also \cite[\S2.1]{KawazumiSoulieI}.
We consider the groupoid $\cM_{2}$ defined by a skeleton of the category of smooth compact connected orientable surfaces with one boundary component, and the isotopy classes of diffeomorphisms restricting to the identity on the boundary component. By the classification of surfaces 
and because $\cM_{2}$ is skeletal, the objects of $\cM_{2}$ is the set of some fixed surfaces $\{\Sigma_{g,1}\}_{g\in\bN}$. For simplicity, we identify the surface $\Sigma_{g,1}$ with its indexing integer $g$, especially when applying a functor on that object.
The groupoid $\cM_{2}$ has a braided monoidal structure $\natural$ induced by the boundary connected sum. The \emph{Quillen bracket construction} applied to $\cM_{2}$, denoted by $\U\cM_{2}$ is the category with the same objects as $\cM_{2}$ and for morphisms
$$\U\cM_{2}(\Sigma_{g,1},\Sigma_{g',1})=\textrm{Colim}_{\cM_{2}}[\cM_{2}(-\natural \Sigma_{g,1},\Sigma_{g',1})].$$
We may now encode \emph{compatible representations} of mapping class groups by considering functors with the category $\U\cM_{2}$ as source and a module category as target. We distinguish two types of such functors because of their distinct qualitative properties with respect to cohomological stability shown in \S\ref{ss:homological_stability}.
A \emph{covariant system} over $\U\cM_2$ is a functor $F\colon \U\cM_2 \to \Ab$. We recursively define the notion of \emph{polynomiality} of covariant systems as follows:
\begin{itemizeb}
\item the constant functors $\U\cM_2 \to \Ab$ are the polynomial covariant systems of degree $0$;
\item for an integer $d\geq1$, the functor $F\colon \U\cM_2 \to \Ab$ is a 
polynomial covariant system of degree less than or equal to $d$ if the morphism $F([\Sigma_{1,1},\id_{\Sigma_{1,1}\natural S}])$ is injective for each surface $S$ of $\U\cM_2$, and the induced functor $\delta_{1}(F)\colon \U\cM_2 \to \Ab$ defined by $S\mapsto (\Cok(F(S) \to F(\Sigma_{1,1} \natural S))$ is a polynomial covariant system of degree less than or equal to $d-1$.
\end{itemizeb}

\begin{rmk}
The above notion of polynomiality corresponds to that of \emph{finite degree coefficient systems at $0$} in the sense of \cite[Def.~4.10]{WahlRandal-Williams}, and to that of \emph{very strong polynomiality} from \cite[Def.~4.3]{soulieLMgeneralised}. 
\end{rmk}

We record here the two following classical properties allowing us to estimate degrees of coefficient systems, whose proofs may be found in \cite[Prop.~4.4]{soulieLMgeneralised} and \cite[Lem.~2.10]{soulie3} respectively.
\begin{prop}\label{prop:standard_results_polynomiality}
A covariant system which is a quotient or an extension of two polynomial covariant systems of degree less than or equal to $d$ is polynomial of degree at most $d$.

Let $F,G \colon \U\cM_2 \to \Ab$ be polynomial covariant systems of degree $d$ and $d'$ respectively. The \emph{tensor product} of $F$ and $G$ is the functor $F \otimes_{\bZ}G \colon \U\cM_2 \to \Ab$ defined by $(F \otimes_{\bZ}G )(g) = F(g) \otimes_{\bZ}G(g)$ for each $g\geq 0$. Then $F \otimes_{\bZ}G$ is polynomial of degree less than or equal to $d+d'$.
\end{prop}
\begin{eg}\label{eg:ext_powers_H_tildeH}
A first example of polynomial covariant system is given by the first homology group of the surfaces, defining a covariant system $H\colon(\U\cM_2,\natural,\bD^{2}) \to (\Ab,\oplus,0)$, which is strong monoidal and thus polynomial of degree one; see \cite[\S2.1]{KawazumiSoulieI}. Also, the constant functor at any ring $R$ defines a polynomial covariant system $R\colon\U\cM_2 \to \Ab$ of degree $0$.
Moreover, we prove in \cite{KawazumiSoulieI} that the groups $\{\tilde{H}(g);g\in\bN\}$ define a polynomial covariant system $\tilde{H}\colon\U\cM_2 \to \Ab$ of degree $1$; see \cite[Prop.~3.3]{KawazumiSoulieI}. Finally, for each $d\geq1$, it is easy to deduce from Proposition~\ref{prop:standard_results_polynomiality} that the $d^{\ith}$ exterior powers of $H$ and $\tilde{H}$ define polynomial covariant systems $\U\cM_2 \to \Ab$ of degree exactly $d$, denoted by $\Lambda^{d}H$ and $\Lambda^{d}\tilde{H}$ respectively. We also consider rational versions of the above functors, namely functors $\Lambda^{d}H_{\bQ}$ and $\Lambda^{d}\tilde{H}_{\bQ}$ of the form $\U\cM_{2}\to\bQ\Mod\to\Ab$ for all $d\geq1$, satisfying the exact same polynomiality properties.
\end{eg}

We also note the following property for the covariant systems of Example~\ref{eg:ext_powers_H_tildeH} which will be of key use for cohomological stability with twisted coefficients (see \S\ref{s:stable_cohomology_framework} and \S\ref{sss:exotic_framework}):
\begin{lem}\label{lem:canonical_splitting_functors}
For each $d\geq1$, the covariant systems $\Lambda^{d}H$ and $\Lambda^{d}\tilde{H}$ are functors $\U\cM_2 \to \Ab$ such that the image of each morphism of type $[\Sigma_{1,1},\id_{\Sigma_{g+1,1}}]$ has a canonical $\Gamma_{g,1}$-equivariant splitting.
\end{lem}
\begin{proof}
For $\Lambda^{d}H$, this follows from the fact that the functor $H$ is a strong monoidal functor $(\U\cM_2,\natural,\bD^{2}) \to (\Ab,\oplus,0)$. We denote by $H^{-1}([\Sigma_{1,1},\id_{\Sigma_{g+1,1}}])$ the canonical splitting of the map $H([\Sigma_{1,1},\id_{\Sigma_{g+1,1}}])$. For $\Lambda^{d}\tilde{H}$, we deduce by a clear computation using the representation structure \eqref{eq:matrix_action} that the canonical splitting is defined by $\id_{\bQ}\oplus H^{-1}([\Sigma_{1,1},\id_{\Sigma_{g+1,1}}])$.
\end{proof}
A \emph{contravariant system} over $\U\cM_2$ is a functor $F\colon \U\cM_2^{\opp} \to \Ab$.
\begin{eg}\label{eg:tilde_H_functor_vee}
For each $d\geq1$, the $d^{\ith}$ exterior power of the first cohomology groups of the unit tangent bundle of the surfaces define a contravariant system $\Lambda^{d}\tilde{H}^{\vee}\colon\U\cM_2^{\opp} \to \Ab$. We also consider the rational version $\Lambda^{d}\tilde{H}^{\vee}\colon\U\cM_{2}\to\bQ\Mod\to\Ab$ for each $d\geq1$.
\end{eg}
\begin{rmk}\label{rmk:UM2opp_poly}
Although a natural notion of polynomiality may be defined over the category $\U\cM_{2}^{\opp}$ (following for instance the analogous opposite notions of \cite[\S2]{DV3} or \cite[\S4]{soulieLMgeneralised}), as far as we know, such notion has no application for the questions addressed in this paper.
\end{rmk}

\begin{rmk}
The existence of $\Gamma_{g,1}$-equivariant splitting for each $[\Sigma_{1,1},id_{\Sigma_{g+1,1}}]$ does not imply that a functor $F\colon \U\cM_2 \to \Ab$ is contravariant. For instance, $\tilde{H}$ cannot be turned into a contravariant system, while $H$ actually can because it is strong monoidal.
\end{rmk}


\subsection{Cohomological stability and stable (co)homology}\label{ss:homological_stability}

We now review some classical results on cohomological stability with twisted coefficients for mapping class groups. In particular, all the twisted coefficient systems in this paper satisfy the cohomological stability property, which motivates the computations of \S\ref{s:contravariant_coeff_syst_exterior}--\S\ref{s:covariant_coeff_syst_exterior_computations}. Also, we recall some stable cohomology computations of mapping class groups further used in \S\ref{s:contravariant_coeff_syst_exterior}--\S\ref{s:covariant_coeff_syst_exterior_computations}.

\subsubsection{Classical framework and results}\label{s:stable_cohomology_framework}

\paragraph*{Stable twisted cohomology framework.}
First of all, we recall the notions of cohomological stability and stable cohomology with twisted coefficients. For $F\colon \U\cM_2 \to \Ab$ a covariant system, we note that $F$ induces a $\Gamma_{g,1}$-equivariant morphism $F([\Sigma_{1,1},\id_{\Sigma_{g+1,1}}])\colon F(g) \to F(g+1)$ for each $g\geq0$. We recall that, viewing $\Sigma_{g,1}$ as a subsurface of $\Sigma_{g+1,1}\cong \Sigma_{1,1}\natural \Sigma_{g,1}$ and extending the diffeomorphisms of $\Sigma_{g,1}$ by the identity on the complement $\Sigma_{1,1}$, we define a canonical injection $\Gamma_{g,1}\hookrightarrow\Gamma_{g+1,1}$. These maps induce maps in cohomology for all $i,g\geq0$:
\begin{equation}
    \Phi_{i,g}\colon H^{i}(\Gamma_{g+1,1}; F^{\vee}(g+1)) \to H^{i}(\Gamma_{g,1}; F^{\vee}(g)).
\end{equation}

\begin{defn}\label{def:homological_stability_covariant}
Let $F\colon \U\cM_2 \to \Ab$ be a covariant system. For each $i\geq0$, the \emph{stable cohomology group} $H_{\st}^{i}(F^{\vee})$ is the inverse limit $\underleftarrow{\mathrm{Lim}}_{g\geq0}H^{i}(\Gamma_{g,1}; F^{\vee}(g))$ induced by the maps $\{\Phi_{i,g}\}_{g\in\bN}$.
We say that there is \emph{(classical) cohomological stability} with twisted coefficients in $F^{\vee}\colon \U\cM_2^{\opp} \to \Ab$ if, for each $i\geq0$, the maps $\Phi_{i,g}$ for $g\geq0$ are isomorphisms if $g\geq B(i,F)$, where $B(i,F)\in\bN$ is a bound depending on $i$ and $F$.
\end{defn}

We now recall the following fundamental result on twisted (classical) cohomological stability for mapping class groups.

\begin{thm}[{\cite[Th.~4.1]{Ivanov}, \cite[Th.~5.26]{WahlRandal-Williams}}]\label{thm:hom_stab_MCG}
Let $F\colon \U\cM_2 \to \Ab$ be a polynomial covariant system of degree $d$. Then, there is cohomological stability with twisted coefficients in $F^{\vee}$, where the stability bound $B(i,F)$ is equal to $2i+2d+3$.
\end{thm}
In particular, we deduce from Example~\ref{eg:ext_powers_H_tildeH} and Theorem~\ref{thm:hom_stab_MCG} that there is cohomological stability for the mapping class groups with twisted coefficients in $\Lambda^{d}H$ (via the isomorphism \eqref{eq:dual_H}) and in $\Lambda^{d}\tilde{H}^{\vee}$ for all $d\geq1$. \eqref{eq:Madsen_weiss_thm}

Furthermore, we have the following general formula for the stable cohomology of a family of groups with twisted coefficients given by the dual of a covariant system. It is a rewording of \cite[Th.~C]{soulie3} in terms of twisted cohomology groups.
\begin{thm}\label{thm:stable_functor_homology}
Let $F\colon \U\cM_2 \to \bQ\Mod$ be any covariant system. Note that the cup product with elements of $H^{*}(\Gamma_{\infty,1};\bQ)$ induces a canonical $\Sym_{\bQ}(\cE)$-module structure on the graded stable twisted cohomology groups $H_{\st}^{*}(F^{\vee}):=\bigoplus_{i\in\bN} H_{\st}^{i}(F^{\vee})$. We have a natural isomorphism of $\bQ$-vector spaces for each $i\geq0$:
$$H_{\st}^{i}(F^{\vee})\cong\underset{k+l=i}{\bigoplus}H^{k}(\Sym_{\bQ}(\cE)\underset{\bQ}{\otimes} H_{l}(\U\cM_{2};F).$$
Here, $H_{l}(\U\cM_{2};F)$ denotes the homology of the category $\U\cM_{2}$ with coefficient in $F$; see \cite[Appendice A]{DV1} for instance.
Therefore, the $\Sym_{\bQ}(\cE)$-module $H_{\st}^{*}(F^{\vee})$ is free.
\end{thm}

\begin{rmk}
The result of Theorem~\ref{thm:stable_functor_homology} does not depend on any polynomiality condition and on whether there is cohomological stability or not. The formula is for the limit of the cohomology groups which always exists.
Also, it may be stated with any field $\bK$ as ground ring rather than just $\bQ$.
\end{rmk}

\paragraph*{Twisted Mumford-Morita-Miller classes.}
We now review the theory of \emph{twisted} Mumford-Morita-Miller classes, drawing on the works of the first author \cite{KawazumitwistedMMM,stablecohomologyKawazumi,KawazumiMorita}. For the reference \cite{KawazumiMorita}, we will rather quote the preprint version \cite{KawazumiMorita1} as it contains more material and details. In the original definition \cite{KawazumitwistedMMM} these cohomology classes are constructed on the group $\Gamma_{g,1}$. The following alternative definition is introduced in \cite{KawazumiMorita1}.

First of all, for any $g\geq1$, we recall that gluing a disc with a marked point $\Sigma_{0,1}^{1}$ on the boundary component of $\Sigma_{g,1}$ induces the following short exact sequence:
\begin{equation}\label{eq:Birman_SES_2}
1 \longrightarrow
\bZ \longrightarrow
\Gamma_{g,1} \overset{\mathrm{Cap}}\longrightarrow
\Gamma_{g}^{1} \longrightarrow
1,
\end{equation}
where $\Gamma_{g}^{1}$ denotes the mapping class group of the punctured surface $\Sigma_{g}^{1}$; see \cite[\S3.6.2]{farbmargalit}.
We denote by $e\in H^2(\Gamma_{g}^{1};\bZ)$ the Euler class of the short exact sequence \eqref{eq:Birman_SES_2} seen as a central extension.

We now define a key cohomology class $\overline{\omega}\in H^1(\Gamma_{g}^{1}; H(g))$ defined as follows. Let $p\colon \Gamma_{g}^{1} \twoheadrightarrow \Gamma_{g}$ be the forgetful map of the puncture, whose kernel is isomorphic to $\pi_1(\Sigma_{g})$; see \cite[Th.~4.6]{farbmargalit}.
We denote by $\overline{\Gamma}_{g}^{1}$ the pullback $\Gamma_{g}^{1}\times_{\Gamma_{g}} \Gamma_{g}^{1}$. More precisely, there is a defining fiber square
\[
\xymatrix{
\overline{\Gamma}_{g}^{1} \ar[r] \ar[d]_-{\pi}	&  \Gamma_{g}^{1} \ar[d]^-{\overline{p}}\\
\Gamma_{g}^{1} \ar[r]_{p}	 \ar@/_/_{\sigma}[u]						& \Gamma_{g},
}
\]
where the section $\sigma: \Gamma_{g}^{1} \to \overline{\Gamma}_{g}^{1}$ is given by $\sigma(\phi) = (\phi, \phi)$. Then there is an isomorphism $\overline{\Gamma}_{g}^{1} \cong   \pi_1(\Sigma_{g}) \rtimes\Gamma_{g}^{1}$
defined by $(\phi, \psi) \mapsto (\psi \phi^{-1}, \phi)$.
Under this isomorphism, $\sigma$ is given by $\sigma(\phi) = (1, \phi)$ and the action of $\overline{\Gamma}_{g}^{1}$ on $H(g)$ is induced by that of $\Gamma_{g}^{1}$ using the projection $\overline{\Gamma}_{g}^{1}\twoheadrightarrow \Gamma_{g}^{1}$ defined by the semi-direct product decomposition.
 Similarly to \cite[\S7]{MoritaJacobianII}, this gives rise to a cocycle $\overline{\omega} \in Z^1(\overline{\Gamma}_{g}^{1}; H(g))$ given by $\overline{\omega}((x, \phi)) = [x]$ for all $x\in \pi_1(\Sigma_{g})$ and $\phi\in \Gamma_{g}^{1}$. By an abuse of notation, we also use $\overline{\omega}$ to denote the associated element of $H^1(\overline{\Gamma}_{g}^{1}; H(g))$.
Finally, we denote by $\bar{e} \in H^2(\overline{\Gamma}_{g}^{1}; \bZ)$ the pullback of the Euler class $e \in H^2(\Gamma_{g}^{1}; \bZ)$ along the projection $\pi\colon\overline{\Gamma}_{g}^{1} \to \Gamma^{1}_{g}$, $(\phi, \psi) \mapsto \psi$. The \emph{twisted} Mumford-Morita-Miller classes are defined as follows:
\begin{defn}\label{def:mij}
Let $i\geq0$ and $j\geq0$ be integers such that $i+j\geq2$. We denote by $T^{j}:\bQ\Mod \to \bQ\Mod$ the $j^{\ith}$ tensor product functor where the order of the powers is that of $\bj$.
The \emph{generalised twisted Mumford-Morita-Miller class} $m_{i,\bj}$ is the pullback to $H^{2i +j -2}(\Gamma_{g,1}; T^{j}H(g))$ along $\mathrm{Cap}:\Gamma_{g,1}\twoheadrightarrow \Gamma_{g}^{1}$ of the class
\begin{equation}\label{eq:def:mij_tensor}
    \pi_{!}(\bar{e}^{i}\cup\overline{\omega}^{\bj}) \in H^{2i+j-2}(\Gamma^{1}_{g}; T^{j}H_{\bQ}(g)).
\end{equation}
Here, $\pi_{!}$ denotes the Gysin map induced by $\pi\colon \overline{\Gamma}_{g}^{1}\to \Gamma_{g}^{1}$ and $\overline{\omega}^{\bj}$ is the cup product of the class $\overline{\omega}$ following the numerical order of $\bj$.
The alternating operator $A_{\bj} \colon T^{j}H_{\bQ}(g)\to T^{j}H_{\bQ}(g)$ is defined for $v_i \in H_{\bQ}(g)$ by 
\begin{equation}\label{eq:alternation}
A_{\bj}(v_1\otimes \cdots \otimes v_{j}) := 
\sum_{\tau\in \mathfrak{S}_j} \mathrm{sgn}(\tau)(v_{\tau(1)}\otimes\cdots \otimes v_{\tau(j)}).
\end{equation}
The exterior algebra is defined by the image of this alternating operator $A_{\bj}$. The image of the class $m_{i,\bj}$ by the projection map $H^{*}(\Gamma^{1}_{g}; T^{j}H_{\bQ}(g))\twoheadrightarrow  H^{*}(\Gamma^{1}_{g}; \Lambda^{j}H_{\bQ}(g))$ induced by \eqref{eq:alternation} is denoted by $m_{i,j}$. Because of the graded-commutativity of the cup product, the class $m_{i,\bj}$ belongs to the image of the homomorphism induced by the map \eqref{eq:alternation}. Then, we deduce that $m_{i,j} = j!\cdot m_{i,\bj}$ in $H_{\bQ}^{*}(\Gamma^{1}_{g}; \Lambda^{j}H_{\bQ}(g))$.
\end{defn}
When $j = 0$, Definition~\ref{def:mij} specialises to $m_{i+1,0} = \pi_{!}(\bar{e}^{i+1}) = e_{i}$ recovers the $i^{\ith}$ classical Mumford-Morita-Miller class.
\begin{rmk}\label{rmk:identification_m_1,1}
For $i=j=1$ and $g\geq 2$, using Lemma~\ref{lem:cocycles_relations}, the extension classes $[k_{\xi}(g,-)]$ and $-[k^{\vee}_{\xi}(g,-)]$ of \S\ref{ss:MCG_representations} identify with $m_{1,1}\in H_{\st}^1(H_{\bQ})$ under the duality isomorphism \eqref{eq:dual_H}.
\end{rmk}

When we consider the cup product structure for the cohomology with coefficients in the \emph{exterior} powers, we stress that we implicitly do the composite
\begin{equation}\label{eq:cup_exterior}
H_{\bQ}^{*}(\Gamma^{1}_{g};\Lambda^{j}H_{\bQ}(g))\otimes H_{\bQ}^{*}(\Gamma^{1}_{g};\Lambda^{j'}H_{\bQ}(g))\overset{\cup}{\to} H_{\bQ}^{*}(\Gamma^{1}_{g};\Lambda^{j}H_{\bQ}(g)\otimes \Lambda^{j'}H_{\bQ}(g))\twoheadrightarrow H_{\bQ}^{*}(\Gamma^{1}_{g};\Lambda^{j+j'}H_{\bQ}(g)),
\end{equation}
where the first map is the classical cup product morphism and the second map is the canonical projection map induced from \eqref{eq:alternation}.
Furthermore, the alternating operator \eqref{eq:alternation} induces a lift of the exterior algebra $L_{j}:\Lambda^{j}H_{\bQ}(g)\to T^{j}H_{\bQ}(g)$. This induces a map
$$L_{j}^{*}\colon H_{\bQ}^{*}(\Gamma^{1}_{g}; \Lambda^{j}H_{\bQ}(g))\to H_{\bQ}^{*}(\Gamma^{1}_{g}; T^{j}H_{\bQ}(g)).$$
We note that the map \eqref{eq:cup_exterior} is also naturally induced from that of the cohomology with coefficients in the \emph{tensor} powers, using the canonical projection map $H^{*_1}(\Gamma^{1}_{g}; T^{*_2}H_{\bQ}(g))\twoheadrightarrow H^{*_1}(\Gamma^{1}_{g}; \Lambda^{*_2}H_{\bQ}(g))$ induced from \eqref{eq:alternation} and the associated lift map $L_{*_2}^{*_1}$.

\begin{convention}
The lift map $L_{j}^{*}$ provides a canonical decomposition in $H_{\bQ}^{*}(\Gamma^{1}_{g}; T^{j}H_{\bQ}(g))$ associated of the cup product of two classes $m_{i,j}m_{k,l}$ following \eqref{eq:alternation}, that we arrange following the canonical ordering using the graded-commutativity of the cup product. For example, the class $m_{i,1}m_{k,1}$ is identified with $m_{i,1}\otimes m_{k,1} + m_{k,1}\otimes m_{i,1}$ via $L_{2}^{*}$. This type of decomposition will often implicitly used for our computations; see Proposition~\ref{prop:Contraction_Formula} for instance.
\end{convention}

The main results related to these classes are the computations of the stable cohomology with twisted coefficients in $\{\Lambda^{d}H_{\bQ}(g)\}_{d\geq1}$ by the first author \cite[Th.~3.3]{stablecohomologyKawazumi}, and those with twisted coefficients in $\tilde{H}^{\vee}_{\bQ}$ by \cite[Th.~A]{KawazumiSoulieI}. We recall from Theorem~\ref{thm:stable_functor_homology} that the graded stable twisted cohomology groups $H_{\st}^{*}(\tilde{H}^{\vee}_{\bQ})$ and $H_{\st}^{*}(\Lambda^{d}H_{\bQ})$ are free $\Sym_{\bQ}(\cE)$-modules.
Beforehand, we recall that:
\begin{defn}\label{def:weighted_partition}
A \emph{weighted partition} of the number $d$ is a set of numbers $((i_{1},j_{1}),\ldots,(i_{\nu},j_{\nu}))$) where $i_{1},\ldots,i_{\nu}$ are non-negative integers, $j_{1}\geq\ldots\geq j_{\nu}$ are non-negative integers such that $j_{1}+\dots+j_{\nu}=d$, and each $(i_{a},j_{a})$ satisfies the condition $i_a + j_a \geq 2$ and $i_a\geq i_{a+1}$ if $j_a = j_{a+1}$.
We denote by $\cQ_{d}$ the set of all weighted partitions of the number $d$. For $Q=((i_{1},j_{1}),\ldots,(i_{\nu},j_{\nu}))\in \cQ_{d}$, the cup product $m_{i_{1},j_{1}} m_{i_{2},j_{2}}\cdots m_{i_{\nu},j_{\nu}}$ of the twisted Mumford-Morita-Miller classes defined in Definition~\ref{def:mij} is denoted by $m_{Q}$.
\end{defn}

\begin{thm}[{\cite[Th.~3.3]{stablecohomologyKawazumi}; \cite[Th.~A]{KawazumiSoulieI}}]\label{stablecohomologyKawazumi}
There is a graded natural isomorphism of free $\Sym_{\bQ}(\cE)$-modules
\begin{equation}\label{eq:stablecohomologyKawazumi}
H_{\st}^{*}(\Lambda^{d}H_{\bQ})\cong \bigoplus_{Q\in \cQ_{d}} \Sym_{\bQ}(\cE)m_{Q}.
\end{equation}
Furthermore, we consider the bigraded $\Sym_{\bQ}(\cE)$-module $H_{\st}^{*}(\Lambda^{*}H_{\bQ}) := \bigoplus_{d\in\bN} H_{\st}^{*}(\Lambda^{d}H_{\bQ})$. The cup product induces a commutative\footnote{The graded-commutativities of the cup product and the exterior powers cancel each other, so there are no Koszul signs.} bigraded $\Sym_{\bQ}(\cE)$-algebra structure on $H_{\st}^{*}(\Lambda^{*}H_{\bQ})$. A fortiori, the commutative bigraded algebra $H_{\st}^{*}(\Lambda^{*}H_{\bQ})$ is the polynomial algebra in the twisted Mumford-Morita-Miller classes $\mathfrak{M} := \{m_{i,j};i\geq0,j\geq1\}$ over the ring $\Sym_{\bQ}(\cE)$. 

Finally, the $\Sym_{\bQ}(\cE)$ $H_{\st}^{*}(\tilde{H}^{\vee}_{\bQ})$ is isomorphic to the free $\Sym_{\bQ}(\cE)$-module with basis $\{m_{i,1},i\geq2\}$.
\end{thm}

Finally, we introduce the following mild variation of the twisted Mumford-Morita-Miller classes which are more convenient to handle in \S\ref{s:contravariant_coeff_syst_exterior}--\S\ref{s:covariant_coeff_syst_exterior_computations}:
\begin{notation}\label{nota:barm}
For all $i+j\geq1$, we denote by $\mm_{i,j}$ denote the cohomology class $((-1)^j/j!)\cdot m_{i,j}$.
\end{notation}

\paragraph*{Contraction formula.}
We recall a classical operation on the twisted Mumford-Morita-Miller classes induced by the contraction of the twisted coefficients.
Let $\mu\colon H(g)\otimes H(g)\to\bZ$ be the intersection pairing associated to Poincar{\'e}-Lefschetz duality. Let $M$ and $M'$ be two $\Gamma_{g,1}$-modules. The induced \emph{contraction map} for the cohomology groups $H^{*}(\Gamma_{g,1};M\wedge (H(g)\otimes H(g)) \wedge M') \to H^{*}(\Gamma_{g,1};M \wedge M')$ is generically denoted by $(\id_{M}\wedge\mu\wedge \id_{M'})_{*}$. The following formulas are directly deduced from \cite[Th.~6.2]{KawazumiMorita1}, we sketch a proof below for the sake of completeness.
\begin{prop}\label{prop:Contraction_Formula}
We have the following formulas for all integers $i,k\geq0$ and $j,l\geq1$ such that $i+j\geq2$ and $k+l\geq2$:
\begin{equation}\label{eq:m11mij}
\mu_{*}(m_{1,1}, m_{i,j}) = -jm_{i+1, j-1};
\end{equation}
\begin{equation}\label{eq:m11mij_1}
(\mu\wedge\id_{H(g)})_{*}(m_{1,1},m_{i,j}m_{k,l}) = -jm_{i+1, j-1}m_{k,l} - l m_{k+1, l-1}m_{i,j}.
\end{equation}
More generally, the computation of $(\mu\wedge\id_{\Lambda^{d-1}H(g)})_{*}(m_{1,1},m_{i_1,j_1}\cdots m_{i_d,j_d})$ is a straightforward generalisation of \eqref{eq:m11mij_1} by transiting trough the lift map $L_{j}^{*}$ and applying \eqref{eq:m11mij}.
\end{prop}

\begin{proof}[Sketch of proof]
As is proved in \cite[Th.~5.3]{KawazumiMorita1}, the Lyndon-Hochschild-Serre spectral sequence for the group extension \eqref{eq:Birman_SES_2} gives a canonical decomposition for the twisted cohomology of $\Gamma^{1}_{g}$. The following formula is then deduced from 
a direct computation based on this decomposition. For $M$ and $M'$ two $\Gamma_{g,1}$-modules, for $m\in H^{*}(\Gamma^{1}_{g};M)$ and $m'\in H^{*}(\Gamma^{1}_{g};M')$, we have
\begin{eqnarray*}
&&(\id_{M}\otimes\mu\otimes \id_{M'})_{*}(\pi_{!}(m\otimes \overline{\omega})\cup\pi_{!}(\overline{\omega}\otimes m')) \\
&& = -\pi_{!}(m\otimes m') + \sigma^{*}(m)\pi_{!}(m') + \pi_{!}(m)\sigma^{*}(m') - e\pi_{!}(m)\pi_{!}(m')
\end{eqnarray*}
where we recall that $\sigma: \Gamma^{1}_{g} \to \overline{\Gamma}^1_{g}$ is the diagonal map. We then deduce the results from that general formula using the fact that the Euler class $e$ vanishes on $\Gamma_{g,1}$.
\end{proof}


\subsubsection{Exotic framework}\label{sss:exotic_framework}

On the basis of current knowledge, there is no general framework for cohomological stability with twisted coefficients taking the functor $F^{\vee}$ to be $\Lambda^{d}\tilde{H}\colon\U\cM_{2} \to\Ab$ in Theorems~\ref{thm:hom_stab_MCG} and \ref{thm:stable_functor_homology}.
However, we have the following framework to deal with cohomological stability questions with this kind of twisted coefficient.

Let $M\colon \U\cM_2 \to \Ab$ be a covariant system such that the image of each morphism of type $[\Sigma_{1,1},\id_{\Sigma_{g+1,1}}]$ has a canonical $\Gamma_{g,1}$-equivariant splitting denoted by $M^{-1}([\Sigma_{1,1},\id_{\Sigma_{g+1,1}}])$. The canonical injection $\Gamma_{g,1}\hookrightarrow\Gamma_{g+1,1}$ and the splitting $M^{-1}([\Sigma_{1,1},\id_{\Sigma_{g+1,1}}])$ induce a map in cohomology $\Phi'_{i,g}\colon H^{i}(\Gamma_{g+1,1}; M(g+1)) \to H^{i}(\Gamma_{g,1}; M(g))$ for each $i,g\geq0$.
\begin{defn}\label{def:homological_stability_contravariant}
For each $i\geq0$, the \emph{stable twisted cohomology group} $H_{\st}^{i}(M)$ is the inverse limit $\underleftarrow{\mathrm{Lim}}_{g\geq0}H^{i}(\Gamma_{g,1}; M(g))$ induced by the maps $\{\Phi'_{i,g}\}_{g\in\bN}$. Also, we say that there is \emph{(exotic) cohomological stability} with twisted coefficients in $M\colon \U\cM_2 \to \Ab$ if, for each $i\geq0$, the maps $\Phi_{i,g}$ for $g\geq0$ are isomorphisms if $g\geq B(i,M)$, where $B(i,M)\in\bN$ is a bound depending on $i$ and $M$.
\end{defn}

\begin{lem}\label{lem:homological_stability_covariant}
Let $F_{2}$ be a functor $\U\cM_{2}\to\Ab$ such that the image of each morphism of type $[\Sigma_{1,1},\id_{\Sigma_{g+1,1}}]$ has a canonical $\Gamma_{g,1}$-equivariant splitting.
Let $F_{1}$ and $F_{3}$ be polynomial functors $\U\cM_{2}\to\Ab$ such that we have a short exact sequence $0\rightarrow F^{\vee}_{1}\rightarrow F_{2}\rightarrow F^{\vee}_{3}\rightarrow0$ of functors $\U\cM_{2}^{\opp}\to\Ab$.
Then, there is cohomological stability with twisted coefficients in $F_{2}$.
\end{lem}
\begin{proof}
By Theorem~\ref{thm:hom_stab_MCG}, there is cohomological stability with twisted coefficients given by $F_{1}$ and $F_{3}$.
Considering the long exact sequences in cohomology, we obtain the following commutative diagram for each $g\geq0$ and assuming that we are in the stability range for $F_{1}^{\vee}$ and $F_{3}^{\vee}$:
$$\xymatrix{\cdots\ar@{->}[r] & 
H^{i}(\Gamma_{g,1};F_{1}^{\vee}(g)) \ar@{->}[r]  \ar@{->}[d]^-{F_1^{\vee}([\Sigma_{1,1},\id_{\Sigma_{g+1,1}}])}_{\cong}
& H^{i}(\Gamma_{1,1};F_{2}(g))\ar@{->}[d]^-{F_2^{-1}([\Sigma_{g,1},\id_{\Sigma_{g+1,1}}])} \ar@{->}[r]
& H^{i}(\Gamma_{g,1};F_{3}^{\vee}(g)) \ar@{->}[d]^-{F_3^{\vee}([\Sigma_{1,1},\id_{\Sigma_{g+1,1}}])}_{\cong} \ar@{->}[r] & \cdots \\
\cdots\ar@{->}[r]
& H^{i}(\Gamma_{g+1,1};F_{1}^{\vee}(g+1))\ar@{->}[r]
& H^{i}(\Gamma_{g+1,1};F_{2}(g+1))\ar@{->}[r]
& H^{i}(\Gamma_{g+1,1};F_{3}^{\vee}(g+1)) \ar[r]
& \cdots}$$
The results thus follows from a clear recursion and the five lemma.
\end{proof}
\begin{eg}\label{eg:cohomological_stability_contra}
Using the short exact sequence \eqref{exteriorSES_contra}, it follows from Lemmas~\ref{lem:canonical_splitting_functors}, \ref{lem:homological_stability_covariant} and Theorem~\ref{thm:hom_stab_MCG} that there is cohomological stability for mapping class groups with twisted coefficients in $\Lambda^{d}\tilde{H}_{\bQ}\colon\U\cM_{2} \to \bQ\Mod$ for all $d\geq1$.
\end{eg}
Also, the stable twisted cohomology groups have natural $\Sym_{\bQ}(\cE)$-module structures in the exotic situations:
\begin{lem}\label{lem:canonical_sym_E_structure}
Let $M\colon \U\cM_{2}\to\bQ\Mod$ be a functor with canonical splittings as in Lemma~\ref{lem:homological_stability_covariant}.
The graded stable twisted cohomology groups $H_{\st}^{*}(M)$, $H_{\st}^{\odd}(M)$ and $H_{\st}^{\even}(M)$ inherit canonical $\Sym_{\bQ}(\cE)$-module structures from the cup product with elements of $H^{*}(\Gamma_{\infty,1};\bQ)$.
Then, the decomposition $H_{\st}^{*}(M) = H_{\st}^{\even}(M)\oplus H_{\st}^{\odd}(M)$ is stable under the action of the algebra $\Sym_{\bQ}(\cE)$.
\end{lem}
\begin{proof}
Since $\Sym_{\bQ}(\cE)$ is concentrated in even degrees, the $\Sym_{\bQ}(\cE)$-action via the cup product satisfies the compatibility with respect to the grading of $H_{\st}^{*}(M)$ providing the result.    
\end{proof}

Finally, we recall the stable cohomology computations of \cite{KawazumiSoulieI} for this type of exotic situations.
\begin{thm}[{\cite[Th.~B]{KawazumiSoulieI}}]\label{thm:main_Thm_1}
The stable twisted cohomology module $H_{\st}^{*}(\tilde{H}_{\bQ})$ is isomorphic to the direct sum $\bQ\theta\bigoplus \mathscr{M}$, where $\mathscr{M}$ is the $\Sym_{\bQ}(\cE)$-module generated by the classes $M_{i,j}:= e_{i}m_{j,1}- e_{j}m_{i,1}$ for all distinct $ j,i \geq 1$, and with relations $ M_{j,i}=- M_{i,j}$
$e_{i} M_{j,k} + e_{j} M_{k,i} + e_{k} M_{i,j} \sim 0$ for all pairwise distinct $i,j,k\geq 1$.
\end{thm}

\subsection{Computations of the homology groups over a polynomial algebra}\label{ss:homology_of_algebras}

Following Theorem~\ref{thm:stable_functor_homology} and Lemma~\ref{lem:canonical_sym_E_structure}, we may consider the twisted cohomology groups we study as forming a module over the ring $\Sym_{\bQ}(\cE)$.
In order to give some qualitative properties on these algebras, we compute the homology group of the twisted cohomology modules. The following results for the homology groups of an algebra will be of key use in \S\ref{ss:theory_Tor} and \S\ref{s:covariant_coeff_syst_exterior_computations}; see \cite[\S1]{Loday} or \cite[\S~IX]{CartanEilenberg} for a complete introduction to this topic.

We recall that we denote $\Tor_{*}^{\Sym_{\bQ}(\cE)}(\bQ,-)$ by $H_{*}(\Sym_{\bQ}(\cE);-)$ for simplicity.
For $\ell$ an integer, we denote by $\Sym_{\bQ}(\cE_{\leq \ell})$ and $\Sym_{\bQ}(\cE_{\geq \ell +1})$ the subalgebras $\bQ [e_{1},\ldots,e_{\ell}]$ and $\bQ [e_{\ell+1},e_{\ell+2},\ldots]$ of $\Sym_{\bQ}(\cE)$ respectively.
We first have a general decomposition result for the $\Tor$-groups of a tensor product module with respect to these subalgebras:
\begin{prop}\label{lem:general_properties_homology_algebra}
Let $L$ be a $\Sym_{\bQ}(\cE_{\leq \ell})$-module and $L'$ a $\Sym_{\bQ}(\cE_{\geq \ell +1})$-module. We also consider $(m)$ the ideal of the $\bQ$-algebra $\Sym_{\bQ}(\cE_{\leq \ell})$ generated by an element $m$ in the augmentation ideal of $\Sym_{\bQ}(\cE_{\leq \ell})$.
Then there is a graded isomorphism:
\begin{equation}\label{eq:computation_Tor_general}
H_{*}(\Sym_{\bQ}(\cE);L\otimes L')\cong
\begin{cases}
H_{*}(\Sym_{\bQ}(\cE_{\leq \ell});L)\otimes\Lambda^{*}\cE_{\geq \ell+1} & \text{if $L' = \bQ$;}\\
H_{*-1}(\Sym_{\bQ}(\cE_{\geq \ell +1}); L') \oplus H_{*}(\Sym_{\bQ}(\cE_{\geq \ell +1}); L') & \text{if $L = \Sym_{\bQ}(\cE_{\leq \ell})/(m)$.}\\
\end{cases}
\end{equation}
\end{prop}
\begin{proof}
First of all, we note that there is an isomorphism of $\bQ$-algebras $\Sym_{\bQ}(\cE)\cong \Sym_{\bQ}(\cE_{\leq \ell})\otimes_{\bQ} \Sym_{\bQ}(\cE_{\geq \ell+1})$. Then, it follows from the K{\"u}nneth formula for homology of algebras (see for instance \cite[\S~XI, Th.~3.1]{CartanEilenberg}) that we have a graded isomorphism
$$H_{*}(\Sym_{\bQ}(\cE);L\otimes L')\cong H_{*}(\Sym_{\bQ}(\cE_{\leq \ell});L)\otimes H_{*}(\Sym_{\bQ}(\cE_{\geq \ell +1});L').$$
For the first computation of \eqref{eq:computation_Tor_general}, the result follows for the fact that there is an isomorphism of graded $\bQ$-vector spaces $H_{*}(\Sym_{\bQ}(\cE);\bQ)\cong \Lambda^{*}\cE$ using the Koszul resolution (see for instance \cite[\S 3.4.6]{Loday}) or the modules of differential forms (see Lemma~\ref{lem:refined_Tor_method}).
For the second computation of Lemma~\ref{lem:general_properties_homology_algebra}, the short exact sequence $0 \to \Sym_{\bQ}(\cE_{\leq \ell}) \overset{\cdot m}\longrightarrow \Sym_{\bQ}(\cE_{\leq \ell}) \to \Sym_{\bQ}(\cE_{\leq \ell})/(m) \to 0$ defines a $\Sym_{\bQ}(\cE_{\leq \ell})$-free resolution which allows us to compute that
\begin{equation}\label{eq:iso_tor_ideal}
H_{j}(\Sym_{\bQ}(\cE_{\leq \ell}); \Sym_{\bQ}(\cE_{\leq \ell})/(m)) = \begin{cases}
\bQ & \text{if $j = 0,1$,}\\
0 & \text{otherwise},\\
\end{cases}
\end{equation}
thus ending the proof.
\end{proof}

\begin{eg}\label{thm:homology_algebra_first}
Using the tools of Lemma~\ref{lem:general_properties_homology_algebra}, we compute that  $H_{0}(\Sym_{\bQ}(\cE);H_{\st}^{*}(\tilde{H}_{\bQ}))\cong\Lambda^{2}\cE\oplus \bQ\theta$ and that
$H_{j}(\Sym_{\bQ}(\cE));H_{\st}^{*}(\tilde{H}_{\bQ}))\cong\Lambda^{j}\cE\oplus\Lambda^{j+2}\cE$ for any $j > 0$; see \cite[Th.~3.8]{KawazumiSoulieI}.
\end{eg}

Moreover, we have the following lemma in order to deduce general facts about $\Tor$-computations (see Proposition~\ref{prop:Ker_Delta_method}) or to make more refined computations for the $\Tor$-groups with respect to the finite dimensional $\bQ$-algebra $\Sym_{\bQ}(\cE_{\leq \ell})$ (see \S\ref{s:Tor_groups_lambda_5}).
We consider a subspace $\cE'\subseteq\cE$ and the $\Sym_{\bQ}(\cE')$-module $\Omega^{n}_{\Sym_{\bQ}(\cE')|\bQ}$ of $n$-differential forms; see \cite[pp.26-27]{Loday}. The exterior derivative $d: \Omega_{\Sym_{\bQ}(\cE')|\bQ}^{n} \to \Omega^{n+1}_{\Sym_{\bQ}(\cE')|\bQ}$ is defined in a usual way. We consider the Euler vector field
$$D := \sum^\infty_{i=1}e_{i}\dfrac{\pa}{\pa e_{i}}$$ and its associated interior product $p_{D}: \Omega_{\Sym_{\bQ}(\cE')|\bQ}^{n} \to \Omega^{n-1}_{\Sym_{\bQ}(\cE')|\bQ}$.

\begin{lem}\label{lem:refined_Tor_method}
The sequence 
\begin{equation}\label{eq:differential_forms_resolution}
\xymatrix{\cdots\ar@{->}[r]^-{p_{D}} & \Omega_{\Sym_{\bQ}(\cE')|\bQ}^{n}\ar@{->}[r]^-{p_{D}} & \Omega_{\Sym_{\bQ}(\cE')|\bQ}^{n-1}\ar@{->}[r]^-{p_{D}} & \cdots\ar@{->}[r]^-{p_{D}} & \Sym_{\bQ}(\cE')\ar@{->}[r]^-{\aug} & \bQ\ar@{->}[r] & 0.}
\end{equation}
is a free resolution of the trivial $\Sym_{\bQ}(\cE')$-module $\bQ$. Moreover, for any $\Sym_{\bQ}(\cE')$-module $M$, the $\Tor$-groups $H_{*}(\Sym_{\bQ}(\cE');M)$ are computed by the homology group of the chain complex
\begin{equation}\label{eq:chain_complex_Tor}
\xymatrix{\cdots\ar@{->}[r]^-{\pa_{n+1}} & M\otimes \Lambda^{n}\langle \{de; e\in\cE' \}\rangle \ar@{->}[r]^-{\pa_{n}} & M\otimes \Lambda^{n-1}\langle \{de; e\in\cE' \}\rangle \ar@{->}[r]^-{\pa_{n-1}} & \cdots\ar@{->}[r]^-{\pa_{1}}& M,}
\end{equation}
where the chain map 
$\pa_{n} : M\otimes \Lambda^{n}\langle \{de; e\in\cE' \}\rangle
\to M\otimes \Lambda^{n-1}\langle \{de; e\in\cE' \}\rangle$
is given by 
\begin{equation}\label{eq:iD}
\pa_{n}(m\otimes de_{i_1}\wedge\cdots \wedge de_{i_n})
= \sum^n_{k=1}(-1)^{k-1}e_{i_{k}}m\otimes 
de_{i_1}\wedge\cdots \wedge \widehat{de_{i_{k}}}\wedge\cdots \wedge de_{i_n}
\end{equation}
for any $m \in M$.
\end{lem}

\begin{proof}
We denote by $\cL_D$ the Lie derivative associated to the Euler vector field. We then note that each one of the modules $\Omega^{n}_{\Sym_{\bQ}(\cE')|\bQ}$ for $n > 0$ and $\Ker(\aug:\Sym_{\bQ}(\cE')\to\bQ)$ 
is the direct sum of eigenspaces of $\cL_D$ with positive eigenvalues. Then the Cartan formula 
$\cL_D = dp_{D} + p_{D} d$ implies that the sequence \eqref{eq:differential_forms_resolution} of $\Sym_{\bQ}(\cE')$-modules is exact, whence the first result. The second result straightforwardly follows from the fact that $\Omega_{\Sym_{\bQ}(\cE')|\bQ}^{n} \cong \Sym_{\bQ}(\cE')\otimes \Lambda^{n}\langle \{de; e\in\cE' \}\rangle$ as $\Sym_{\bQ}(\cE')$-modules.
\end{proof}

\section{Stable cohomology with contravariant coefficients}\label{s:contravariant_coeff_syst_exterior}

In this section, we compute the stable twisted cohomology of the mapping class groups given by the exterior powers of $\tilde{H}^{\vee}(g)$, i.e.~the first \emph{cohomology} group of its unit tangent bundle.
Contrary to the case of the covariant coefficient system $\tilde{H}(g)$ studied in \S\ref{s:covariant_coeff_syst_exterior_general_theory}, these stable twisted cohomology groups are much more accessible; see Theorem~\ref{thm:stable_cohomology_contravariant_tildeH_ext}. First, we have the following starting tool:
\begin{prop}\label{prop:SES_tilde_H_vee}
For all $d\geq1$, there is a short exact sequence of $\Gamma_{g,1}$-modules: \begin{equation}\label{exteriorSES_cov}
\xymatrix{0\ar@{->}[r] & \Lambda^{d}H_{\bQ}(g)\ar@{->}[r]^-{\Lambda^{d}(\varpi^{*})} & \Lambda^{d}\tilde{H}^{\vee}_{\bQ}(g)\ar@{->}[r]^-{(\hbar\wedge-)^{\vee}} & \Lambda^{d-1}H_{\bQ}(g)\ar@{->}[r] & 0.}
\end{equation}
In particular, this $\Gamma_{g,1}$-module extension corresponds in $\Ext_{\ensuremath{\Gamma_{g,1}}}^{1}(\Lambda^{d-1}H_{\bQ}(g),\Lambda^{d}H_{\bQ}(g))$ to the class $k^{\vee}(g) \wedge \id_{\Lambda^{d-1}H_{\bQ}(g)}$ for $g\geq2d+5$.
\end{prop}
\begin{proof}
We recall that, for a finite-dimensional $\bQ$-module $M$, there is a canonical isomorphism $\Lambda^{d}(M^{\vee}) \cong (\Lambda^{d}M)^{\vee}$ for all $d\geq1$; see \cite[Chapter~III, \S11.5, Prop.~7]{Bourbakialgebra} for instance. Then, the short exact sequence \eqref{exteriorSES_cov} is obtained from the $\Gamma_{g,1}$-module short exact sequence \eqref{exteriorSES_contra} below, by applying the contravariant dualising (exact) functor $\bQ\Mod(-,\bQ)$.
That the class $k^{\vee}(g) \wedge \id_{\Lambda^{d-1}H_{\bQ}(g)}$ is the extension \eqref{exteriorSES_cov} follows from the fact that it is the Yoneda product of the class $k^{\vee}(g)\in H^{1}(\Gamma_{g,1};H_{\bQ}(g))$ of \eqref{SESchutb} with the trivial class $\id_{\Lambda^{d-1}H_{\bQ}(g)}\in \Ext^{0}(\Lambda^{d-1}H_{\bQ}(g);\Lambda^{d-1}H_{\bQ}(g))$.
\end{proof}
Let $\delta^{i}\colon H^{i}(\Gamma_{g,1};\Lambda^{d-1}H_{\bQ}(g))\to H^{i+1}(\Gamma_{g,1};\Lambda^{d}H_{\bQ}(g))$ be the $i^{\ith}$ connecting homomorphism of the cohomology long exact sequence associated with the extension \eqref{exteriorSES_cov}.
\begin{lem}\label{lem:connecting_hom_tilde_H_vee}
For $g\geq 2d+2i+5$, the morphism $\delta^{i}$ is equal to $-(m_{1,1} \wedge \id_{\Lambda^{d-1}H_{\bQ}(g)})\cup -$.
\end{lem}

\begin{proof}
Let $[z]$ be a cohomology class of $H^{i}(\Gamma_{g,1};\Lambda^{d-1}H_{\bQ}(g))$. We use the normalised cochain complex and generically denote by $\pa$ its differentials. We recall that $\hbar \in \tilde{H}_{\bQ}(g)$ denotes the homology class of the fiber of the unit tangent bundle, and that $\xi$ denotes the homotopy class of a framing of the tangent bundle $\mathrm{UT}\Sigma_{g,1}$.
We deduce from the action of $\Gamma_{g,1}$ on $\tilde{H}^{\vee}(g)$ (i.e.~the dual of the matrix \eqref{eq:matrix_action}) that the map $s_{\xi}\colon \bZ\hookrightarrow\tilde{H}^{\vee}_{\bQ}(g)$ defined by $u \mapsto u + k^{\vee}_{\xi}(g,-)$ defines a splitting (as an abelian group morphism) to the surjection $\tilde{H}^{\vee}_{\bQ}(g)\twoheadrightarrow\bZ$, $v \mapsto v(\hbar)$.
Hence the map $s_{\xi} \wedge \Lambda^{d-1}\varpi^{*}:\Lambda^{d-1}H_{\bQ}(g)\hookrightarrow\Lambda^{d}\tilde{H}^{\vee}_{\bQ}(g)$ defines a 
splitting (as an abelian group morphism) to the surjection $\Lambda^{d}\tilde{H}^{\vee}_{\bQ}(g)\twoheadrightarrow\Lambda^{d-1}H_{\bQ}(g)$. Since $z$ is a cocycle, we have $\pa z = 0$ and then compute that
\begin{align}
\begin{split}\label{eq:trivialliftchaincomplex_contravariant}
-(s_{\xi} \wedge \Lambda^{d-1}\varpi^{*})\varphi_{0}(z([\varphi_{1}\mid\cdots\mid\varphi_{i}])) & \;=\;\sum_{i=1}^{p}(-1)^{i}(s_{\xi} \wedge \Lambda^{d-1}\varpi^{*})(z([\varphi_{0}\mid\cdots\mid\varphi_{i}\varphi_{i+1}\mid\cdots\mid\varphi_{p}]))\\
 & \;\;\;+(-1)^{p+1}(s_{\xi} \wedge \Lambda^{d-1}\varpi^{*})(z([\varphi_{0}\mid\cdots\mid\varphi_{p-1}])).
\end{split}
\end{align}
We note that $\delta^{i}[z]=[\pa(s_{\xi}(z))]$ by the formal definition of 
$\delta^{i}$. The result follows from the equalities:
\begin{align*}\delta^{i}([z])([\varphi_{0}\mid\cdots\mid\varphi_{i}]) & \;=\;(\varphi_{0}(s_{\xi} \wedge \Lambda^{d-1}\varpi^{*})-(s_{\xi} \wedge \Lambda^{d-1}\varpi^{*})\varphi_{0}) z([\varphi_{1}\mid\cdots\mid\varphi_{i}])\\
 & \;=\;(k^{\vee}_{\xi}(g,-) \wedge \id_{\Lambda^{d-1}H_{\bQ}(g)})(\varphi_{0})\varphi_{0}z([\varphi_{1}\mid\cdots\mid\varphi_{i}])\\
 & \;=\;((-m_{1,1} \wedge \id_{\Lambda^{d-1}H_{\bQ}(g)})\cup z)([\varphi_{0}\mid\cdots\mid\varphi_{i}]),
\end{align*}
More precisely, the first equality is deduced from the computation \eqref{eq:trivialliftchaincomplex_contravariant}, the second equality follows from the fact that $k^{\vee}_{\xi}(g,\varphi)=\varphi\circ s_{\xi}\circ\varphi^{-1}-s_{\xi}$ since $k^{\vee}_{\xi}$ is the generating $1$-cocycle of the extension \eqref{SESchutb}, and the identification of $[k^{\vee}_{\xi}(g,-)]=-m_{1,1}$ (see Remark~\ref{rmk:identification_m_1,1}) gives the third equality.
\end{proof}

The cohomology long exact sequence applied to \eqref{exteriorSES_cov} provides a $4$-terms exact sequence, associated with the connecting homomorphism $-(m_{1,1} \wedge \id_{\Lambda^{d-1}H_{\bQ}(g)})\cup -$ from $H_{\st}^{i}(\Lambda^{d}H_\bQ(g))$ to $H_{\st}^{i+1}(\Lambda^{d-1}H_{\bQ}(g))$ for each $i\geq0$ such that $i + d = 1 \Modulo{2}$.
Let $\cQ_{d}(m_{1,1})$ be the subset of $\cQ_{d}$ (see Definition~\ref{def:weighted_partition}) of the weighted partitions $\{(i_{1},j_{1}),\ldots,(i_{\nu},j_{\nu})\}$ where $(i_{i},j_{i})=(1,1)$ for some $i\in \{1,\ldots,\nu\}$. The stable map $-(m_{1,1} \wedge \id_{\Lambda^{d-1}H_{\bQ}(g)})\cup -$ being injective, we deduce that:

\begin{thm}\label{thm:stable_cohomology_contravariant_tildeH_ext}
The stable twisted cohomology module $H_{\st}^{*}(\Lambda^{d}\tilde{H}^{\vee}_{\bQ})$ is isomorphic to the free $\Sym_{\bQ}(\cE)$-module on the twisted Mumford-Morita-Miller classes $\{m_{Q},Q\in\cQ_{d}-\cQ_{d}(m_{1,1})\}$.
In particular, the $\Sym_{\bQ}(\cE)$-module $H_{\st}^{*}(\Lambda^{d}\tilde{H}^{\vee}_{\bQ})$ is concentrated the degrees of the same parity as $d$.
The commutative bigraded stable cohomology algebra $H_{\st}^{*}(\Lambda^{*}\tilde{H}^{\vee}_{\bQ})$ is the polynomial algebra in the twisted Mumford-Morita-Miller classes $\{m_{i,j};i\geq0,j\geq1,(i,j)\neq(1,1)\}$ over the ring $\Sym_{\bQ}(\cE)$.
\end{thm}

\section{Stable cohomology in covariant coefficients: general theory}\label{s:covariant_coeff_syst_exterior_general_theory}

In this section, we study the stable twisted cohomology of the mapping class groups $\Gamma_{g,1}$ with twisted coefficient given by the exterior powers of $\tilde{H}_{\bQ}(g)$. We make the first steps for the computations of these stable twisted cohomology groups in \S\ref{ss:preliminary_computations}, in particular the determination of the connecting homomorphisms in cohomology for the short exact sequence of $\Gamma_{g,1}$-modules \eqref{exteriorSES_contra}; see Proposition~\ref{lem:connecting_hom_exterior_contra}.
We provide a first computation of the stable cohomology algebra with twisted coefficient in the graded exterior algebra over the localisation of $\Sym_{\bQ}(\cE)$ in \S\ref{ss:first_localisation}.
Finally, we introduce in \S\ref{ss:key_tools} the key tools for this work, in particular the determination of the connecting homomorphisms in cohomology for the short exact sequence of $\Gamma_{g,1}$-modules \eqref{exteriorSES_contra}; see Proposition~\ref{lem:connecting_hom_exterior_contra}.
We fix the following conventions and notations for the remainder of the paper.
\begin{convention}\label{conv:section_5}
From now on, we implicitly assume that $g\geq2i+2d+3$ each time we consider a cohomological degree $i$ for $H^{i}(\Gamma_{g,1};M(g))$ where $M(g)=\Lambda^{d}\tilde{H}_{\bQ}(g)$ or $\Lambda^{d}\tilde{H}^{\vee}_{\bQ}(g)$. This is for the cohomological stability bound of Theorem~\ref{thm:hom_stab_MCG} to be reached, so that we can consider the stable twisted cohomology graded modules $H_{\st}^{*}(\Lambda^{d}\tilde{H})$ and $H_{\st}^{*}(\Lambda^{d}\tilde{H}^{\vee})$.
\end{convention}

\subsection{Short exact sequences and derivations}\label{ss:preliminary_computations}

We recall that $\hbar \in \tilde{H}_{\bQ}(g)$ is the homology class of the fiber of the unit tangent bundle, which spans the kernel of the canonical map $\varpi_{*}: \tilde{H}_{\bQ}(g)\to H_{\bQ}(g)$ induced by the projection $\mathrm{UT}\Sigma_{g,1} \to \Sigma_{g,1}$.
The first key tool for studying the stable twisted cohomology module $H_{\st}^{*}(\Lambda^{d}\tilde{H}_{\bQ})$ is the following short exact sequence:
\begin{prop}\label{prop:ses_exterior_power}
For all $d\geq1$, there is a short exact sequence of $\Gamma_{g,1}$-modules: \begin{equation}\label{exteriorSES_contra}
\xymatrix{0\ar@{->}[r] & \Lambda^{d-1}H_{\bQ}(g)\ar@{->}[r]^-{\hbar\wedge-} & \Lambda^{d}\tilde{H}_{\bQ}(g)\ar@{->}[r]^-{\Lambda^{d}(\varpi_{*})} & \Lambda^{d}H_{\bQ}(g)\ar@{->}[r] & 0.}
\end{equation}
In particular, as an element of $\Ext_{\ensuremath{\Gamma_{g,1}}}^{1}(\Lambda^{d}H_{\bQ}(g),\Lambda^{d-1}H_{\bQ})(g))$, the $\Gamma_{g,1}$-module extension \eqref{exteriorSES_contra} corresponds to the cohomology class of $k(g)\wedge \id_{\Lambda^{d-1}H_{\bQ}(g)}$ for $g\geq2d+5$.
\end{prop}
\begin{proof}
The map $\varpi_{*}:\tilde{H}_{\bQ}(g) \to H_{\bQ}(g)$ defines the algebra homomorphism $\Lambda^{d}(\varpi_{*}): \Lambda^{d}\tilde{H}_{\bQ}(g) \to \Lambda^{d} H_{\bQ}(g)$, which is clearly $\Gamma_{g,1}$-equivariant. Let $\left\langle \hbar\right\rangle _{\Lambda^{d}\tilde{H}_{\bQ}(g)}$ be the two-sided ideal generated by $\hbar$ in the exterior algebra $\Lambda^{d}\tilde{H}_{\bQ}(g)$. By \cite[Chapter~III, \S7, Proposition 3]{Bourbakialgebra}, the morphism $\Lambda^{d}(\varpi_{*})$ is surjective and its kernel is isomorphic to $\left\langle \hbar\right\rangle _{\Lambda^{d}\tilde{H}_{\bQ}(g)}$. We define a $\Gamma_{g,1}$-morphism 
$$\Psi_{d}: \Lambda^{d-1}H_{\bQ}(g) \to \Ker(\Lambda^{d}(\varpi_{*}))$$
as follows. For $s\colon H_{\bQ}(g) \to \tilde{H}_{\bQ}(g)$ a splitting (as an abelian group morphism) to the map $\varpi_{*}$, we assign $\Psi_{d}(u) := \hbar \wedge s(u)$ for all $u \in \Lambda^{d-1}H_{\bQ}(g)$. If $s'\colon H_{\bQ}(g) \to \tilde{H}_{\bQ}(g)$ is another splitting of $\varpi_{*}$, then $s'(u)- s(u)$ 
belongs to $\left\langle \hbar\right\rangle _{\Lambda^{d-1}\tilde{H}_{\bQ}(g)}$, so there exists an element 
$v \in \Lambda^{d-1}\tilde{H}_{\bQ}(g)$ such that $s'(u)- s(u)= 
\hbar\wedge v$ and therefore $\hbar\wedge s'(u) -  \hbar\wedge s(u)= \hbar\wedge\hbar\wedge v = 0$. 
Hence, the morphism $\Psi_{d}$ is independent of the choice of $s$. In particular, this shows that it is by construction equivariant with respect to the canonical $\Gamma_{g,1}$-actions on $\Lambda^{d-1}\tilde{H}_{\bQ}(g)$ and $\Ker(\Lambda^{d}(\varpi_{*}))$. It is clearly injective and the surjectivity follows from the fact that any element of $\left\langle \hbar\right\rangle _{\Lambda^{d}\tilde{H}_{\bQ}(g)}$ may be written as $\tilde w\wedge\hbar$ where $w\in \Lambda^{d-1}\tilde{H}_{\bQ}(g)$ is such that $\Lambda^{d-1}(\varpi_{*})(w)\neq0$, which provides the short exact sequence \eqref{exteriorSES_contra}.

Finally, that the class $k(g)\wedge \id_{\Lambda^{d-1}H_{\bQ}(g)}$ is that of the extension \eqref{exteriorSES_cov} follows from the fact that it is the Yoneda product of the extension class $k(g)\in H^{1}(\Gamma_{g,1};H_{\bQ}(g))$ of \eqref{SEShutb} with the trivial class $\id_{\Lambda^{d-1}H_{\bQ}(g)}\in \Ext^{0}(\Lambda^{d-1}H_{\bQ}(g),\Lambda^{d-1}H_{\bQ}(g))$.
\end{proof}

Let $\xi: \mathrm{UT}\Sigma_{g,1} \to \bS^1$ be a framing of $\mathrm{UT}\Sigma_{g,1}$, we denote by $\xi_{*}: \tilde{H}_{\bQ}(g)\to \bQ = H_{1}(\bS^1;\bQ)$ the induced homomorphism in homology and we recall that $\Ker(\xi_{*})\cong H_{\bQ}(g)$ by the projection $\varpi_{*}$.
The framing $\xi$ induces a unique splitting (as an abelian group morphism) $s_{\xi}\colon H_{\bQ}(g) \to \tilde{H}_{\bQ}(g)$ to the map $\varpi_{*}$, such that the composite $\xi_{*}\circ s_{\xi}$ is the zero map $0\colon H_{\bQ}(g) \to \tilde{H}_{\bQ}(g)\to \bQ$.
We have two keys equalities on the interaction of the splitting $s_{\xi}$ with the contraction $\mu$ and the exterior product with the class $\hbar$:
\begin{lem}
For all $\varphi \in \Gamma_{g,1}$ and $v \in H_{\bQ}(g)$, we have:
\begin{equation}\label{eq:varphi_s}
(\varphi\circ s_{\xi}\circ\varphi^{-1}-s_{\xi})(v)=\mu(k_{\xi}(g,\varphi),v)\hbar.
\end{equation}
In particular, we deduce that
\begin{equation}\label{eq:varphi_s_theta}
\hbar \wedge\varphi s_{\xi}\varphi^{-1}(v) = \hbar\wedge s_{\xi}(v).
\end{equation}
\end{lem}

\begin{proof}
For any immersed loop $\alpha: \bS^1 \to \Sigma_{g,1}$, the homology class $[\overset\cdot\alpha/\Vert \overset\cdot\alpha\Vert]$ of the normalised velocity vector 
$\overset\cdot\alpha/\Vert \overset\cdot\alpha\Vert: \bS^1 
\to \mathrm{UT}\Sigma_{g,1}$ satisfies $[\overset\cdot\alpha/\Vert \overset\cdot\alpha\Vert] 
= s_{\xi} [\alpha] + (\rot_\xi(\alpha))\hbar \in \tilde{H}_{\bQ}(g)$.
Hence, we have $(s_{\xi} - s_{\varphi\cdot\xi} )[\alpha] = (\rot_{\varphi\cdot\xi}(\alpha) - \rot_\xi(\alpha))\hbar = ([\alpha]\cdot k_{\xi}(g,\varphi))\hbar$.
Then, since any homology class in $H_{1}(\Sigma_{g,1}; \bZ)$ is represented 
by an immersed loop, we have in $\tilde{H}_{\bQ}(g)$:
\begin{equation}\label{eq:varphi_s_step}
(s_{\xi}- s_{\varphi\cdot\xi})(v) = -\mu(k_{\xi}(g,\varphi),v)\hbar.
\end{equation}
Furthermore, we recall that
$\xi_{*}\circ\varphi^{-1}=(\varphi\cdot\xi)_{*}$, which implies that 
$\xi_{*}\circ\varphi^{-1}\circ s_{\varphi\cdot\xi} \circ \varphi = 0$ as a map $H_{\bQ}(g) \to \bQ$. Hence, we have $\varpi_{*}\circ \varphi^{-1}\circ s_{\varphi\cdot\xi} \circ \varphi = \id_{H_{\bQ}(g)}$, and a fortiori
$s_{\varphi\cdot\xi} = \varphi\circ s_{\xi} \circ \varphi^{-1}$.
Therefore, we deduce the equality \eqref{eq:varphi_s} from \eqref{eq:varphi_s_step}.
The equality \eqref{eq:varphi_s_theta} then follows from \eqref{eq:varphi_s} since $\hbar\wedge(\varphi s_{\xi} - s_{\xi} \varphi)(v)=0$. 
\end{proof}

For any $\varphi \in \Gamma_{g,1}$, we consider the (anti-)derivation
$\mu(k_{\xi}(g,\varphi), -)\wedge \id_{\Lambda^{d-1}H_{\bQ}(g)}: \Lambda^{d}H_{\bQ}(g) \to \Lambda^{d-1}H_{\bQ}(g)$ induced from the map $\mu(k_{\xi}(g,\varphi), -): H_{\bQ}(g)\to \bQ$, $v \mapsto 
\mu(k_{\xi}(g,\varphi), v)$.
\begin{prop}\label{prop:varphi_s_mu} For any $u \in \Lambda^{d}H$ and $\varphi \in \Gamma_{g,1}$, we have
$$
(\varphi(\Lambda^{d}s_{\xi})\varphi^{-1} - \Lambda^{d}s_{\xi})(u) 
= \Psi_{d}((\mu(k_{\xi}(g,\varphi), -)\wedge \id_{\Lambda^{d-1}H_{\bQ}(g)})(u)).
$$
\end{prop}
\begin{proof}
Using \eqref{eq:varphi_s}, for $v_1, v_2, \dots, v_d \in H_{\bQ}(g)$, we compute:
\begin{eqnarray*}
&& (\varphi(\Lambda^{d}s_{\xi})\varphi^{-1} - (\Lambda^{d}s_{\xi}))(v_1\wedge v_2
\wedge \cdots \wedge v_d)
\\
&=& \sum^{d}_{j=1}
s_{\xi} (v_1) \wedge\cdots\wedge s_{\xi} (v_{j-1})\wedge
(\varphi s_{\xi}\varphi^{-1} - s_{\xi})(v_j)\wedge \varphi s_{\xi}\varphi^{-1}(v_{j+1})\wedge\cdots 
\wedge\varphi s_{\xi}\varphi^{-1}(v_d)\\
&=& \sum^{d}_{j=1}
s_{\xi} (v_1) \wedge\cdots\wedge s_{\xi} (v_{j-1})\wedge
\mu(k_{\xi}(g,\varphi), v_j)\hbar\wedge \varphi s_{\xi}\varphi^{-1}(v_{j+1})\wedge\cdots 
\wedge\varphi s_{\xi}\varphi^{-1}(v_d)\\
&=& \sum^{d}_{j=1}
s_{\xi} (v_1) \wedge\cdots\wedge s_{\xi} (v_{j-1})\wedge
\mu(k_{\xi}(g,\varphi), v_j)\hbar\wedge  s_{\xi}(v_{j+1})\wedge\cdots 
\wedge s_{\xi}(v_d)\\
&=& \hbar\wedge\sum^{d}_{j=1}(-1)^{j-1} \mu(k_{\xi}(g,\varphi), v_j)
s_{\xi} (v_1) \wedge\cdots\wedge s_{\xi} (v_{j-1})\wedge
s_{\xi}(v_{j+1})\wedge\cdots 
\wedge s_{\xi}(v_d)\\
&=& \hbar\wedge s_{\xi}\left(\sum^{d}_{j=1}(-1)^{j-1} \mu(k_{\xi}(g,\varphi), v_j)
v_1\wedge\cdots\wedge v_{j-1}\wedge v_{j+1}\wedge\cdots\wedge v_d \right)\\
&=& \Psi_{d}((\mu(k_{\xi}(g,\varphi), -)\wedge \id_{\Lambda^{d-1}H_{\bQ}(g)})(v_1\wedge\cdots\wedge v_d))
\end{eqnarray*}
Here we used the formula \eqref{eq:varphi_s} to get the second equality and the formula \eqref{eq:varphi_s_theta} to get the third equality; the other equalities follow from the definitions or are clear algebraic manipulations.
\end{proof}

We now consider the cohomology long exact sequence associated with \eqref{exteriorSES_contra} and we denote by $\delta^{i}\colon H^{i}(\Gamma_{g,1};\Lambda^{d}H_{\bQ}(g))\to H^{i+1}(\Gamma_{g,1};\Lambda^{d-1}H_{\bQ}(g))$ its $i^{\ith}$ connecting homomorphism.
For each $i\geq0$ such that $i = d\Modulo{2}$ and $g\geq 2d+2i+5$, using the computation \eqref{eq:stablecohomologyKawazumi}, we obtain the following four term exact sequence:
\begin{equation}\label{ES_cohomology_exterior_general}
\xymatrix{
H^{i}(\Lambda^{d}\tilde{H}_{\bQ}(g))\ar@{^{(}->}[r] & H^{i}(\Lambda^{d}{H}_{\bQ}(g))\ar@{->}[r]^-{\delta^{i}} & H^{i+1}(\Lambda^{d-1}{H}_{\bQ}(g))\ar@{->>}[r] & H^{i+1}(\Lambda^{d}\tilde{H}_{\bQ}(g))}
\end{equation}
We denote by $\mu_{d,i}(m_{1,1},-)$ the map $H_{\st}^{i}(\Lambda^{d}{H}_{\bQ})\to H_{\st}^{i+1}(\Lambda^{d-1}{H}_{\bQ})$ defined by $v \mapsto \mu_{d}(m_{1,1}, v)$.

\begin{prop}\label{lem:connecting_hom_exterior_contra}
For $g\geq 2d+2i+5$, the morphism $\delta^{i}$ is equal to $\mu_{d,i}(m_{1,1}, -)$.
\end{prop}
\begin{proof}
Let $[z]$ be a cohomology class of $H^{i}(\Gamma_{g,1};\Lambda^{d}H_{\bQ}(g))$. We use the normalised cochain complex and generically denote by $\pa$ its differentials. For the fixed framing $\xi$ and associated splitting $s_{\xi}\colon H_{\bQ}(g) \to \tilde{H}_{\bQ}(g)$, we note that $\Lambda^{d}s_{\xi}\colon \Lambda^{d}H_{\bQ}(g)\to\Lambda^{d}\tilde{H}_{\bQ}(g)$ defines a splitting (as an abelian group morphism) the canonical surjection 
$\Lambda^{d}(\varpi_{*}):\Lambda^{d}\tilde{H}_{\bQ}(g)\twoheadrightarrow \Lambda^{d}H_{\bQ}(g)$.

Since $z$ is a cocycle, we have $\pa z = 0$ and we deduce that
\begin{align}
\begin{split}\label{eq:trivialliftchaincomplex_covariant}
-(\Lambda^{d}s_{\xi})\varphi_{0}(z([\varphi_{1}\mid\cdots\mid\varphi_{i}])) & \;=\;\sum_{j=1}^{i}(-1)^{j}(\Lambda^{d}s_{\xi})(z([\varphi_{0}\mid\cdots\mid\varphi_{j}\varphi_{j+1}\mid\cdots\mid\varphi_{i}]))\\
 & \;\;\;+(-1)^{i+1}(\Lambda^{d}s_{\xi})(z([\varphi_{0}\mid\cdots\mid\varphi_{i-1}])).
\end{split}
\end{align}
Since $\delta^{i}[z]=[\pa(s_{\xi}(z))]$ by the formal definition of $\delta^{i}$, it follows from \eqref{eq:trivialliftchaincomplex_covariant} and Proposition~\ref{prop:varphi_s_mu} that
\begin{eqnarray*}
\delta^{i}([z])([\varphi_{0}\mid\cdots\mid\varphi_{i}])&=&(\varphi_{0}(\Lambda^{d}s_{\xi}){\varphi_0}^{-1}-(\Lambda^{d}s_{\xi}))\varphi_{0}(z([\varphi_{1}\mid\cdots\mid\varphi_{i}]))\\
&=& \Psi_{d}((\mu(k_{\xi}(g,\varphi_0), -)\wedge \id_{\Lambda^{d-1}H_{\bQ}(g)})(\varphi_{0}(z([\varphi_{1}\mid\cdots\mid\varphi_{i}])))
\end{eqnarray*}
Since $[k_{\xi}(g,-)]$ is identified with $m_{1,1}$ in $H_{\st}^1(H_{\bQ})$ (see Remark~\ref{rmk:identification_m_1,1}), we have 
$\delta^{i}([z]) = \mu_{d,i}(m_{1,1}, [z])$, which ends the proof.
\end{proof}

\begin{rmk}
Based on the proof of Proposition~\ref{prop:ses_exterior_power}, 
the induced action of the map $\hbar\wedge-$ on the $i^{\ith}$ twisted cohomology group of the mapping class group is given as follows. For an $i$-cocycle $z \in Z^{i}(\Lambda^{d-1} H_{\bQ}(g))$, we take 
an $i$-cochain $\tilde{z} = (\Lambda^{d-1} s)\circ z \in C^{i}(\Lambda^{d-1} \tilde{H}_{\bQ}(g))$. 
Then $\pa\tilde{z}$ is an $(i+1)$-cochain with values in 
$\hbar\wedge\Lambda^{d-2}\tilde{H}_{\bQ}(g)$. Then we have $\pa(\tilde{z}\wedge\hbar) = \pa(\tilde{z})\wedge\hbar = 
0$ since $\hbar$ appears twice there. The cohomology class 
$[\tilde{z}\wedge\hbar]$ 
is equal to $(\hbar\wedge-)[z] \in H^{i}(\Lambda^{d}\tilde{H}_{\bQ}(g))$. We also note that $e_{1}\cup\hbar = 0$ since $H^{2}(\tilde{H}_{\bQ}(g))=0$ by Theorem~\ref{thm:main_Thm_1}. However, we have $\hbar\wedge e_{1}m_{0,2} \neq 0
\in H^2(\Lambda^3\tilde{H}_{\bQ}(g))$. Indeed, $H^2(\Lambda^3H_{\bQ}(g))$ is of dimension $3$ and spanned by the classes $m_{1,2}$, $e_{1}m_{0,2}$ and $m_{1,1}m_{1,1}$, while $H^1(\Lambda^3H_{\bQ}(g))$ is of dimension $2$ is spanned by $m_{0,3}$ and $m_{0,2}m_{1,1}$, and we have the following special case of the exact sequence
\eqref{ES_cohomology_exterior_general}
$$\xymatrix{H^{1}(\Lambda^{3}H_{\bQ}(g))\ar@{->}[rr]^-{\mu_{3,1}(m_{1,1},-)} &  & H^{2}(\Lambda^{2}H_{\bQ}(g))\ar@{->>}[r]^-{\hbar\wedge-} & H^{2}(\Lambda^{3}\tilde{H}_{\bQ}(g)).}$$
Since $\mu_{3,1}(m_{1,1}, m_{0,3}) = -3m_{1,2}$ and 
$\mu_{3,1}(m_{1,1}, m_{0,2}m_{1,1}) = -2m_{1,1}^2 - e_{1}m_{0,2}$, 
the element $e_{1}m_{0,2}$ is not in the image of $\mu_{3,1}(m_{1,1}, -)$.
This is not a contradiction. Let $\tilde w \in \Lambda^2\tilde{H}_{\bQ}(g)$ 
be a lift of $m_{0,2} \in H^0(\Lambda^2H_{\bQ}(g))$. 
Since $e_{1}\cup\hbar$ vanishes, it is represented by the 
coboundary of a $1$-cochain $c: \Gamma_{g,1} \to \tilde{H}_{\bQ}(g)$.
Therefore the class $\Psi_{d}^{*}(e_{1}m_{0,2})$ is represented by $(\pa c)\wedge\tilde{w}$, which is a cocycle but not a coboundary. Indeed we have $\pa ((\pa c)\wedge\tilde{w})=\pm(\pa c)\wedge(\pa \tilde{w}) = 0$ since $\hbar$ appears twice there. 
The non-vanishing $\hbar\wedge e_{1}m_{0,2} \neq 0$ just means that 
the image of $c$ is not included in $\bQ\hbar$.
\end{rmk}

We recall from Lemma~\ref{lem:canonical_sym_E_structure} that we have a canonical $\Sym_{\bQ}(\cE)$-module decomposition $H_{\st}^{*}(\Lambda^{d}\tilde{H}_{\bQ})
\cong H_{\st}^{\even}(\Lambda^{d}\tilde{H}_{\bQ})
\oplus H_{\st}^{\odd}(\Lambda^{d}\tilde{H}_{\bQ})$.
Therefore, we confine ourselves to studying the derivation
$$\ddd_{d} := \bigoplus_{i\geq0} \mu_{d,i}(m_{1,1}, -): H_{\st}^{*}(\Lambda^{d}H_{\bQ}) \to H_{\st}^{*}(\Lambda^{d-1}H_{\bQ}),$$
and the induced exact sequence of $\Sym_{\bQ}(\cE)$-modules
\begin{equation}\label{ES_cohomology_exterior}
\xymatrix{H_{\st}^{\dagger}(\Lambda^{d}\tilde{H}_{\bQ})\ar@{^{(}->}[r] & H_{\st}^{\dagger}(\Lambda^{d}{H}_{\bQ})\ar@{->}[r]^-{\ddd_{d}} & H_{\st}^{\ddagger}(\Lambda^{d-1}{H}_{\bQ})\ar@{->>}[r] & H_{\st}^{\ddagger}(\Lambda^{d}\tilde{H}_{\bQ}),}
\end{equation}
where $\dagger=``\odd"$ and $\ddagger=``\even"$ if $d$ is odd, while $\dagger=``\even"$ and $\ddagger=``\odd"$ if $d$ is even.
The key tool to compute the derivation $\ddd_{d}$ is provided by the contraction formulas of Proposition~\ref{prop:Contraction_Formula}.

\subsection{Computations over the localisation of $\Sym_{\bQ}(\cE)$}\label{ss:first_localisation}
In this section, we study the $\Sym_{\bQ}(\cE)$-algebra $H_{\st}^{*}(\Lambda^{*}\tilde{H}_{\bQ})$ rather each $\Sym_{\bQ}(\cE)$-module $H_{\st}^{*}(\Lambda^{d}\tilde{H}_{\bQ})$ individually. Also, we consider the localisation of $\Sym_{\bQ}(\cE)$ in order to simplify the computation. Namely, let $\cE^{\pm}$ be the set obtained from $\cE$ by inverting all the classical Mumford-Morita-Miller classes. Then the symmetric algebra $\Sym_{\bQ}(\cE^{\pm})$ is the extension $\bQ[e_{i}^{\pm 1};  i \geq 1]$ of $\Sym_{\bQ}(\cE)$. This induces the (exact) localisation functor $\Sym_{\bQ}(\cE^{\pm})\otimes_{\Sym_{\bQ}(\cE)}-$ and we compute the localised algebra $\Sym_{\bQ}(\cE^{\pm})\otimes_{\Sym_{\bQ}(\cE)}H_{\st}^{*}(\Lambda^{*}\tilde{H}_{\bQ})$; see Theorem~\ref{thm:full_stable_cohomology_extension}.
In particular, this $\Sym_{\bQ}(\cE^{\pm})$-algebra is free. This contrasts with the $\Sym_{\bQ}(\cE)$-module $H_{\st}^{*}(\Lambda^{*}\tilde{H}_{\bQ})$ which is not free for $d\geq3$ by Corollary~\ref{coro:stable_cohomo_algebra_not_free} and requires further complicated analysis to be computed for $*\leq 5$; see \S\ref{s:covariant_coeff_syst_exterior_computations}. Furthermore, the localisation provides a way to calculate the $\Sym_{\bQ}(\cE)$-algebra $H_{\st}^{*}(\Lambda^{*}\tilde{H}_{\bQ})$; see Corollary~\ref{coro:full_stable_cohomology}.

\subsubsection{A model for the algebra $H_{\st}^{*}(\Lambda^{*}H_{\bQ})$}
First of all, we introduce the following alternative convenient description of the stable cohomology algebra of the mapping class groups with twisted coefficient given by the exterior algebra $\Lambda^{*}H_{\bQ}$. For each $n \geq 1$, let $R_{n}$ be the polynomial algebra $\bQ[x_{n,0}, x_{n,1}, \dots, x_{n,n}]$, and $D_n$ be the following derivation on the algebra $R_{n}$
$$D_n := \sum^{n-1}_{k=0}x_{n,k}\dfrac{\pa}{\pa x_{n, k+1}}.$$
We stress that there are no Koszul signs in the derivation $D_n$, and then $D^{2}_n$ is not a derivation.
\begin{defn}
Let $x:=x_{n,0}^{i_{0}}x_{n,1}^{i_{1}}\cdots x_{n,n}^{i_{n}}$ be an element of $R_{n}$. We respectively define the \emph{weight} $\wt$ and the \emph{degree} $\deg$ of $x$ by 
$$\wt(x_{n,0}^{i_{0}}x_{n,1}^{i_{1}}\cdots x_{n,n}^{i_{n}}):=\sum_{j=0}^{n}ji_{j}\text{   and   }\deg(x_{n,0}^{i_{0}}x_{n,1}^{i_{1}}\cdots x_{n,n}^{i_{n}}):=\sum_{j=0}^{n}(2n-2-j)i_{j}.$$
\end{defn}

We deduce from these definitions that the derivation $D_n$ decreases the weight by $1$ and increases the degree by $1$. Hence the kernel $\Ker(D_n)$ and 
the cokernel $\Cok(D_n)$ are homogeneous both in weight 
and degree.
We also note that every non-constant homogeneous term in $R_{n}$ has at least degree $n-2$. 

\subsubsection{Key polynomials $\bar{x}_{n,k}$}
We now consider some particular key polynomials of $\Ker(D_{\infty})$.
\begin{defn}\label{def:bar_n_polynomials}
We fix $n\geq2$. First, we define $\bar{x}_{n,1} := x_{n,1}$. Now, for each $2 \leq k \leq n$, let $\bar{x}_{n,k}\in R_{n}$ be the polynomial defined as follows.
\begin{itemizeb}
    \item \textbf{If $k:=2\ell$:} we assign
    $$\bar{x}_{n, 2\ell} := x_{n,\ell}^2 + 2\sum^{\ell-1}_{i=0}(-1)^{\ell+i}x_{n,i}x_{n,2\ell-i}.$$
    For instance, $\bar{x}_{n,2}=x_{n,1}^2 + 2x_{n,0}^2$ and $\bar{x}_{n,4}=x_{n,2}^2 + 2x_{n,0}x_{n,4}-2x_{n,1}x_{n,3}$.
    The element $\bar{x}_{n, 2\ell}$ is homogeneous both in weight and degree with $\wt(\bar{x}_{n, 2\ell}) = 2\ell$ and $\deg (\bar{x}_{n, 2\ell}) = 4n-4-2\ell$.
    \item \textbf{If $k:=2\ell +1$:} we first define
    $$\hat x_{n, 2\ell+1} := x_{n,\ell}x_{n,\ell+1} + \sum^{\ell-1}_{i=0}(-1)^{\ell+i}(2\ell-2i+1)x_{n,i}x_{n,2\ell+1-i}.$$
    Then, we define $\bar{x}_{n, 2\ell+1} := x_{n,1}\bar{x}_{n,2\ell} - x_{n,0}\hat x_{n,2\ell+1}$. For instance, $\bar{x}_{n, 3} := x_{n,1}^3 + 2x_{n,0}^2 x_{n,1} - x_{n,0}x_{n,1}x_{n,2} - 3x_{n,0}^2 x_{n,3}$.
    The element $\bar{x}_{n, 2\ell+1}$ is homogeneous both in weight and degree with $\wt(\hat x_{n, 2\ell+1}) = 2\ell+1$ and $\deg (\hat x_{n, 2\ell+1}) = 6n-6-2\ell-1$.
\end{itemizeb}
\end{defn}
We note the following key property on these new variables:
\begin{lem}\label{lem:xbar_algebraically_independent}
For each $n\geq2$, the elements $x_{n,0}, \bar{x}_{n,2}, \dots, \bar{x}_{n,n}$ are algebraically independent over the algebra $R_{n}$.
\end{lem}
\begin{proof} 
We compute from Definition~\ref{def:bar_n_polynomials} that
$\bar{x}_{n,2\ell} \equiv (-1)^n2x_{n,0}x_{n,2\ell} \Modulo{R_{2\ell-1}}$, while
$\bar{x}_{n,2\ell+1} \equiv (-1)^{n+1}(2\ell+1)x_{n,0}^2x_{n,2\ell} \Modulo{R_{2\ell}}$. This implies that the set $\{\bar{x}_{n,2}, \dots, \bar{x}_{n,n}\}$ is algebraic independent over $R_{n}$.
Now, we consider a linear combination such that
\begin{equation}\label{eq:proof_algebraic_independence}
\sum_{(i_0,i_2,\cdots ,i_n)} \lambda_{(i_0,i_2,\cdots ,i_n)} x_{n,0}^{i_0}\bar{x}_{n,2}^{i_2}\cdots 
\bar{x}_{n,n}^{i_n} = 0    
\end{equation}
with $\lambda_{(i_0,i_2,\cdots ,i_n)} \in \bQ$. Because $x_{n,0}$ and all the $\bar{x}_{n,i}$'s are homogeneous both in weight and degree, we may decompose \eqref{eq:proof_algebraic_independence} into sums indexed by $i_{0}$ and assume that each of these summands is null. So there is no loss of generality in assuming that there is at most one $i\geq0$ such that $\lambda_{(i_{0},i_2,\cdots ,i_n)}= 0$ if $i_{0}\neq i$. Then, substituting $x_{n,0} = 1$ in \eqref{eq:proof_algebraic_independence}, we deduce that
$\sum_{(i_2,\cdots ,i_n)} \lambda_{(i,i_2,\cdots ,i_n)} (\bar{x}_{n,2}^{i_2}\cdots\bar{x}_{n,n}^{i_n})\vert_{x_0=1} = 0$ in $\bQ[x_{n,2}, \dots, x_{n,n}]$. We thus deduce from the algebraic independence proved above that $\lambda_{(i,i_2,\cdots ,i_n)}=0$, which thus end the proof.
\end{proof}

We also have the following key property on the new variables: 
\begin{lem}\label{lem:bar_n_kernel}
For all $n\geq2$ and $2 \leq k \leq n$, then $\bar{x}_{n, k}\in \Ker(D_{n})$.
\end{lem}
\begin{proof}
If $k=2\ell$, then we have
$$D_{n}(\bar{x}_{n, 2\ell})= 2x_{n,\ell-1}x_{n,\ell}+ 2\sum^{\ell-2}_{i=0}(-1)^{\ell+i+1}x_{n,i}x_{n,2\ell-i-1}
+ 2\sum^{\ell-1}_{i=0}(-1)^{\ell+i}x_{n,i}x_{n,2\ell-i-1} = 0.$$
If $k=2\ell +1$, we compute that
$$
\aligned
 D_{n}(\hat x_{n,2\ell+1})= & x_{n,\ell}^2 + x_{n,\ell-1}x_{n,\ell+1} - 3x_{n,\ell-1}x_{n,\ell+1}
+ 2 \sum^{n,\ell-2}_{i=0}(-1)^{\ell+i}x_{n,i}x_{n,2\ell-i}\\
= & x_{n,\ell}^2 + 2 \sum^{n,\ell-1}_{i=0}(-1)^{\ell+i}x_{n,i}x_{n,2\ell-i} = \bar{x}_{n,2\ell}.
\endaligned
$$
We then deduce that $D_{n}(\bar{x}_{n,2\ell+1})={x_{n,0}}\bar{x}_{n,2\ell} - {x_{n,0}}\bar{x}_{n,2\ell} = 0$.
\end{proof}

We consider the extension of $R_{n}$ defined by
$$R'_n := R_{n}[x^{-1}_{n, 0}] = \bQ[x^{\pm1}_{n,0}, x_{n, 1}, \dots, x_{n,n}].$$
Since $D_n(x_{n, 0}) = 0$, the derivation $D_n$ naturally extends to a derivation $D'_n$ on $R'_n$.
Since $x_{n,0}$ is invertible in $R'_n$, we deduce from Definition~\ref{def:bar_n_polynomials} that  we may inductively rewrite the generators $\{x_{n,1}, x_{n,2}, \ldots, x_{n,n}\}$ as polynomials over the ground ring $\bQ[x^{\pm1}_{n,0}]$ over the variables $\{\bar{x}_{n,1}, \bar{x}_{n,2}, \ldots, \bar{x}_{n,n}\}$. Therefore, we deduce from the linear independence result of Lemma~\ref{lem:xbar_algebraically_independent} that
$$R'_n \cong \bQ[x^{\pm1}_{n,0}][\bar{x}_{n,1}, \bar{x}_{n,2}, \ldots, \bar{x}_{n,n}].$$
In terms of these new variables, we compute that $D'_n = x_{n,0}(\pa/\pa\bar{x}_{n,1})$.
The kernel and cokernel of $D'_n$ are completely computed as follows:

\begin{lem}\label{lem:stable_cohomology_extension}
For each $n\geq2$, $\Ker(D'_n)$ is the $\bQ[x^{\pm1}_{n,0}]$-subalgebra of $R'_n$ generated by $\{\bar{x}_{n,k}\}_{k\geq 2}$, while $\Cok(D'_n)=0$.
\end{lem}
\begin{proof}
Since $x_{n,0}$ is invertible in $R'_n$, the result follows from elementary facts about the derivation $\pa/\pa\bar{x}_{n,1}$.
\end{proof}

\begin{rmk}
Lemma~\ref{lem:stable_cohomology_extension} cannot be true if we consider the derivation $D_n$ on $R_{n}$. For example, the composite $\bQ[x_{n,n}]\hookrightarrow R_{n} \twoheadrightarrow \Cok(D_n)$ is an injective map for any $n \geq 2$, and the substitution $x_{n,0}=x_{n,1}=\cdots= x_{n,n-1} = 0$ defines a left inverse map $\Cok(D_n) \to \bQ[x_{n,n}]$. Moreover, for $n\geq3$, the element $x_{n,0}^{-2}(\bar{x}_{n,3}^{2}-\bar{x}_{n,2}^{3})=6x_{n,1}^{3}x_{n,2}-3x_{n,1}^{2}x_{n,2}^{2}-18x_{n,0}x_{n,1}x_{n,2}x_{n,3}+9x_{n,0}^{2}x_{n,3}^{2}+8x_{n,0}x_{n,2}^{3}$ of $R_{n}$ belongs to $\Ker(D_{n})\subseteq R_{n}$ but not to $\bQ[\{\bar{x}_{n,k};k\geq2\}]$. 
\end{rmk}

\subsubsection{Correspondence between the algebras $H_{\st}^{*}(\Lambda^{*}H_{\bQ})$, $R_{\infty}$ and $R'_{\infty}$}

We consider the infinite tensor products
$$R_{\infty} := \bigotimes^\infty_{n=2}R_{n} \qquad\text{and}\qquad R'_{\infty} := \bigotimes^\infty_{n=2}R'_n.$$
We note that their respective homogeneous parts in each degree are of finite dimension. We recall that the description of the stable twisted cohomology module $H_{\st}^{*}(\Lambda^{d}H_{\bQ})$ for each $d\geq1$ is given in Theorem~\ref{stablecohomologyKawazumi}. The assignment for each $k\geq0$
$$\cR(x_{n,k}) := \mm_{n-k,k} = \frac{(-1)^k}{k!}m_{n-k,k} \in H^{2n-2-k}(\Lambda^k H_{\bQ})$$
induces an $\Sym_{\bQ}(\cE)$-algebra isomorphism $\cR:R_{\infty}\overset{\sim}{\to}H_{\st}^{*}(\Lambda^{*}H_{\bQ})$.
We note that if $u \in R_{\infty}$ is homogeneous both in the weight and the degree, then we have $\cR(u) \in H^{\deg (u)}(\Lambda^{\wt (u)}H_{\bQ})$.
Using the localisation functor $\Sym_{\bQ}(\cE^{\pm})\otimes_{\Sym_{\bQ}(\cE)}-$, the isomorphism $\cR$ extends to an isomorphism
$$\cR': R'_{\infty} \overset{\sim}\to \Sym_{\bQ}(\cE^{\pm})\otimes_{\Sym_{\bQ}(\cE)}H_{\st}^{*}(\Lambda^{*}H_{\bQ}).$$

Moreover, we extend the derivation $D_n$ (resp. $D'_n$) by the identity on each component other than $R_{n}$ (resp. $R'_n$). Then, we define derivations $D_{\infty}:=\sum^\infty_{n=2}D_n$ and  $D'_{\infty} = \sum^\infty_{n=2}x_{n,0}(\pa/\pa\bar{x}_{n,1})$ on the algebras $R_{\infty}$ and $R'_{\infty}$ respectively.
It follows from the definitions that
$$\cR \circ D_{\infty} = \ddd\circ \cR  \qquad\text{and}\qquad \cR' \circ D'_{\infty} = (\Sym_{\bQ}(\cE^{\pm})\otimes_{\Sym_{\bQ}(\cE)}\ddd)\circ \cR'$$
as morphisms $R_{\infty} \to H_{\st}^{*}(\Lambda^{*}H_{\bQ})$ and $R'_{\infty} \to \Sym_{\bQ}(\cE^{\pm})\otimes_{\Sym_{\bQ}(\cE)}H_{\st}^{*}(\Lambda^{*}H_{\bQ})$ respectively.
Then, we deduce from the universal properties of kernels and cokernels and from the exactness of the localisation functor that:
\begin{lem}\label{lem:isos_ker_coker_derivations}
The isomorphism $\cR$ induces $\Sym_{\bQ}(\cE)$-algebra isomorphisms $\Ker(\ddd) \cong \Ker(D_{\infty})$ and $\Cok(\ddd) \cong \Cok(D_{\infty})$. The isomorphism $\cR'$ induces $\Sym_{\bQ}(\cE^{\pm})$-algebra isomorphisms
$$\Sym_{\bQ}(\cE^{\pm})\otimes_{\Sym_{\bQ}(\cE)}\Ker(\ddd) \cong \Ker(D'_{\infty})\qquad\text{and}\qquad \Sym_{\bQ}(\cE^{\pm})\otimes_{\Sym_{\bQ}(\cE)}\Cok(\ddd) \cong \Cok(D'_{\infty}).$$
\end{lem}

\subsubsection{Computation of $H_{\st}^{*}(\Lambda^{*}\tilde{H}_{\bQ})$ over the localised algebra}

We are now able to compute $\Sym_{\bQ}(\cE^{\pm})\otimes_{\Sym_{\bQ}(\cE)}H_{\st}^{*}(\Lambda^{*}\tilde{H}_{\bQ})$ thanks to the algebra $R'_\infty$ and the derivation $D'_\infty$.

\begin{thm}\label{thm:full_stable_cohomology_extension}
The $\Sym_{\bQ}(\cE^{\pm})$-algebra $\Sym_{\bQ}(\cE^{\pm})\otimes_{\Sym_{\bQ}(\cE)}H_{\st}^{*}(\Lambda^{*}\tilde{H}_{\bQ})$ is free and generated by 
\begin{equation}\label{eq:generators_localised_stable_cohomology_exterior_tildeH}
\mathfrak{B}:=\left\{ e_{1} m_{n,1} - e_n m_{1,1},\, \cR'(\bar{x}_{n,2\ell}),\, \cR'(\bar{x}_{n,2\ell+1});\,\, n \geq 2,\, 1\leq \ell \leq n/2\right\}.
\end{equation}
In particular, if $d+i \equiv 1\Modulo{2}$, then $\Sym_{\bQ}(\cE^{\pm})\otimes_{\Sym_{\bQ}(\cE)}H^d(\Lambda^{i}\tilde{H}_{\bQ}) = 0$.
\end{thm}
\begin{proof}
By Lemma~\ref{lem:isos_ker_coker_derivations}, the result boils down to computing $\Ker(D'_{\infty})$ and $\Cok(D'_{\infty})$.
For $n \geq 2$, we define $y_{n+1} := x_{2,0}x_{n+1,1} - x_{n+1,0}x_{2,1}= x_{2,0}\bar{x}_{n+1,1} - x_{n+1,0}\bar{x}_{2,1}$ and $y_2 := x_{2,1} = \bar{x}_{2,1}$, and we compute that $\cR'(y_{n+1})=e_n m_{1,1} - e_{1} m_{n,1}$.
It is a routine to check analogously to Lemma~\ref{lem:xbar_algebraically_independent} from these assignments and Definition~\ref{def:bar_n_polynomials} that the generators $\{y_n\}^\infty_{n=2}\sqcup\{\bar{x}_{n,k}\}_{n\geq 2, 2 \leq k \leq n}$ are algebraically independent over $\Sym_{\bQ}(\cE^{\pm})$. Also, in terms of these variables, we have $D'_{\infty} = x_{2,0}(\pa/\pa y_2)$. We deduce that $D'_{\infty} (y_{n+1}) = 0$ and that $\Image(D'_{\infty})=R'_{\infty}$ since $x_{2,0}$ is invertible in $R'_{\infty}$. Hence $\Cok(D'_{\infty}) = 0$, and the subalgebra $\Ker(D'_{\infty})$ is a polynomial algebra over $\bQ[x^{\pm 1}_{n,0};  n\geq 2]$ in variables $\{y_n\}_{n\geq 3}\sqcup\{\bar{x}_{k,n}\}_{n\geq 2, 2 \leq k \leq n}$.
\end{proof}
\begin{rmk}
For all $n\geq2$, the generator $e_{1} m_{n,1} - e_{n} m_{1,1}$ is the class $M_{1,n}$ introduced in \cite{KawazumiSoulieI} and recalled in Theorem~\ref{thm:main_Thm_1} which generates the stable twisted cohomology module $H_{\st}^{*}(\tilde{H}_{\bQ})$.
Also, for all $1\leq \ell \leq n/2$, we compute that:
$$\cR'(\bar{x}_{n,2\ell}) = \mm^{2}_{n-\ell,\ell} + 2\sum^{\ell-1}_{k=0} \mm_{n-k,k}\mm_{n-2\ell+k, 2\ell-k}$$
and
\begin{align*}
\cR'(\bar{x}_{n,2\ell+1})  \;=\;& (\mm_{n-1,1}-e_{n-1}) \sum^{\ell-1}_{k=0}(-1)^{\ell+k}(2\ell-2k+1)\mm_{n-k,k}\mm_{n-2\ell+k+1, 2\ell+1-k}\\  & + (\mm_{n-1,1}-e_{n-1}) \mm_{n-\ell,\ell}\mm_{n-\ell-1,\ell+1}.
\end{align*}
\end{rmk}

As a result of Theorem~\ref{thm:full_stable_cohomology_extension}, we deduce the following property about the cokernel $\Cok(\ddd)$ which corresponds to the $\Sym_{\bQ}(\cE)$-module $H_{\st}^{\even}(\Lambda^{\odd}\tilde{H}_{\bQ})\oplus H_{\st}^{\even}(\Lambda^{\odd}\tilde{H}_{\bQ})$:
\begin{coro}\label{coro:cok_torsion}
The $\Sym_{\bQ}(\cE)$-module $H_{\st}^{\even}(\Lambda^{\odd}\tilde{H}_{\bQ})\oplus H_{\st}^{\even}(\Lambda^{\odd}\tilde{H}_{\bQ})$ is torsion.
\end{coro}
\begin{proof}
This follows from the fact that  $\Cok(\ddd)$ vanishes after applying the localisation functor $\Sym_{\bQ}(\cE^{\pm})\otimes_{\Sym_{\bQ}(\cE)}-$.
\end{proof}

Furthermore, Theorem~\ref{thm:full_stable_cohomology_extension} allows us to measure to what extent the $\Sym_{\bQ}(\cE)$-algebra $H_{\st}^{*}(\Lambda^{*}\tilde{H}_{\bQ})$ may be generated from the $\Sym_{\bQ}(\cE)$-module $H_{\st}^{*}(\tilde{H}_{\bQ})$ computed in \cite[Th.~A]{KawazumiSoulieI}; see Theorem~\ref{thm:main_Thm_1}. Namely, we define the free commutative algebra $\Sym_{\Sym_{\bQ}(\cE)}(H_{\st}^{*}(\tilde{H}_{\bQ}))$ from the $\Sym_{\bQ}(\cE)$-module structure of $H_{\st}^{*}(\tilde{H}_{\bQ})$. Because there is a canonical inclusion $H_{\st}^{*}(\tilde{H}_{\bQ})\hookrightarrow H_{\st}^{*}(\Lambda^{*}\tilde{H}_{\bQ})$, the universal property of $\Sym_{\Sym_{\bQ}(\cE)}(H_{\st}^{*}(\tilde{H}_{\bQ}))$ provides a unique canonical $\Sym_{\bQ}(\cE)$-algebra morphism:
\begin{equation}\label{eq:morphism_Sym_SymE}
\Upsilon\colon\Sym_{\Sym_{\bQ}(\cE)}(H_{\st}^{*}(\tilde{H}_{\bQ})) \longrightarrow H_{\st}^{*}(\Lambda^{*}\tilde{H}_{\bQ}).
\end{equation}

\begin{coro}\label{coro:Sym_SymE}
The $\Sym_{\bQ}(\cE)$-algebra morphism $\Upsilon$ is injective but not surjective.
\end{coro}
\begin{proof}
The localisation functor $\Sym_{\bQ}(\cE^{\pm})\otimes_{\Sym_{\bQ}(\cE)}-$ induces the following commutative square, where $\Upsilon':=\Sym_{\bQ}(\cE^{\pm})\otimes_{\Sym_{\bQ}(\cE)}$:
\begin{equation}\label{eq:diag_proof_Sym_SymE}
\xymatrix{\Sym_{\Sym_{\bQ}(\cE)}(H_{\st}^{*}(\tilde{H}_{\bQ}))\ar@{>}[rr]^-{\Upsilon}\ar@{^{(}->}[d] &  & H_{\st}^{*}(\Lambda^{*}\tilde{H}_{\bQ})\ar@{^{(}->}[d]\\
\Sym_{\Sym_{\bQ}(\cE^{\pm})}(H_{\st}^{*}(\tilde{H}_{\bQ}))\ar@{->}[rr]^-{\Upsilon'} &  & \Sym_{\bQ}(\cE^{\pm})\otimes_{\Sym_{\bQ}(\cE)}H_{\st}^{*}(\Lambda^{*}\tilde{H}_{\bQ}).}
\end{equation}
Following Theorem~\ref{thm:main_Thm_1}, we deduce from the relations in $H_{\st}^{*}(\tilde{H}_{\bQ})$ that, for all $ j>i \geq 1$, we have $M_{i,j}=e_{1}^{-1}(e_{i}M_{1,j}-e_{j}M_{1,i})$ in $\Sym_{\Sym_{\bQ}(\cE^{\pm})}(H_{\st}^{*}(\tilde{H}_{\bQ}))$. Hence $\Sym_{\Sym_{\bQ}(\cE^{\pm})}(H_{\st}^{*}(\tilde{H}_{\bQ}))$ is the free $\Sym_{\bQ}(\cE^{\pm})$-algebra generated by $\{ M_{1,n}; n \geq 2\}$.
Therefore, it follows from Theorem~\ref{thm:full_stable_cohomology_extension} and from the definition of $\Upsilon'$ the $\Sym_{\bQ}(\cE^{\pm})$-algebra morphism $\Upsilon'$ is injective, and so is $\Sym_{\bQ}(\cE)$-algebra morphism $\Upsilon$ by the commutativity of \eqref{eq:diag_proof_Sym_SymE}. Finally, the morphism $\Upsilon$ cannot be surjective, otherwise $\Upsilon'$ would be an isomorphism of $\Sym_{\bQ}(\cE^{\pm})$-algebras which would contradict Theorem~\ref{thm:full_stable_cohomology_extension} by the above computation of $\Sym_{\Sym_{\bQ}(\cE^{\pm})}(H_{\st}^{*}(\tilde{H}_{\bQ}))$.
\end{proof}

Finally, Theorem~\ref{thm:full_stable_cohomology_extension} leads us to ask the open question:
\begin{question}
Does there exist a twisted contravariant coefficient systems whose stable cohomology graded module tensored by $\Sym_{\bQ}(\cE^\pm)$ is not $\Sym_{\bQ}(\cE^\pm)$-free ?
\end{question}

\paragraph*{Computation of $H_{\st}^{\even}(\Lambda^{\even}\tilde{H}_{\bQ})\oplus H_{\st}^{\odd}(\Lambda^{\odd}\tilde{H}_{\bQ})$.}

We recall that the kernel $\Ker(\ddd)$ is isomorphic to the $\Sym_{\bQ}(\cE)$-module $H_{\st}^{\even}(\Lambda^{\even}\tilde{H}_{\bQ})\oplus H_{\st}^{\odd}(\Lambda^{\odd}\tilde{H}_{\bQ})$, that $\mathfrak{B}$ is the set of generators introduced in Theorem~\ref{thm:full_stable_cohomology_extension}, that $\mathfrak{M}$ is the basis of the twisted Mumford-Morita-Miller classes of Theorem~\ref{stablecohomologyKawazumi} and that $\Sym_{\bQ}(\cE^{\pm})\cS$ (resp. $\Sym_{\bQ}(\cE)\cS$) denotes the free $\Sym_{\bQ}(\cE^{\pm})$-module (resp. $\Sym_{\bQ}(\cE)$-module) on the set $\cS$.
Theorem~\ref{thm:full_stable_cohomology_extension} provides a canonical bigraded $\Sym_{\bQ}(\cE)$-module injection $\Ker(\ddd)\hookrightarrow \Sym_{\bQ}(\cE^\pm)\eqref{eq:generators_localised_stable_cohomology_exterior_tildeH}$ and thus a way to compute $H_{\st}^{\even}(\Lambda^{\even}\tilde{H}_{\bQ})\oplus H_{\st}^{\odd}(\Lambda^{\odd}\tilde{H}_{\bQ})$ as follows:
\begin{coro}\label{coro:full_stable_cohomology}
The bigraded $\Sym_{\bQ}(\cE)$-algebra $\Ker(\ddd)$ may be computed as the following pullback square:
\begin{equation}\label{eq:pullback}
\begin{tikzpicture}
[x=1mm,y=1mm]
\node (tl) at (0,12) {$\Ker(\ddd)$};
\node (tr) at (35,12) {$\Sym_{\bQ}(\cE^{\pm})\mathfrak{B}$};
\node (bl) at (0,0) {$\Sym_{\bQ}(\cE) \mathfrak{M} $};
\node (br) at (35,0) {$\Sym_{\bQ}(\cE^{\pm})\mathfrak{M} $};
\draw[right hook->] (tl) to (tr);
\draw[right hook->] (bl) to (br);
\draw[right hook->] (tl) to (bl);
\draw[right hook->] (tr) to (br);
\node at (4,8) {$\lrcorner$};
\end{tikzpicture}
\end{equation}
where the horizontal arrows are defined by the localisation functor $\Sym_{\bQ}(\cE^{\pm})\otimes_{\Sym_{\bQ}(\cE)}-$.
In particular, $\Sym_{\bQ}(\cE)\mathfrak{B}\subsetneq \Ker(\ddd)\subsetneq \Sym_{\bQ}(\cE^{\pm})\mathfrak{B}$.
\end{coro}

\subsection{Key tools to study stable twisted cohomology}\label{ss:key_tools}

We now introduce some new tools in order to make computations for the $\Sym_{\bQ}(\cE)$-modules $H_{\st}^{\dagger}(\Lambda^{d}\tilde{H}_{\bQ})$ and $H_{\st}^{\ddagger}(\Lambda^{d}\tilde{H}_{\bQ})$ when fixing $d\geq2$.

\subsubsection{A filtration by the monomials}\label{ss:filtration_monomials}
First, we define the following filtration of the module $H_{\st}^{*}(\Lambda^{d}H_{\bQ})$ for each $d\geq1$.

\begin{defn}\label{def:filtration_monomial}
For each $k \geq 0$, let $F_{k}H_{\st}^{*}(\Lambda^{d}{H}_{\bQ})$ be the free $\Sym_{\bQ}(\cE)$-submodule of $H_{\st}^{*}(\Lambda^{d}{H}_{\bQ})$ generated by the monomials in $\{\mm_{i,j}; i\geq0,j\geq1,i+j\geq2\}$ of length less than or equal to $k+1$.
In particular, $F_{d-1}H_{\st}^{*}(\Lambda^{d}{H}_{\bQ}) = H_{\st}^{*}(\Lambda^{d}{H}_{\bQ})$ and $F_{0}H_{\st}^{*}(\Lambda^{d}{H}_{\bQ})$ is generated by $\{m_{i,d}; i\geq0\}$ if $d \geq 2$ and $\{m_{i,1}; i\geq1\}$ if $d =1$.
\end{defn}

Since $\ddd_d(F_{k}H_{\st}^{*}(\Lambda^{d}{H}_{\bQ})) \subset F_{k}H_{\st}^{*}(\Lambda^{d-1}{H}_{\bQ})$ by formula \eqref{eq:m11mij}, the map $\ddd_d$ induces a morphism $\ddd_{d,k}\colon F_{k}H_{\st}^{*}(\Lambda^{d}{H}_{\bQ}) \to F_{k}H_{\st}^{*}(\Lambda^{d-1}{H}_{\bQ})$ for each $k\geq0$, and thus a morphism for each $k\geq1$:
\begin{equation}\label{eq:ddd_k_k-1}
\ddd_{d,k/(k-1)} \colon
F_{k}H_{\st}^{*}(\Lambda^{d}{H}_{\bQ})/F_{k-1}H_{\st}^{*}(\Lambda^{d}{H}_{\bQ})\longrightarrow F_{k}H_{\st}^{*}(\Lambda^{d-1}{H}_{\bQ})/F_{k-1}H_{\st}^{*}(\Lambda^{d-1}{H}_{\bQ}).
\end{equation}
By the snake lemma, we obtain an exact sequence for each $k\geq1$
\begin{equation}\label{eq:snake}
\xymatrix{\Ker(\ddd_{d,k/(k-1)})\ar@{->}[r]^-{\delta_{d,k-1}} & \Cok(\ddd_{d,k-1})\ar@{->}[r]^-{i_{d,k-1}} & \Cok(\ddd_{d,k})\ar@{->}[r] & \Cok(\ddd_{d,k/(k-1)})\ar@{->}[r] & 0.}
\end{equation}

\begin{prop}\label{prop:general_result_filtration}
We have the following properties for the filtration by the monomials.
\begin{enumerate}
\renewcommand{\labelenumi}{\textbf{\theenumi}}
\renewcommand{\theenumi}{(\arabic{enumi})}
    \item \label{(1)} The $\Sym_{\bQ}(\cE)$-modules $\Cok(\ddd_{d,k/(k-1)})$ 
and $\Ker(\ddd_{d,k/(k-1)})$ are free.
    \item \label{(2)} The $\Sym_{\bQ}(\cE)$-module $\Cok(\ddd_{d,0})$ is free, generated by $\mm_{0,d-1}$ if $d \geq 3$ and vanishes if $d=2$.
    \item \label{(3)} For integers $n\geq 1$ and $a,b, p,q\geq 0$ such that $a+p\geq 2$ and $b+q\geq 2$, we have
    $$\mm_{a,p}\mm_{b,q}=(-1)^{k}\mm_{a-k,p+k}\mm_{b+k,q-k}$$
    as elements of $\Cok(\ddd_{d,1})$, for $k\leq \mathrm{min}(a,q)$.
    \item \label{(4)} For $a_0\geq 1$, $a_i, p_0, p_i\geq 0$ with $a_0+p_0, a_i+p_i\geq 2$, we have 
\begin{eqnarray*}
&& \mm_{a_0,p_0}\mm_{a_1,p_1}\cdots \mm_{a_n,p_n}\\
&& = 
\begin{cases}
0 & \text{if $a_0 > \sum^n_{i=1}p_i$,}\\
(-1)^{a_0} \frac{a_0!}{p_1!\ldots p_n!}
e_{a_1+p_1-1}\cdots e_{a_n+p_n-1}m_{0, a_0+p_0} & \text{if $a_0 = \sum^n_{i=1}p_i$,}\\
(-1)^{a_0} \sum^n_{i=1}
\frac{a_0!}{p_1!\ldots (p_i-1)! \ldots p_n!}
e_{a_1+p_1-1}\cdots  m_{a_i+p_i-1,1}\cdots e_{a_n+p_n-1}m_{0, a_0+p_0} & \text{if $a_0 = -1+\sum^n_{i=1}p_i$,}
\end{cases}
\end{eqnarray*}
as an element of $\Cok(\ddd_{d,n})$.
\item \label{(5)} If some integers $k_i\geq 1$ satisfy 
$k_1+\cdots+ k_r \geq d-r$, then we have $e_{k_1}\cdots e_{k_r}u = 0$
as an element of $\Cok(\ddd_{d,l+r})$, for any $u \in F_{l}H_{\st}^{*}(\Lambda^{d-1}H_{\bQ})$. In particular, if $k \geq d-1$, we have $e_{k}u = 0 
\in \Cok(\ddd_d)$ for any $u \in H_{\st}^{*}(\Lambda^{d-1}H_{\bQ})$.
\end{enumerate}
\end{prop}
\begin{proof}
The points \ref{(1)} and \ref{(2)} are straightforward consequences from the formulas of Proposition~\ref{prop:Contraction_Formula} and of the definition of the filtration.
The point \ref{(3)} follows from a clear recursion on the relation
$$\ddd_d(\mm_{a-1,p+1}\mm_{b,q})=\mm_{a,p}\mm_{b,q}+\mm_{a-1,p+1}\mm_{b+1,q-1}$$
computed from the contraction formula \eqref{eq:m11mij_1} of Proposition~\ref{prop:Contraction_Formula}. For the point \ref{(4)}, we denote $\ddd_{d+1}\circ\cdots \circ \ddd_{d+k-1}$ by $\ddd^{k-1}$ for simplicity. Since $a_0 \geq 1$, we compute from Proposition~\ref{prop:Contraction_Formula} that 
$$\aligned
& \ddd_d\left(\sum^{a_0}_{k=1}(-1)^k\mm_{a_0-k,p_0+k}
\ddd^{k-1}(\mm_{a_1,p_1}\ldots \mm_{a_n,p_n})\right)\\
= & \sum^{a_0}_{k=1}(-1)^k\mm_{a_0-k+1,p_0+k-1}\ddd^{k-1}(\mm_{a_1,p_1}\ldots \mm_{a_n,p_n})
+ \sum^{a_0}_{k=1}(-1)^k\mm_{a_0-k,p_0+k} \ddd^{k}(\mm_{a_1,p_1}\ldots \mm_{a_n,p_n})\\
= & -\mm_{a_0,p_0}\mm_{a_1,p_1}\ldots \mm_{a_n,p_n}
+ (-1)^{a_0}\mm_{0, p_0+a_0}\ddd^{a_0}(\mm_{a_1,p_1}\ldots \mm_{a_n,p_n}).
\endaligned$$
Hence, as an element of $\Cok(\ddd_{d,n})$, we have 
\begin{equation}\label{eq:key_rel_general_result_filtration}
\mm_{a_0,p_0}\mm_{a_1,p_1}\ldots \mm_{a_n,p_n} = 
(-1)^{a_0}\mm_{0, p_0+a_0}\ddd^{a_0}(\mm_{a_1,p_1}\ldots \mm_{a_n,p_n}).
\end{equation}
Here we have 
$$\ddd^{k}(\mm_{a_i,p_i}) = 
\begin{cases} 
0 & \text{if $k > p_i$,}\\
e_{a_i+p_i-1} & \text{if $k = p_i$,}\\
\mm_{a_i+p_i-1,1} & \text{if $k = p_i-1$,}
\end{cases}$$
and then
\begin{eqnarray*}
&&\ddd^{k}(\mm_{a_1,p_1}\cdots \mm_{a_n,p_n}) \\
&& = 
\begin{cases} 
0 & \text{if $k > \sum p_i$,}\\
\frac{k!}{p_1!\ldots p_n!}
e_{a_1+p_1-1}\cdots e_{a_n+p_n-1} & \text{if $k = \sum p_i$,}\\
\sum^n_{i=1}
\frac{k!}{p_1!\ldots (p_i-1)! \ldots p_n!}
e_{a_1+p_1-1}\cdots  m_{a_i+p_i-1,1}\cdots e_{a_n+p_n-1} & \text{if $k = -1+\sum p_i$.}
\end{cases}
\end{eqnarray*}
For the point \ref{(5)}, we first note that $\ddd^{d-1}u \in \Sym_{\bQ}(\cE)$, hence $\ddd^d u = 0$. Since $(k_1+1)+\cdots +(k_r+1) \geq d$, there exists 
$d_i \geq 0$ such that $d_1+\cdots +d_r = d$ and $k_i+1 \geq d_i$ for each $1 \leq i \leq r$. 
We consider $v := \tfrac{{d_1}!\cdots {d_r}!}{d!}
\mm_{k_1-d_1+1,d_1}\cdots\mm_{k_r-d_r+1,d_r}\in F_{r-1}H_{\st}^{*}(\Lambda^{d}H_{\bQ})$. Then we have $\ddd^d (v) = e_{k_1}\cdots e_{k_r}$ and $\ddd^i (u)\ddd^j (v) \in F_{r-1}H_{\st}^{*}(\Lambda^{d}H_{\bQ})$ for all $i+j=d-1$. Hence
$$e_{k_1}\cdots e_{k_r}u = \ddd^{d}(v)u = (-1)^{d-1}\ddd_d(v)\ddd^{d-1}(u) = \ddd_d((-1)^{d-1}v\ddd^{d-1}(u)) = 0 \in \Cok(\ddd_{d,l+r})$$ 
which ends the proof.
\end{proof}

These tools also allow us to exhibit the following qualitative property for the stable cohomology graded modules for the exterior powers of degree greater or equal to three: 
\begin{thm}\label{thm:general_qualitative_thm}
Let us fix $d \geq 3$. Then $m_{0, d-1}\in \Cok(\ddd_{d})$ is a non-trivial torsion element.
\end{thm}

\begin{proof}
First, we recall from \ref{(2)} in Proposition~\ref{prop:general_result_filtration} that $m_{0, d-1}\in \Cok(\ddd_{d,0})$ is non-trivial. We consider the exact sequence \eqref{eq:snake} for each $k\geq1$. By definition, $\Sym_{\bQ}(\cE)$-modules $\Ker(\ddd_{d,k/(k-1)})$ is free on a linear combination of products of $k+1$ twisted Mumford-Morita-Miller classes $\{m_{i,j};i\geq0,j\geq1\}$. Then, using the formal definition from the snake lemma of a connecting homomorphism, we deduce that the image $\delta_{d,k-1}$ is generated by linear combinations of products of $k$ classes $\{m_{i,j};i\geq0,j\geq1\}$, each product being multiplied by at least one classical Mumford-Morita-Miller class $\{e_{i};i\geq1\}$. Hence the class $m_{0, d-1}$ does not belong to the image of any $\delta_{d,k-1}$ for all $1\leq k\leq d-2$. Therefore $m_{0, d-1}$ is a non-trivial class in each $\Cok(\ddd_{d,k})$ for all $0\leq k\leq d-1$, and in particular on $\Cok(\ddd_{d,d-1})=\Cok(\ddd_{d})$. That $m_{0, d-1}$ is a torsion element follows from \ref{(5)} in Proposition~\ref{prop:general_result_filtration}, whence the result.
\end{proof}

\begin{coro}\label{coro:stable_cohomo_algebra_not_free}
If $d \geq 3$, the $\Sym_{\bQ}(\cE)$-module $H_{\st}^{\ddagger}(\Lambda^{d}\tilde{H}_{\bQ})$ is a non-trivial and torsion.
\end{coro}

\begin{proof}
We already know from Corollary~\ref{coro:cok_torsion} that the $\Sym_{\bQ}(\cE)$-module $H_{\st}^{\ddagger}(\Lambda^{d}\tilde{H}_{\bQ})=\Cok(\ddd_{d})$ is torsion. Now it follows from Theorem~\ref{thm:general_qualitative_thm} that $H_{\st}^{\ddagger}(\Lambda^{d}\tilde{H}_{\bQ})$ has a non-trivial torsion $\Sym_{\bQ}(\cE)$-module summand containing the class $m_{0,d-1}$.
\end{proof}

Finally, we have the following estimate for annihilators of each element in
$H_{\st}^{\even}(\Lambda^{\odd}\tilde{H}_{\bQ})\oplus H_{\st}^{\odd}(\Lambda^{\even}\tilde{H}_{\bQ})$, 
which is a supporting evidence for our computations in the following sections. This follows immediately from \ref{(5)} in Proposition~\ref{prop:general_result_filtration}.
\begin{coro}\label{coro:annihilator}
For any element $u \in H_{\st}^{\even}(\Lambda^{\odd}\tilde{H}_{\bQ})\oplus H_{\st}^{\odd}(\Lambda^{\even}\tilde{H}_{\bQ})$, there exists
a {\bf monomial} $v$ in $\cE$ such that $vu = 0$ in $H_{\st}^{\even}(\Lambda^{\odd}\tilde{H}_{\bQ})\oplus H_{\st}^{\odd}(\Lambda^{\even}\tilde{H}_{\bQ})$.
\end{coro}

\subsubsection{First application: low-dimensional paired stable twisted cohomologies}\label{ss:low_dim_paired_cohomologies}

We now apply the filtration of \S\ref{ss:filtration_monomials} to prove the triviality of each low-dimensional stable cohomology group $H_{\st}^{i}(\Lambda^{d}\tilde{H}_{\bQ})$ when $i$ has the same parity as $d$. Beforehand, we introduce the following truncation of the derivation $\ddd_{d}$ and its filtration by the monomials.
We consider an integer $j\geq2$ and denote by $\ddd_{d}^{j}$ the direct sum
$$\bigoplus_{0\leq i\leq j-2} \mu_{d,i}(m_{1,1}, -): \bigoplus_{0\leq i\leq j-2} H_{\st}^{i}(\Lambda^{d}H_{\bQ}) \to \bigoplus_{0\leq i\leq j-2} H_{\st}^{i+1}(\Lambda^{d-1}H_{\bQ}).$$    
The filtration by the monomials of Definition~\ref{def:filtration_monomial} restricts well to $H_{\st}^{*\leq d-2}(\Lambda^{d}H_{\bQ})$ (i.e.~the direct sum $\bigoplus_{0\leq i\leq d-2} H_{\st}^{i}(\Lambda^{d}H_{\bQ})$). Then, by formula \eqref{eq:m11mij}, the map $\ddd_{d,k}^{d}$ induces a morphism $\ddd_{d,k}^{d}\colon F_{k}H_{\st}^{*\leq d-2}(\Lambda^{d}{H}_{\bQ}) \to F_{k}H_{\st}^{*\leq d-1}(\Lambda^{d-1}{H}_{\bQ})$ for each $k\geq 0$, and thus a morphism $\ddd_{d,k/(k-1)}^{d}$ analogous to \eqref{eq:ddd_k_k-1} for each $k\geq1$.

\begin{thm}\label{thm:ker_D_d_d}
We assume that $d\geq3$. Then $\Ker(\ddd_{d}^{d})=0$.
\end{thm}
\begin{proof}
Let us prove by induction on $0\leq k\leq d-2$ that $\Ker(\ddd_{d,k}^{d})=0$, and we will be done since $\ddd_{d,d-2}^{d}=\ddd_{d}^{d}$ because $F_{k}H_{\st}^{*\leq d-2}(\Lambda^{d}{H}_{\bQ})=F_{d-2}H_{\st}^{*\leq d-2}(\Lambda^{d}{H}_{\bQ})$) if $k\geq d-2$.
Indeed, a monomial $\mm_{a_{0},b_{0}}\mm_{a_{1},b_{1}}\cdots \mm_{a_{k},b_{k}} \in H_{\st}^{*\leq d-2}(\Lambda^{d}{H}_{\bQ})$ with $b_i \geq 1$ satisfies $\sum_{0\leq i\leq k} b_i = d$ and 
$\sum_{0\leq i\leq k} (2a_i + b_i - 2) \leq d-2$. A fortiori, we have $\sum_{0\leq i\leq k} a_i \leq k$, so there is some $i_0$ such that $a_{i_0} = 0$ and $b_{i_0} \geq 2$ (since there does not exist $m_{0,1}$), and thus $d = \sum_{0\leq i\leq k} b_i \geq k+2$.

By Definitions~\ref{def:mij} and \ref{def:weighted_partition}, we note that a generator $\mm_{a_{0},b_{0}}$ belongs to $F_{0}H_{\st}^{*\leq d-2}(\Lambda^{d}{H}_{\bQ})$ if and only if $b_{0}=d$ and $a_{0}=0$. So, by the formula \eqref{eq:m11mij}, we deduce that $\ddd_{d,0}^{d}$ is injective. Now, let us assume that $\Ker(\ddd_{d,k-1}^{d})=0$ for some $k\geq1$. By the snake lemma, there is an exact sequence for each $k\geq1$
\begin{equation}\label{eq:snake_ker}
\xymatrix{0\ar@{->}[r] & \Ker(\ddd_{d,k-1}^{d})\ar@{->}[r] & \Ker(\ddd_{d,k}^{d})\ar@{->}[r] & \Ker(\ddd_{d,k/(k-1)}^{d})\ar@{->}[r]^-{\delta_{d,k-1}^{d}} & \Cok(\ddd_{d,k-1}^{d}).}
\end{equation}
Hence, it is enough to prove that $\Ker(\ddd_{d,k/(k-1)}^{d})=0$ for $0\leq k\leq d-2$, because then we are done by using \eqref{eq:snake_ker} and the inductive assumption on $k$.

Beforehand, we note from Definitions~\ref{def:mij} and \ref{def:weighted_partition} that a generator $\mm_{a_{0},b_{0}}\mm_{a_{1},b_{1}}\cdots \mm_{a_{k},b_{k}}$ belongs to (and does not vanish in) $F_{k}H_{\st}^{*\leq d-2}(\Lambda^{d}{H}_{\bQ})/F_{k-1}H_{\st}^{*\leq d-2}(\Lambda^{d}{H}_{\bQ})$ if and only if $\sum_{0\leq i\leq k}a_{i}\leq k$ and $\sum_{0\leq i\leq k}b_{i}= d$ with $b_{i}\geq1$ for each $0\leq i\leq k$. In particular, there must then be at least one $a_{i}=0$.
We also recall from Definition~\ref{def:weighted_partition} that $b_{0}\geq b_{1}\geq \cdots \geq b_{k}$ and that $a_{i}\geq a_{i+1}$ if $b_{i} = b_{i+1}$.
An element of the kernel of $\ddd_{d,k/(k-1)}^{d}$ may be written as a formal sum $\mathsf{M}_{\ell}:=\sum_{1\leq j\leq \ell}\lambda_{j}\mm_{a_{0},b_{0}}^{(j)}\cdots \mm_{a_{k},b_{k}}^{(j)}$, where $\ell\geq1$, $\lambda_{j}\in \Sym_{\bQ}(\cE)$ and $\mm_{a_{0},b_{0}}^{(j)}\cdots \mm_{a_{k},b_{k}}^{(j)}$ belongs to (and does not vanish in) $F_{k}H_{\st}^{*\leq d-2}(\Lambda^{d}{H}_{\bQ})/F_{k-1}H_{\st}^{*\leq d-2}(\Lambda^{d}{H}_{\bQ})$ for each $1\leq j\leq \ell$. For convenience, we denote by $a_{i}^{(j)}$ and $b_{i}^{(j)}$ the indices of the $\mm_{a_{i},b_{i}}^{(j)}$.

Let us prove by induction on $\ell\geq1$ that any $\mathsf{M}_{\ell}$ in the kernel of $\ddd_{d,k/(k-1)}^{d}$ vanishes if its number of terms is not greater than $\ell$. Since $k\leq d-2$, note that $b_{0}^{(1)} \geq 2$ because $\sum_{0\leq i\leq k}b_{i}^{(1)}= d$, while $b_{0}^{(1)}\geq b_{1}^{(1)}\geq \cdots \geq b_{k}^{(1)}$. Hence $\sum_{0\leq i\leq k}\mm_{a_{0},b_{0}}^{(1)}\cdots  \mm_{a_{i}+1,b_{i}-1}^{(1)}\cdots \mm_{a_{k},b_{k}}^{(1)}$ does not vanish in $F_{k}H_{\st}^{*\leq d-1}(\Lambda^{d-1}{H}_{\bQ})/F_{k-1}H_{\st}^{*\leq d-1}(\Lambda^{d-1}{H}_{\bQ})$, and so $\ddd_{d,k/(k-1)}^{d}(\mathsf{M}_{1})=0$ implies that $\lambda_{1}=0$, and thus $\mathsf{M}_{1}=0$.

Let us now assume $\ell\geq 2$ and that
any element of the kernel of $\ddd_{d,k/(k-1)}^{d}$ vanishes if its number of terms is strictly smaller than $\ell$.
We denote by $L$ the set $\{1,\ldots,\ell\}$ and by $L_{c}$ the subset $\{l\in L \mid c^{(l)}\geq c^{(l')}, \forall l'\in L\}$ for $c\in\{a_{0},b_{0}\}$. The rest of the proofs is the following case disjunction.

\emph{\textbf{\underline{Step $1$:}} assume that $L_{b_{0}}\neq L$.} We denote the summand $\sum_{j\in L_{b_{0}}}\lambda_{j} \mm_{a_{0},b_{0}}^{(j)}\cdots \mm_{a_{k},b_{k}}^{(j)}$ of $\mathsf{M}_{\ell}$ by $\mathsf{M}_{L_{b_{0}}}$, and the other summand $\sum_{j'\in L\setminus L_{b_{0}}}\lambda_{j'}\mm_{a_{0},b_{0}}^{(j')}\cdots \mm_{a_{k},b_{k}}^{(j')}$ by $\mathsf{M}_{L\setminus L_{b_{0}}}$. Decomposing the equality $\ddd_{d,k/(k-1)}^{d}(\mathsf{M}_{\ell})=0$ with respect to these summands provides that:
\begin{equation}\label{eq:step_1_equality}
\sum_{j\in L_{b_{0}}}\lambda_{j}\sum_{0\leq i\leq k}\mm_{a_{0},b_{0}}^{(j)}\cdots  \mm_{a_{i}+1,b_{i}-1}^{(j)}\cdots \mm_{a_{k},b_{k}}^{(j)} = 
-\sum_{j'\in L\setminus L_{b_{0}}}\lambda_{j'}\sum_{0\leq i\leq k}\mm_{a_{0},b_{0}}^{(j')}\cdots  \mm_{a_{i}+1,b_{i}-1}^{(j')}\cdots \mm_{a_{k},b_{k}}^{(j')}.
\end{equation}

\emph{\underline{Case $A$:} for all $j\in L_{b_{0}}$ and all $1\leq i\leq k$, $a_{i}^{(j)}=0$ (and thus $b_{i}^{(j)}\geq2$ by Definition~\ref{def:mij}).}
For each $j\in L_{b_{0}}$, let $I_{j}$ be the maximum $i\in\{0,\ldots, k\}$ such that $b_{i}^{(j)}= b_{0}^{(j)}$. We denote by $I$ the maximum of the set $\{I_{j}\mid j\in L_{b_{0}} \}$ and by $L_{b_{0}}(I)$ the set $\{j\in L_{b_{0}}\mid I_{j}=I \}$. Then, by analysing the homogeneity of the $b_{0}$'s, we deduce that the following summand of the left-hand side of \eqref{eq:step_1_equality} for each $j\in L_{b_{0}}(I)$
$$\lambda_{j}\sum_{I+1\leq i\leq k}\mm_{0,b_{0}}^{(j)}\cdots  \mm_{1,b_{i}-1}^{(j)}\cdots \mm_{0,b_{k}}^{(j)}$$
cannot be cancelled by linear combinations of terms
\begin{itemizeb}
    \item of the left-hand side of \eqref{eq:step_1_equality} either indexed by $j_{1}\in L_{b_{0}}\setminus L_{b_{0}}(I)$ or those of the form $\lambda_{j_{2}}\sum_{0\leq i\leq I}\mm_{0,b_{0}}^{(j_{2})}\cdots  \mm_{1,b_{i}-1}^{(j_{2})}\cdots \mm_{0,b_{k}}^{(j_{2})}$ with $j_{2}\in L_{b_{0}}(I)$;
    \item of the right-hand side of \eqref{eq:step_1_equality}.
\end{itemizeb}
Therefore, denoting by $\mm_{0,b_{0}}$ the common class $\mm_{0,b_{0}}^{(j)}$ when $j\in L_{b_{0}}(I)$, we deduce that 
$$\mm_{0,b_{0}}^{\cup I}\sum_{j\in L_{b_{0}}(I)}\lambda_{j}\sum_{I+1\leq i\leq k}\mm_{0,b_{I+1}}^{(j)}\cdots  \mm_{1,b_{i}-1}^{(j)}\cdots \mm_{0,b_{k}}^{(j)} = 0.$$
Note that this is equal to $\mm_{0,b_{0}}^{\cup I}\ddd_{d,k-I-1/(k-I-2)}^{d}(\mathsf{M}_{L_{b_{0}},I})$, where $\mathsf{M}_{L_{b_{0}},I}$ denotes the element $\sum_{j\in L_{b_{0}}(I)}\lambda_{j}\mm_{0,b_{I+1}}^{(j)}\cdots \mm_{0,b_{k}}^{(j)}$. Then it follows from the inductive assumption on $k$ that $\mathsf{M}_{L_{b_{0}},I}=0$. Hence there is at least one coefficient $\lambda_{j}$ of $\mathsf{M}_{\ell}$ which is null and thus we are done by the inductive assumption on $\ell$.

\emph{\underline{Case $B$:} there exists $j\in L_{b_{0}}$ and $1\leq i\leq k$ such that $b_{i}^{(j)}\geq2$.}
By Case~$A$, we may assume that there exists $j\in L_{b_{0}}$ and $1\leq i\leq k$ such that $a_{i}^{(j)}\neq 0$.

Let $L_{b_{0},a_{0}}$ be the set $\{l\in L_{b_{0}} \mid a_{0}^{(l)}\geq a_{0}^{(l')}, \forall l'\in L_{b_{0}}\}$. For each $j\in L_{b_{0},a_{0}}$, let $I_{j}$ be the maximum $i\in\{0,\ldots, k\}$ such that $a_{i}^{(j)}= a_{0}^{(j)}$ and $b_{i}^{(j)}= b_{0}^{(j)}$.
We denote by $I$ the maximum of the set $\{I_{j}\mid j\in L_{b_{0},a_{0}} \}$ and by $L_{b_{0},a_{0}}(I)$ the set $\{j\in L_{b_{0},a_{0}}\mid I_{j}=I \}$.
By definition of the source of $\ddd_{d,k/(k-1)}^{d}$, note that $I\leq k-1$ since $\sum_{0\leq i\leq k}a_{i}^{(j)}\leq k$ and there is at least one $a_{i}^{(j)}\geq 1$. Then, by analysing the homogeneity of the $a_{0}$'s and $b_{0}$'s, we deduce that the following summand of the left-hand side of \eqref{eq:step_1_equality} for each $j\in L_{b_{0},a_{0}}(I)$
$$\lambda_{j}\sum_{I+1\leq i\leq k}\mm_{a_{0},b_{0}}^{(j)}\cdots  \mm_{a_{i}+1,b_{i}-1}^{(j)}\cdots \mm_{a_{k},b_{k}}^{(j)}$$
cannot be cancelled by linear combinations of terms
\begin{itemizeb}
    \item of the left-hand side of \eqref{eq:step_1_equality} either indexed by $j_{1}\in L_{b_{0}}\setminus L_{b_{0},a_{0}}(I)$ or those of the form $\lambda_{j_{2}}\sum_{0\leq i\leq I}\mm_{a_{0},b_{0}}^{(j_{2})}\cdots  \mm_{a_{i}+1,b_{i}-1}^{(j_{2})}\cdots \mm_{a_{k},b_{k}}^{(j_{2})}$ with $j_{2}\in L_{b_{0},a_{0}}(I)$;
    \item of the right-hand side of \eqref{eq:step_1_equality}.
\end{itemizeb}
Therefore, denoting by $\mm_{a_{0},b_{0}}$ the common class $\mm_{a_{0},b_{0}}^{(j)}$ when $j\in L_{b_{0},a_{0}}(I)$, we deduce that 
$$\mm_{a_{0},b_{0}}^{\cup I}\sum_{j\in L_{b_{0},a_{0}}(I)}\lambda_{j}\sum_{I+1\leq i\leq k}\mm_{a_{I+1},b_{I+1}}^{(j)}\cdots  \mm_{a_{i}+1,b_{i}-1}^{(j)}\cdots \mm_{a_{k},b_{k}}^{(j)} = 0.$$
Note that this is equal to $\mm_{a_{0},b_{0}}^{\cup I}\ddd_{d,k-I-1/(k-I-2)}^{d}(\mathsf{M}_{L_{b_{0}},I})$, where $\mathsf{M}_{L_{b_{0}},I}$ denotes the element $\sum_{j\in L_{b_{0},a_{0}}(I)}\lambda_{j}\mm_{a_{I+1},b_{I+1}}^{(j)}\cdots \mm_{a_{k},b_{k}}^{(j)}$. Then it follows from the inductive assumption on $k$ that $\mathsf{M}_{L_{b_{0}},I}=0$. Hence there is at least one coefficient $\lambda_{j}$ of $\mathsf{M}_{\ell}$ which is null and thus we are done by the inductive assumption on $\ell$.

\emph{\underline{Case $C$:} $b_{i}^{(j)}=1$ for all $j\in L_{b_{0}}$ and $1\leq i\leq k$.}
Recall that $\sum_{0\leq i\leq k}a_{i}^{(j)}\leq k$ by the above description of the source of $\ddd_{d,k/(k-1)}^{d}$ and that the cup products of monomials of length strictly smaller than $k+1$ vanish in the target of $\ddd_{d,k/(k-1)}^{d}$. Then, the only possibility for the left-hand side of \eqref{eq:step_1_equality} to not be null is that it is of the form
$$\sum_{j\in L_{b_{0}}}\lambda_{j}\mm_{1,b_{0}-1}^{(j)}\mm_{1,1}^{(j)}\cdots \mm_{1,1}^{(j)}.$$
In order to match each of these summands, the preimage $\mathsf{M}_{L\setminus L_{b_{0}}}$ of the right-hand side of \eqref{eq:step_1_equality} must contain the generator $\mm_{1,b_{0}-1}\mm_{0,2}\mm_{1,1}\cdots \mm_{1,1}$. But then, in order to cancel that element, $\mathsf{M}_{L\setminus L_{b_{0}}}$ must also contain either $\mm_{1,b_{0}-1}\mm_{0,2}\mm_{0,2}\mm_{1,1}\cdots \mm_{1,1}$ or $\mm_{2,b_{0}-2}\mm_{0,2}\mm_{1,1}\mm_{1,1}\cdots \mm_{1,1}$, whereas these elements do not belong to the source of $\ddd_{d,k/(k-1)}^{d}$. Therefore we must have $\eqref{eq:step_1_equality}=0$, and since each side of the equality is the image of some $\mathsf{M}_{\ell'}$ with $\ell' < \ell$, we are done by the inductive assumption on $\ell$.

\emph{\textbf{\underline{Step $2$:}} assume that $L_{a_{0}}\neq L$.} By Step~$1$, we may assume that $L_{b_{0}}=L$. It follows from the equality $\ddd_{d,k/(k-1)}^{d}(\mathsf{M}_{\ell})=0$ that:
\begin{equation}\label{eq:step_2_equality}
\sum_{j\in L_{a_{0}}}\lambda_{j}\sum_{0\leq i\leq k}\mm_{a_{0},b_{0}}^{(j)}\cdots  \mm_{a_{i}+1,b_{i}-1}^{(j)}\cdots \mm_{a_{k},b_{k}}^{(j)} = 
-\sum_{j'\in L\setminus L_{a_{0}}}\lambda_{j'}\sum_{0\leq i\leq k}\mm_{a_{0},b_{0}}^{(j')}\cdots  \mm_{a_{i}+1,b_{i}-1}^{(j')}\cdots \mm_{a_{k},b_{k}}^{(j')}.
\end{equation}
The reasoning is then analogous to that of Step~$1$.

\emph{\underline{Case $A$:}} Since we assume $L_{a_{0}}\neq L$, it is impossible that $a_{i}^{(j)}=0$ for all $j\in L_{b_{0}}$ and all $1\leq i\leq k$.

\emph{\underline{Case $B$:} there exists $j\in L_{a_{0}}$ and $1\leq i\leq k$ such that $b_{i}^{(j)}\geq2$.} By Case~$A$, we may assume that there exists $j\in L_{b_{0}}$ and $1\leq i\leq k$ such that $a_{i}^{(j)}\neq 0$.
For each $j\in L_{a_{0}}$, let $I_{j}$ be the maximum $i\in\{0,\ldots, k\}$ such that $a_{i}^{(j)}= a_{0}^{(j)}$ and $b_{i}^{(j)}= b_{0}^{(j)}$.
We denote by $I$ the maximum of the set $\{I_{j}\mid j\in L_{a_{0}} \}$ and by $L_{a_{0}}(I)$ the set $\{j\in L_{a_{0}}\mid I_{j}=I \}$. By definition of the source of $\ddd_{d,k/(k-1)}^{d}$, note that $I\leq k-1$ since $\sum_{0\leq i\leq k}a_{i}^{(j)}\leq k$ and there is at least one $a_{i}^{(j)}\geq 1$.
Then, by a clear analysis on the homogeneity of the $a_{0}$'s and $b_{0}$'s, we deduce that the following summand of the left-hand side of \eqref{eq:step_2_equality} for each $j\in L_{a_{0}}(I)$
$$\lambda_{j}\sum_{I+1\leq i\leq k}\mm_{a_{0},b_{0}}^{(j)}\cdots  \mm_{a_{i}+1,b_{i}-1}^{(j)}\cdots \mm_{a_{k},b_{k}}^{(j)}$$
cannot be cancelled by terms of the form $\lambda_{j}\sum_{0\leq i\leq I}\mm_{a_{0},b_{0}}^{(j)}\cdots  \mm_{a_{i}+1,b_{i}-1}^{(j)}\cdots \mm_{a_{k},b_{k}}^{(j)}$ from the left-hand side of \eqref{eq:step_2_equality}, or by terms of the right-hand side of \eqref{eq:step_2_equality}.
Therefore, denoting by $\mm_{a_{0},b_{0}}$ the common class $\mm_{a_{0},b_{0}}^{(j)}$ when $j\in L_{a_{0}}(I)$, we deduce that 
$$\mm_{a_{0},b_{0}}^{\cup I}\sum_{j\in L_{a_{0}}(I)}\lambda_{j}\sum_{I+1\leq i\leq k}\mm_{a_{I+1},b_{I+1}}^{(j)}\cdots  \mm_{a_{i}+1,b_{i}-1}^{(j)}\cdots \mm_{a_{k},b_{k}}^{(j)} = 0.$$
This is equal to $\mm_{a_{0},b_{0}}^{\cup I}\ddd_{d,k-I-1/(k-I-2)}^{d}(\mathsf{M}_{L})$, where $\mathsf{M}_{L}:=\sum_{j\in L_{a_{0}}(I)}\lambda_{j}\mm_{a_{I+1},b_{I+1}}^{(j)}\cdots \mm_{a_{k},b_{k}}^{(j)}$. Then, by the inductive assumption on $k$,
we deduce that $\mathsf{M}_{L}=0$. Hence there is at least one coefficient $\lambda_{j}$ of $\mathsf{M}_{\ell}$ which is null and thus we are done by the inductive assumption on $\ell$.

\emph{\underline{Case $C$:} $b_{i}^{(j)}=1$ for all $j\in L$ and $1\leq i\leq k$.}
By the properties of the source and target of $\ddd_{d,k/(k-1)}^{d}$, the left-hand side of \eqref{eq:step_2_equality} is of the form $\sum_{j\in L_{a_{0}}}\lambda_{j}\mm_{a_{0}+1,b_{0}-1}^{(j)}\mm_{1,1}^{(j)}\cdots \mm_{1,1}^{(j)}$.
Then, analysing the homogeneity of the $b_{0}$'s, we deduce from the right-hand side of \eqref{eq:step_2_equality} is of the form $-\sum_{j\in L\setminus L_{a_{0}}}\lambda_{j}\mm_{a_{0}+1,b_{0}-1}^{(j')}\mm_{1,1}^{(j')}\cdots \mm_{1,1}^{(j')}$. But $b_{0}^{(j')} < b_{0}^{(j)}$, so $\eqref{eq:step_2_equality}=0$. Hence we are done by the inductive assumption on $\ell$.

\emph{\textbf{\underline{Step $3$:}}} by Steps~$1$ and $2$, we may assume that $L_{a_{0}}=L_{b_{0}}=L$.
The element $\mathsf{M}_{\ell}$ may be rewritten as $\mm_{a_{0},b_{0}} M'_{\ell}$ where $M'_{\ell}:=\sum_{1\leq j\leq \ell}\lambda_{j}\mm_{a_{1},b_{1}}^{(j)}\cdots \mm_{a_{k},b_{k}}^{(j)}$.
Then $\ddd_{d,k/(k-1)}^{d}(\mathsf{M}_{\ell})=0$ implies that
\begin{equation}\label{eq:step_3_equality}
\mm_{a_{0},b_{0}} \sum_{1\leq j\leq \ell}\lambda_{j}\sum_{1\leq i\leq k}\mm_{a_{1},b_{1}}^{(j)}\cdots\mm_{a_{i}+1,b_{i}-1}^{(j)}\cdots \mm_{a_{k},b_{k}}^{(j)}=-\mm_{a_{0}+1,b_{0}-1} M'_{\ell}.
\end{equation}
Analysing the homogeneity of the $a_{0}$'s and $b_{0}$'s, the equality \eqref{eq:step_3_equality} implies that $b_{1}^{(j)}=b_{0}$ and $a_{1}^{(j)}=a_{0}$ for all $1\leq j\leq \ell$.
Then \eqref{eq:step_3_equality} is equivalent to:
$$\mm_{a_{0},b_{0}}^{2}\sum_{1\leq j\leq \ell}\lambda_{j}\sum_{2\leq i\leq k} \mm_{a_{2},b_{2}}^{(j)}\cdots\mm_{a_{i}+1,b_{i}-1}^{(j)}\cdots \mm_{a_{k},b_{k}}^{(j)}=
-2\mm_{a_{0},b_{0}}\mm_{a_{0}+1,b_{0}-1}\sum_{1\leq j\leq \ell}\lambda_{j}\mm_{a_{2},b_{2}}^{(j)}\cdots \mm_{a_{k},b_{k}}^{(j)}.$$
Then, by a clear recursion on the homogeneity of the $a$'s and $b$'s, this implies that $b_{i}^{(j)}=b_{0}$ and $a_{i}^{(j)}=a_{0}$ for all $2\leq i\leq k$ and $1\leq j\leq \ell$. Then we obtain that $\mm_{a_{0},b_{0}}^{k} \mm_{a_{0}+1,b_{0}-1}\sum_{1\leq j\leq \ell}\lambda_{j}=
-\mm_{a_{0},b_{0}}^{k}\mm_{a_{0}+1,b_{0}-1}\sum_{1\leq j\leq \ell}k\lambda_{j}$. But then this implies that $\sum_{1\leq j\leq \ell}\lambda_{j}= -k\sum_{1\leq j\leq \ell}\lambda_{j}$ and thus $\sum_{1\leq j\leq \ell}\lambda_{j}=0$. Hence $\mathsf{M}_{\ell}=0$, which ends the proof.
\end{proof}

\begin{coro}\label{coro:low_dim_paired}
We fix integers $d\geq1$ and $d'\geq2$. Then, $H_{\st}^{2i+1}(\Lambda^{2d+1}\tilde{H}_{\bQ})=H_{\st}^{2i'}(\Lambda^{2d'}\tilde{H}_{\bQ})=0$ for $i<d$ and $i'<d'$.
\end{coro}

\subsubsection{Computing the $\Tor$-groups}\label{ss:theory_Tor}
We finally detail the general methods we use in \S\ref{s:covariant_coeff_syst_exterior_computations} to compute the $\Tor$-groups of $H_{\st}^{*}(\Lambda^{d}\tilde{H}_{\bQ})$. We recall that we write $H_{*}(\Sym_{\bQ}(\cE);-)$ for $\Tor_{*}^{\Sym_{\bQ}(\cE)}(\bQ,-)$.

\paragraph*{General properties.}
Since $H_{\st}^{\dagger}(\Lambda^{d}{H}_{\bQ})$ and  $H_{\st}^{\ddagger}(\Lambda^{d-1}{H}_{\bQ})$ are free over $\Sym_{\bQ}(\cE)$, the concatenation of a $\Sym_{\bQ}(\cE)$-free resolution of $H_{\st}^{\dagger}(\Lambda^{d}\tilde{H}_{\bQ})$ and the exact sequence \eqref{ES_cohomology_exterior} is a $\Sym_{\bQ}(\cE)$-free resolution of $H_{\st}^{\ddagger}(\Lambda^{d}\tilde{H}_{\bQ})$. Hence we deduce that for any $j >0$ 
\begin{equation}\label{eq:exterior_Tor}
H_{j}(\Sym_{\bQ}(\cE);H_{\st}^{\dagger}(\Lambda^{d}\tilde{H}_{\bQ}))
\cong H_{j+2}(\Sym_{\bQ}(\cE); H_{\st}^{\ddagger}(\Lambda^{d}\tilde{H}_{\bQ})).
\end{equation}
In general, the module $H_{\st}^{\ddagger}(\Lambda^{d}\tilde{H}_{\bQ})$ is explicitly computable (at least for small $d$, see Theorems~\ref{prop:4th-ext_odd}, \ref{prop:3rd-ext-even} and \ref{thm:5th-ext-even}), and we are then able to fully deduce the $\Tor$-groups of $H_{\st}^{*}(\Lambda^{d}\tilde{H}_{\bQ})$ for $j >0$ from those of $H_{\st}^{\ddagger}(\Lambda^{d}\tilde{H}_{\bQ})$ thanks to \eqref{eq:exterior_Tor}; see Propositions~\ref{prop:3rd-ext_Tor}, \ref{prop:4th-ext_Tor} and Theorem~\ref{thm:5th-ext_Tor}.

Moreover, we have the following general result which refines Corollary~\ref{coro:stable_cohomo_algebra_not_free}:
\begin{thm}\label{thm:Tor_j_non-trivial}
If $d \geq 3$, for any $j \geq 0$, we have $H_{j}(\Sym_{\bQ}(\cE); H_{\st}^{\ddagger}(\Lambda^{d}\tilde{H}_{\bQ}))
\neq 0$, and a fortiori
$$H_{j}(\Sym_{\bQ}(\cE); H_{\st}^{*}(\Lambda^{d}\tilde{H}_{\bQ}))\neq 0.$$
Also, we have $H_{j'}(\Sym_{\bQ}(\cE); H_{\st}^{\dagger}(\Lambda^{d}\tilde{H}_{\bQ}))\neq 0$ for $j'\geq2$. In particular, the $\Sym_{\bQ}(\cE)$-modules $H_{\st}^{\dagger}(\Lambda^{d}\tilde{H}_{\bQ})$ is not free.
\end{thm}
\begin{proof}
Let $\bQ_{\cE_{\geq d-1}}$ be the trivial $\Sym_{\bQ}(\cE_{\geq d-1})$-module $\bQ$. Viewing $H_{\st}^{\ddagger}(\Lambda^{d}\tilde{H}_{\bQ})$ as a $\Sym_{\bQ}(\cE_{\leq d-2})$-module, it follows from \ref{(5)} in Proposition~\ref{prop:general_result_filtration} that $H_{\st}^{\ddagger}(\Lambda^{d}\tilde{H}_{\bQ})$ is isomorphic to $H_{\st}^{\ddagger}(\Lambda^{d}\tilde{H}_{\bQ})\otimes_{\bQ}\bQ_{\cE_{\geq d-1}}$.
Then, by Lemma~\ref{lem:general_properties_homology_algebra}, we have:
$$\aligned
H_{*}(\Sym_{\bQ}(\cE); H_{\st}^{\ddagger}(\Lambda^{d}\tilde{H}_{\bQ}))
& = H_{*}(\Sym_{\bQ}(\cE_{\leq d-2}); H_{\st}^{\ddagger}(\Lambda^{d}\tilde{H}_{\bQ}))\otimes \Lambda^{*}\cE_{\geq d-1}\\
& \supset H_{0}(\Sym_{\bQ}(\cE_{\leq d-2}); H_{\st}^{\ddagger}(\Lambda^{d}\tilde{H}_{\bQ}))\otimes \Lambda^{*}\cE_{\geq d-1}.
\endaligned$$
It follows from Theorem~\ref{thm:general_qualitative_thm} that the $\Sym_{\bQ}(\cE_{\leq d-2})$-module $H_{\st}^{\ddagger}(\Lambda^{d}\tilde{H}_{\bQ})$ has a non-trivial element $m$ of lowest cohomological degree. Using the resolution \eqref{eq:chain_complex_Tor} and formulas \eqref{eq:iD} to compute $H_{0}(\Sym_{\bQ}(\cE_{\leq d-2}); H_{\st}^{\ddagger}(\Lambda^{d}\tilde{H}_{\bQ}))$, the element $m$ does not belong to the image of the boundary map $\pa_{1}$. Indeed, $m$ would otherwise be equal to a linear combination of elements of the form $em'$ with $e\in\Sym_{\bQ}(\cE)$ and $m'\in H_{\st}^{\ddagger}(\Lambda^{d}\tilde{H}_{\bQ})$, where $e$ would have a strictly positive cohomological degree and thus $m'$ would have a lower degree than $m$, which contradicts the assumption on $m$.
Therefore, the image of $m$ in $ H_{0}(\Sym_{\bQ}(\cE_{\leq d-2}); H_{\st}^{\ddagger}(\Lambda^{d}\tilde{H}_{\bQ}))$ is not null, which gives the non-triviality of that $0^{\ith}$ cohomology group. Moreover, we have $\Lambda^j\cE_{\geq d-1}
\neq 0$ for any $j \geq 0$. Hence the right-hand side does not vanish for any degree $j\geq 0$, thus ending the first part of the proof.
Finally, we obtain the non-vanishing of the $\Tor$-group for $H_{\st}^{\dagger}(\Lambda^{d}\tilde{H}_{\bQ})$ by using \eqref{eq:exterior_Tor}.
\end{proof}

\paragraph*{Computation of $\Tor_{0}$.}
In contrast with \eqref{eq:exterior_Tor}, the computation of $H_{0}(\Sym_{\bQ}(\cE);H_{\st}^{*}(\Lambda^{d}\tilde{H}_{\bQ}))$ requires the following more subtle analysis.
We denote by $\Pi$ the surjection $H_{\st}^{\ddagger}(\Lambda^{d-1}{H}_{\bQ})\twoheadrightarrow H_{\st}^{\ddagger}(\Lambda^{d}\tilde{H}_{\bQ})$ of the sequence 
\eqref{ES_cohomology_exterior}. We then consider the decomposition of that sequence into the two short exact sequences where $\Ker(\Pi)$ appears:
\begin{equation}\label{eq:kerdelta_ses_1}
0 \to H_{\st}^{\dagger}(\Lambda^{d}\tilde{H}_{\bQ}) \to H_{\st}^{\dagger}(\Lambda^{d}H_{\bQ}) \to \Ker(\Pi)\to 0,
\end{equation}
\begin{equation}\label{eq:kerdelta_ses_2}
0 \to \Ker(\Pi) \to H_{\st}^{\ddagger}(\Lambda^{d-1}H_{\bQ}) \to H_{\st}^{\ddagger}(\Lambda^{d}\tilde{H}_{\bQ}) \to 0.
\end{equation}
\begin{lem}\label{lem:Tor_0_lemma}
Writing $H_{*}(-) := H_{*}({\Sym_{\bQ}(\cE)};-)$ for the sake of concision, we have an exact sequence:
\begin{equation}\label{eq:Tor_0_ES}
0 \to H_{2}(H_{\st}^{\ddagger}(\Lambda^{d}\tilde{H}_{\bQ}))\to
H_{0}(H_{\st}^{\dagger}(\Lambda^{d}\tilde{H}_{\bQ}))
\to \Ker(H_{0}(\ddd_{d})) \overset{\Delta_{d}}\to 
H_{1}(H_{\st}^{\ddagger}(\Lambda^{d}\tilde{H}_{\bQ})) \to 0.
\end{equation}
In particular, $H_{0}({\Sym_{\bQ}(\cE)};H_{\st}^{\dagger}(\Lambda^{d}\tilde{H}_{\bQ})) \cong H_{2}({\Sym_{\bQ}(\cE)};H_{\st}^{\ddagger}(\Lambda^{d}\tilde{H}_{\bQ})\oplus \Ker(\Delta_{d})$.
\end{lem}
\begin{proof}
The exact sequence \eqref{eq:Tor_0_ES} is constructed from the following combination of the $\Tor$-long exact sequences of the short exact sequences \eqref{eq:kerdelta_ses_1} and \eqref{eq:kerdelta_ses_2}
$$\xymatrix{H_{1}(\Ker(\Pi))\ar@{^{(}->}[r] & 
H_{0}(H_{\st}^{\dagger}(\Lambda^{d}\tilde{H}_{\bQ})) \ar[r] 
& H_{0}(H_{\st}^{\dagger}(\Lambda^{d}H_{\bQ})) \ar@{->>}[r] \ar[d]^-{H_{0}(\ddd_{d})} 
& H_{0}(\Ker(\Pi)) \ar@{=}[d] & \\
& H_{0}(H_{\st}^{\ddagger}(\Lambda^{d}\tilde{H}_{\bQ})) & \ar@{->>}[l] H_{0}(H_{\st}^{\ddagger}(\Lambda^{d-1}H_{\bQ}))
& \ar[l] H_{0}(\Ker(\Pi)) & \ar@{_{(}->}[l] H_{1}(H_{\st}^{\ddagger}(\Lambda^{d}\tilde{H}_{\bQ}))}$$
and the fact that $H_{2}(H_{\st}^{\ddagger}(\Lambda^{d}\tilde{H}_{\bQ})) \cong H_{1}(\Ker(\Pi))$ since $H_{\st}^{\ddagger}(\Lambda^{d-1}H_{\bQ})$ is a free $\Sym_{\bQ}(\cE)$-module.
\end{proof}
Since $H_{\st}^{\ddagger}(\Lambda^{d}\tilde{H}_{\bQ})$ is generally explicitly computable (at least for $d\leq 5$; see \S\ref{s:covariant_coeff_syst_exterior_computations}), it follows from \eqref{eq:Tor_0_ES} that the only remaining point in order to compute $H_{0}({\Sym_{\bQ}(\cE)};H_{\st}^{\dagger}(\Lambda^{d}\tilde{H}_{\bQ}))$ is to determine the kernel of the map $\Delta_{d}: \Ker(H_{0}(\ddd_{d})) \twoheadrightarrow
H_{1}(\Sym_{\bQ}(\cE);H_{\st}^{\ddagger}(\Lambda^{d}\tilde{H}_{\bQ}))$. In that respect, it follows from the freeness of $\Ker(\ddd_{d,k/(k-1)})$ as a $\Sym_{\bQ}(\cE)$-module for each $1\leq k\leq d-1$ (see \ref{(1)} in Proposition~\ref{prop:general_result_filtration}) and from the right-exactness of
$H_{0}({\Sym_{\bQ}(\cE)};-)$ that $$H_{0}(\ddd_{d})\cong \bigoplus_{1\leq k\leq d-1} H_{0}({\Sym_{\bQ}(\cE)};\ddd_{d,k/k-1}).$$
Then, since $H_{\st}^{\ddagger}(\Lambda^{d-1}\tilde{H}_{\bQ})$ is a free $\Sym_{\bQ}(\cE)$-module, we deduce the following isomorphism from the commutation of a kernel with a finite direct sum:
\begin{equation}\label{eq:decomposition_filtration_Ker_H_{0}}
\Ker(H_{0}(\ddd_{d})) \cong \bigoplus_{1\leq k\leq d-1} H_{0}({\Sym_{\bQ}(\cE)};\Ker(\ddd_{d,k/k-1})).
\end{equation}
The resolution \eqref{eq:chain_complex_Tor} with $\cE':=\cE$ induces a short exact sequence of chain complexes
\begin{equation}\label{eq:complex_SES}
\xymatrix{& \vdots \ar[d]& \vdots \ar[d]& \vdots \ar[d]\\
0 \ar[r] & \Ker(\Pi)\otimes\cE \ar[d]^-{\pa_{1}} \ar[r] & 
H_{\st}^{\ddagger}(\Lambda^{d-1}H_{\bQ})\otimes\cE \ar[d]^-{\pa_{1}} \ar[r]^-{\Pi\otimes \cE} &
H_{\st}^{\ddagger}(\Lambda^{d}\tilde{H}_{\bQ}) \otimes\cE\ar[d]^-{\pa_{1}}\ar[r] & 0\\
0 \ar[r] & \Ker(\Pi) \ar[r] & H_{\st}^{\ddagger}(\Lambda^{d-1}H_{\bQ}) \ar[r]^-{\Pi} &
H_{\st}^{\ddagger}(\Lambda^{d}\tilde{H}_{\bQ}) \ar[r] & 0.}
\end{equation}
The following result gives a sufficient condition to describe the image of $\Delta_{d}$, which will be of key use to compute $\Ker(\Delta_{d})$ in \S\ref{s:covariant_coeff_syst_exterior_computations}.
\begin{prop}\label{prop:Ker_Delta_method}
Let $\tilde u\in H_{\st}^{\dagger}(\Lambda^{d}H_{\bQ})$ be a lift of an element $u \in H_{0}(\Ker(\ddd_{d, k/k-1}))$. If there exists $\hat{u} \in H_{\st}^{\ddagger}(\Lambda^{d-1}H_{\bQ})\otimes\cE$ such that
$$\pa_{1}(\hat{u})=\ddd_{d}(\tilde u)\in H_{\st}^{\ddagger}(\Lambda^{d-1}H_{\bQ}),$$
then the homology class of $(\Pi\otimes\cE)(\hat{u})$ is equal to $\Delta_{d}(u)$.
\end{prop}

\begin{proof}
The result follows from the fact that the connecting homomorphism of \eqref{eq:complex_SES} is formally equal to the connecting homomorphism $H_{1}(\Sym_{\bQ}(\cE);H_{\st}^{\ddagger}(\Lambda^{d}\tilde{H}_{\bQ})) \hookrightarrow H_{0}(\Sym_{\bQ}(\cE);\Ker(\Pi))$ of the short exact sequence \eqref{eq:kerdelta_ses_2}.
\end{proof}

\section{Stable cohomology in covariant coefficients: computations for small powers}\label{s:covariant_coeff_syst_exterior_computations}

We recall that we use the conventions and notation
of Convention~\ref{conv:section_5}.
In this section, we study the stable twisted cohomology graded modules $H_{\st}^{*}(\Lambda^{d}\tilde{H}_{\bQ})$ for $d\leq 5$. More precisely, we deal with the exact sequence \eqref{ES_cohomology_exterior} for each for $d\leq 5$ to calculate $H_{\st}^{\ddagger}(\Lambda^{d}\tilde{H}_{\bQ})$ where $\ddagger=``\even"$ if $d$ is odd and $\ddagger=``\odd"$ if $d$ is even: this corresponds to $\Cok(\ddd_{d})$, whose computation is always doable since this is a cokernel. This also allows us to fully compute the $\Tor$-groups of $H_{\st}^{*}(\Lambda^{d}\tilde{H}_{\bQ})$ as a $\Sym_{\bQ}(\cE)$-module via the techniques of \S\ref{ss:theory_Tor}.
We decompose our work in \S\ref{ss:second_exterior}--\S\ref{ss:fifth_exterior} following the degree of the exterior power.

\subsection{Second exterior power}\label{ss:second_exterior}
We consider the exact sequence \eqref{ES_cohomology_exterior} for $d=2$. We apply the contraction formulas of Proposition~\ref{prop:Contraction_Formula} to compute the image of the derivation $\ddd_{2}$ as follows. For all integers $a\geq0$ and $\alpha,\beta\geq1$, we have $\ddd_{2}(\mm_{a,2}) = \mm_{a+1,1}$ and $\ddd_{2}(\mm_{\alpha,1} \mm_{\beta,1})=e_{\alpha}\mm_{\beta,1}+e_{\beta}\mm_{\alpha,1}$.
We deduce from these formulas that the homomorphism $\ddd_{2}$ has 
a right inverse defined by $\mm_{j,1}\mapsto \mm_{j-1,2}$ for all $j\geq1$. Therefore, we have:
\begin{thm}\label{thm:result_second_exterior_power_contra}
The $\Sym_{\bQ}(\cE)$-module $H_{\st}^{*}(\Lambda^{2}\tilde{H}_{\bQ})$ is isomorphic to the free $\Sym_{\bQ}(\cE)$-module with basis
$$\left\{\mm_{j,1} \mm_{l,1}-e_{l}\mm_{j,2}-e_{j}\mm_{l,2};\, j\geq l\geq1 \right\}.$$
In particular, $H_{\st}^{*}(\Lambda^{2}\tilde{H}_{\bQ})=H_{\st}^{\even}(\Lambda^{2}\tilde{H}_{\bQ})$, $H_{\st}^{\odd}(\Lambda^{2}\tilde{H}_{\bQ})=0$ and $H_{j}(\Sym_{\bQ}(\cE);H_{\st}^{*}(\Lambda^{2}\tilde{H}_{\bQ}))=0$ for all $j > 0$.
\end{thm}

\subsection{Third exterior power}\label{ss:third_exterior}
We consider the exact sequence \eqref{ES_cohomology_exterior} for $d=3$.
By computing the image of the derivation $\ddd_{3}$ thanks to the contraction formulas of Proposition~\ref{prop:Contraction_Formula}, we prove the following result.
\begin{thm}\label{prop:3rd-ext-even}
The $\Sym_{\bQ}(\cE)$-module $H_{\st}^{\even}(\Lambda^{3}\tilde{H}_{\bQ})$ is isomorphic to
$$(\Sym_{\bQ}(\cE)/(e_{1}^{2},e_{\alpha},\alpha\geq2))\left\{\mm_{0,2}\right\}.$$
\end{thm}

\begin{proof}
We progressively compute $\Cok(\ddd_{3,k})$ for $0\leq k\leq 2$ using the exact sequence \eqref{eq:snake} for $d=3$.

\underline{Computation of $\Cok(\ddd_{3,1})$:}
we compute that for all $a\geq 0$ and $\alpha\geq1$
$$\ddd_{3,1/0}(\mm_{a,2}\mm_{\alpha,1})=\mm_{a+1,1}\mm_{\alpha,1}=\ddd_{3,1/0}(\mm_{\alpha-1,2}\mm_{a+1,1}).$$
It follows from the structure of the filtration $F_{1}H_{\st}^{*}(\Lambda^{3}{H}_{\bQ})$ and the formulas for $\ddd_{3,1/0}$ that the above equality provides the only way to construct elements of $\Ker(\ddd_{3,1/0})$. Hence $\{[\mm_{a,2}\mm_{\alpha,1}-\mm_{\alpha-1,2}\mm_{a+1,1}];\alpha -1 > a \geq 0 \}$ defines a basis of $\Ker(\ddd_{3,1/0})$ as a free $\Sym_{\bQ}(\cE)$-module and $\Cok(\ddd_{3,1/0})=0$. By the formal definition of the connecting homomorphism defined by the snake lemma, each one of these generators is mapped to $[e_{\alpha}\mm_{a,2}-e_{a+1}\mm_{\alpha-1,2}]$ in $\Cok(\ddd_{3,0})$ by $\delta_{3,0}$. We recall from \ref{(2)} in Proposition~\ref{prop:general_result_filtration} that $\Cok(\ddd_{3,0})\cong \Sym_{\bQ}(\cE)[\mm_{0,2}]$. Hence we have $\mathrm{Im}(\delta_{3,0})= (e_{\alpha}; \alpha\geq2)[\mm_{0,2}]$, and a fortiori $\Cok(\ddd_{3,1}) \cong (\Sym_{\bQ}(\cE)/(e_{\alpha}; \alpha\geq2))[\mm_{0,2}]$.

\underline{Computation of $\Cok(\ddd_{3,2})$:}
since the target of $\ddd_{3,2/1}$ is zero, we deduce that $\Cok(\ddd_{3,2/1}) = 0$ and that $\Ker(\ddd_{3,2/1})$ is the free $\Sym_{\bQ}(\cE)$-module with basis $\{[\mm_{\alpha,1}\mm_{\beta,1}\mm_{\gamma,1}];1\leq\alpha\leq\beta\leq\gamma \}$.
The connecting homomorphism $\delta_{3,2}$ maps each generator $\mm_{\alpha,1}\mm_{\beta,1}\mm_{\gamma,1}$ to the class of $[e_{\alpha}\mm_{\beta,1} \mm_{\gamma,1}+e_{\beta}\mm_{\alpha,1} \mm_{\gamma,1}+e_{\gamma}\mm_{\alpha,1} \mm_{\beta,1}]$ in $\Cok(\ddd_{3,1})$. Since $[\mm_{\alpha,1}\mm_{\beta,1}] = 0$ in $\Cok(\ddd_{3,1})$ for $\beta > 1$ by \ref{(4)} in Proposition~\ref{prop:general_result_filtration} and $[e_{1}\mm_{0,2}]=-[\mm_{1,1}\mm_{1,1}]$ in $\Cok(\ddd_{3,1})$ by \ref{(3)} in Proposition~\ref{prop:general_result_filtration}, we compute that:
\begin{equation}\label{eq:cok_3,2}
[e_{\alpha}\mm_{\beta,1} \mm_{\gamma,1}+e_{\beta}\mm_{\alpha,1} \mm_{\gamma,1}+e_{\gamma}\mm_{\alpha,1} \mm_{\beta,1}]= \begin{cases}
-3[e_{\gamma}e_{1}\mm_{0,2}] & \text{if $\alpha=\beta=1$, }\\
0 & \text{otherwise.}\\
\end{cases}
\end{equation}
This implies the last additional relation $[e_{1}^2\mm_{0,2}]=0$ in $\Cok(\ddd_{3,2})$, which ends the proof.
\end{proof}

It seems complicated to fully compute the kernel 
$\Ker(\ddd_{3}) \cong H_{\st}^{\odd}(\Lambda^{3}\tilde{H}_{\bQ})$. Indeed, the simplest example of elements of that kernel is defined as follows. Denoting by $\chi_{a,b}$ the element $e_{a}\mm_{b-1,3}+\mm_{a,1} \mm_{b,2}$, then the element
$e_{c}\chi_{a,b-1}+e_{b}\chi_{c,a-1}+e_{a}\chi_{b,c-1}-\mm_{a,1} \mm_{b,1} \mm_{c,1}$ belongs to $\Ker(\ddd_{3})$. We also recall from Corollary~\ref{coro:low_dim_paired} that $H_{\st}^{1}(\Lambda^{3}\tilde{H}_{\bQ})=0$.

Furthermore, we compute the $\Tor$-groups of $H_{\st}^{*}(\Lambda^{3}\tilde{H}_{\bQ})$:
\begin{prop}\label{prop:3rd-ext_Tor}
For any $j > 0$, we have 
$$H_{j}(\Sym_{\bQ}(\cE);H_{\st}^{*}(\Lambda^{3}\tilde{H}_{\bQ}))\cong \Lambda^{j -1}\cE_{\geq 2}\oplus\Lambda^{j}\cE_{\geq 2}\oplus\Lambda^{j+1}\cE_{\geq 2}\oplus\Lambda^{j+2}\cE_{\geq 2}.$$
Moreover, we have
\begin{align*}
H_{0}(\Sym_{\bQ}(\cE);H_{\st}^{*}(\Lambda^{3}\tilde{H}_{\bQ}))  \;\cong\;& \bQ\oplus \cE_{\geq 2}\oplus\Lambda^{2}\cE_{\geq 2} \\  & \oplus \bQ\{[\mm_{\alpha-1,2}\mm_{\beta,1} - \mm_{\beta-1,2}\mm_{\alpha,1}];2 \leq \alpha < \beta \}\\  & \oplus \bQ\{[\mm_{\alpha,1}\mm_{\beta,1}\mm_{\gamma,1}];1 \leq \alpha \leq \beta \leq \gamma \geq 2 \}.
\end{align*}
\end{prop}

\begin{proof}
By the computations \eqref{eq:computation_Tor_general} and Theorem~\ref{prop:3rd-ext-even}, we deduce that $H_{j}(\Sym_{\bQ}(\cE);H_{\st}^{\even}(\Lambda^{3}\tilde{H}_{\bQ}))\cong\Lambda^{j-1}\cE_{\geq 2}\oplus\Lambda^{j}\cE_{\geq 2}$
for all $j \geq1$,  while $H_{0}(\Sym_{\bQ}(\cE);H_{\st}^{\even}(\Lambda^{3}\tilde{H}_{\bQ}))\cong \bQ$. Combining this computation with the isomorphism \eqref{eq:exterior_Tor}, we deduce that $H_{j}(\Sym_{\bQ}(\cE);H_{\st}^{\odd}(\Lambda^{3}\tilde{H}_{\bQ}))\cong \Lambda^{j+2}\cE_{\geq 2}\oplus\Lambda^{j+1}\cE_{\geq 2}$ for $j\geq1$. This proves the first part of the proposition.

For the second part of the proposition, we first have the contribution of $H_{0}(\Sym_{\bQ}(\cE);H_{\st}^{\even}(\Lambda^{3}\tilde{H}_{\bQ}))$ computed above.
By Lemma~\ref{lem:Tor_0_lemma}, we compute $H_{0}(\Sym_{\bQ}(\cE);H_{\st}^{\odd}(\Lambda^{3}\tilde{H}_{\bQ}))$ thanks to the exact sequence \eqref{eq:Tor_0_ES}. We recall that $H_{2}(\Sym_{\bQ}(\cE);H_{\st}^{\even}(\Lambda^{3}\tilde{H}_{\bQ})$ is computed above.
Following the decomposition \eqref{eq:decomposition_filtration_Ker_H_{0}}, we now have to progressively determine the generators of $\Ker(\Delta_{3})$ associated to each summand $\Ker(\ddd_{3,k/k-1})$ for $1\leq k\leq 2$ to finish the proof. For this purpose, we use the chain complex \eqref{eq:complex_SES} and Proposition~\ref{prop:Ker_Delta_method}.

\underline{For the summand $\Ker(\ddd_{3, 1/0})$:} for each $1 \leq \alpha < \beta$, the element $\ddd_{3}(\mm_{\alpha-1,2}\mm_{\beta,1} - \mm_{\beta-1,2}\mm_{\alpha,1})$ is equal to $\pa_{1}(\mm_{\alpha-1,2}\otimes e_{\beta} - \mm_{\beta-1,2}\otimes e_{\alpha})$ in $H_{\st}^{\even}(\Lambda^{2}H_{\bQ})$. By Theorem~\ref{prop:3rd-ext-even}, we have in $H_{\st}^{\even}(\Lambda^{3}\tilde{H}_{\bQ})\otimes\cE$
$$\Pi(\mm_{\alpha-1,2})\otimes e_{\beta}
- \Pi(\mm_{\beta-1,2})\otimes e_{\alpha}
= \begin{cases}
0 & \text{if $\alpha\geq 2$,}\\
[\mm_{0,2}]\otimes e_{\beta} & \text{if $\alpha = 1$ and $\beta \geq 2$.}
\end{cases}$$
The element $[\mm_{0,2}]\otimes e_{\beta}$ corresponds to the non-trivial class $e_{\beta} \in \cE_{\geq2} \subset H_{1}(\Sym_{\bQ}(\cE);H_{\st}^{\even}(\Lambda^{3}\tilde{H}_{\bQ}))$. Hence, by Proposition~\ref{prop:Ker_Delta_method}, the restriction of $\Ker(\Delta_{3})$ to $H_{0}({\Sym_{\bQ}(\cE)};\Ker(\ddd_{3,1/0}))$ is isomorphic to the free $\bQ$-module generated by $\{[\mm_{\alpha-1,2}\mm_{\beta,1} - \mm_{\beta-1,2}\mm_{\alpha,1}];2 \leq \alpha < \beta \}$.

\underline{For the summand $\Ker(\ddd_{3,2/1})$:} for each $1\leq\alpha\leq\beta\leq\gamma$, the element $\ddd_{3}(\mm_{\alpha,1}\mm_{\beta,1}\mm_{\gamma,1})$ is equal to $\pa_{1}(\mm_{\alpha,1}\mm_{\beta,1}\otimes e_{\gamma} +\mm_{\beta,1}\mm_{\gamma,1}\otimes e_{\alpha} +\mm_{\alpha,1}\mm_{\gamma,1}\otimes e_{\beta})$ in $H_{\st}^{\even}(\Lambda^{2}H_{\bQ})$. 
Similarly to relations \eqref{eq:cok_3,2} and by Theorem~\ref{prop:3rd-ext-even}, the element $(\Pi\otimes\cE)(\mm_{\alpha,1}\mm_{\beta,1}\otimes e_{\gamma} +\mm_{\beta,1}\mm_{\gamma,1}\otimes e_{\alpha} +\mm_{\alpha,1}\mm_{\gamma,1}\otimes e_{\beta})$ of $H_{\st}^{\even}(\Lambda^{3}\tilde{H}_{\bQ})\otimes\cE$ is equal to $-3[e_{1}\mm_{0,2}\otimes e_{1}] \neq 0\in H_{1}(\Sym_{\bQ}(\cE);H_{\st}^{\even}(\Lambda^{3}\tilde{H}_{\bQ}))$ if $1 = \alpha = \beta =\gamma$, and it vanishes if $\gamma \geq 2$ because $\pa_{2}(\mm_{0,2}\otimes e_{1}\wedge e_{\gamma}) = e_{1}\mm_{0,2}\otimes e_{\gamma}$.
This element is also equal to $\Delta_{3}([\mm_{\alpha,1}\mm_{\beta,1}\mm_{\gamma,1}])$ by Proposition~\ref{prop:Ker_Delta_method}, so the restriction of $\Ker(\Delta_{3})$ to $H_{0}({\Sym_{\bQ}(\cE)};\Ker(\ddd_{3,2/1}))$ is the free $\bQ$-module generated by $\{[\mm_{\alpha,1}\mm_{\beta,1}\mm_{\gamma,1}];1\leq \alpha\leq \beta\leq \gamma \geq 2 \}$.
\end{proof}

\subsection{Fourth exterior power}\label{ss:fourth_exterior}

We consider the exact sequence \eqref{ES_cohomology_exterior} for $d=4$. We apply the contraction formulas of Proposition~\ref{prop:Contraction_Formula} to compute the following images of the derivation $\ddd_{4}$ and prove the following result.
\begin{thm}\label{prop:4th-ext_odd}
The $\Sym_{\bQ}(\cE)$-module $H_{\st}^{\odd}(\Lambda^{4}\tilde{H}_{\bQ})$ is isomorphic to $$(\Sym_{\bQ}(\cE)/(e_{1}^{2},e_{\alpha},\alpha\geq2)) \left\{\mm_{0,3}\right\}.$$
\end{thm}
\begin{proof}
We progressively compute $\Cok(\ddd_{4,k})$ for $0\leq k\leq 3$ using the exact sequence \eqref{eq:snake}.

\underline{Computation of $\Cok(\ddd_{4,1})$:}
we compute that for all $a,b,c\geq 0$ and $\alpha\geq1$
\begin{eqnarray*}
   &&\ddd_{4,1/0}(\mm_{a,3}\mm_{\alpha,1})=\mm_{a+1,2}\mm_{\alpha,1}\,;\\&&\ddd_{4,1/0}(\mm_{b,2}\mm_{c,2}) = \mm_{b,2}\mm_{c+1,1} + \mm_{c,2}\mm_{b+1,1}.
\end{eqnarray*}
Also, we define a section $\sigma_{4,1/0}$ of the map $\ddd_{4,1/0}$ by assigning
$$\sigma_{4,1/0}(\mm_{a,2}\mm_{\alpha,1})
= \begin{cases}
\tfrac{1}{2}\mm_{0,2}\mm_{0,2}, & \text{if $a=0$ and $\alpha = 1$,}\\
\mm_{0,2}\mm_{\alpha-1,2} - \mm_{\alpha-2,3}\mm_{1,1},
& \text{if $a=0$ and $\alpha \geq 2$,}\\
\mm_{a-1,3}\mm_{\alpha,1}, 
& \text{if $a\geq1$ and $\alpha \geq 1$.}\\
\end{cases}$$
Hence $\Cok(\ddd_{4,1/0})=0$. Also, by using the section $\sigma_{4,1/0}$, we deduce from the formulas of $\ddd_{4,1/0}$ that $\{[\mm_{a,2}\mm_{b,2}-\mm_{a-1,3}\mm_{b+1,1}-\mm_{b-1,3}\mm_{a+1,1}];1\leq a\leq b\}$ defines a basis of $\Ker(\ddd_{4,1/0})$ as a free $\Sym_{\bQ}(\cE)$-module.
By the formal definition of the connecting homomorphism defined by the snake lemma, each one of these generators is mapped by $\delta_{4,0}$ to $-[e_{b+1}\mm_{a-1,3}+e_{a+1}\mm_{b-1,3}]$ in $\Cok(\ddd_{4,0})$. Also, we recall from \ref{(2)} in Proposition~\ref{prop:general_result_filtration} that $\Cok(\ddd_{4,0})\cong\Sym_{\bQ}(\cE)[\mm_{0,3}]$. Hence we have $\mathrm{Im}(\delta_{4,0}) = (e_{\alpha}; \alpha\geq2)[\mm_{0,3}]$, and thus $\Cok(\ddd_{4,1}) \cong  (\Sym_{\bQ}(\cE)/(e_{\alpha}; \alpha\geq2))[\mm_{0,3}]$.
Furthermore, we note from \ref{(3)} in Proposition~\ref{prop:general_result_filtration} that in $\Cok(\ddd_{4,1})$
\begin{equation}\label{eq:rel_D_4,1}
 [\mm_{a,2}\mm_{\alpha,1}] = \begin{cases}
0 & \text{if $a \geq 2$, $\alpha\geq 3$ or $(a,\alpha) = (0,1)$},\\
-[e_{\alpha}\mm_{0,3}] & \text{if $a=1$},\\
[e_{a+1}\mm_{0,3}] & \text{if $\alpha=2$},
\end{cases}
\end{equation}
and $\ddd_{4}(\mm_{0,2}\mm_{0,2}) = 2\mm_{0,2}\mm_{1,1}$.

\underline{Computation of $\Cok(\ddd_{4,2})$:}
we compute that for all $a\geq 0$ and $\alpha,\beta\geq1$
\begin{eqnarray*}
\ddd_{4,2/1}(\mm_{a,2} \mm_{\alpha,1} \mm_{\beta,1})=\mm_{a+1,1} \mm_{\alpha,1} \mm_{\beta,1}.
\end{eqnarray*}
We deduce from this formula that $\Cok(\ddd_{4,2/1})=0$ and $\Ker(\ddd_{4,2/1})$ is a free $\Sym_{\bQ}(\cE)$-module with basis:
\begin{equation}\label{eq:basis_Ker_D_4,2/1}
\aligned
\{[\mm_{\gamma-1,2}\mm_{\alpha,1}\mm_{\beta,1} - \mm_{\alpha-1,2}\mm_{\beta,1}\mm_{\gamma,1}], [\mm_{\gamma-1,2}\mm_{\alpha,1}\mm_{\beta,1} - \mm_{\beta-1,2}\mm_{\alpha,1}\mm_{\gamma,1}]; 1\leq \alpha < \beta < \gamma\}\\
\sqcup \{[\mm_{\gamma-1,2}\mm_{\alpha,1}\mm_{\alpha,1} - \mm_{\alpha-1,2}\mm_{\alpha,1}\mm_{\gamma,1}], [\mm_{\gamma-1,2}\mm_{\alpha,1}\mm_{\gamma,1} - \mm_{\alpha-1,2}\mm_{\gamma,1}\mm_{\gamma,1}]; 1\leq \alpha < \gamma\}.
\endaligned
\end{equation}
If $\gamma \geq 3$, using the relation \ref{(4)} of Proposition~\ref{prop:general_result_filtration}, we deduce from the above relations in $\Cok(\ddd_{4,1})$ that the connecting homomorphism $\delta_{3,1}$ sends each one of these generators to $0$ in $\Cok(\ddd_{4,1})$.
For $(\alpha, \gamma) = (1,2)$, by the relations \eqref{eq:rel_D_4,1} and \ref{(3)} of Proposition~\ref{prop:general_result_filtration}, we compute that $\delta_{3,1}(\mm_{1,2}\mm_{1,1}\mm_{1,1} - \mm_{0,2}\mm_{1,1}\mm_{2,1}) = -3e_{1}^2\mm_{0,3}$ and that $\delta_{3,1}(\mm_{1,2}\mm_{1,1}\mm_{2,1} - \mm_{0,2}\mm_{2,1}\mm_{2,1}) = 0$.
Thus $\Cok(\ddd_{4,2}) \cong  (\Sym_{\bQ}(\cE)/(e_{1}^2, e_{\alpha}; \alpha\geq2))[\mm_{0,3}]$.

\underline{Computation of $\Cok(\ddd_{4,3})$:}
since the target of $\ddd_{4,3/2}$ is zero, we deduce that $\Cok(\ddd_{4,3/2}) = 0$ and that $\Ker(\ddd_{4,3/2})$ is the free $\Sym_{\bQ}(\cE)$-module with basis $\{[\mm_{\alpha,1}\mm_{\beta,1}\mm_{\gamma,1}\mm_{\delta,1}];1\leq \alpha\leq\beta\leq\gamma\leq\delta \}$.
The connecting homomorphism $\delta_{4,3}$ maps each generator $\mm_{\alpha,1}\mm_{\beta,1}\mm_{\gamma,1}\mm_{\delta,1}$ to the class of $[e_{\alpha}\mm_{\beta,1} \mm_{\gamma,1}\mm_{\delta,1}+e_{\beta}\mm_{\alpha,1} \mm_{\gamma,1}\mm_{\delta,1}+e_{\gamma}\mm_{\alpha,1} \mm_{\beta,1}\mm_{\delta,1}+e_{\delta}\mm_{\alpha,1} \mm_{\beta,1}\mm_{\gamma,1}]$ in $\Cok(\ddd_{4,2}) \cong  (\Sym_{\bQ}(\cE)/(e_{1}^2, e_{\alpha}; \alpha\geq2))[\mm_{0,3}]$. Using \ref{(4)} and \ref{(5)} of Proposition~\ref{prop:general_result_filtration}, we deduce that each one of these generators vanish in $\Cok(\ddd_{4,3})$ if $\delta\geq3$, and it is equal to $(3\cdot2!-3!)[e_{\alpha}e_{\beta}e_{\gamma}\mm_{0,3}]= 0$ if $\delta=2$.
If $\delta=1$, the generator is equal to $4e_{1}[\mm_{1,1}\mm_{1,1}\mm_{1,1}] = -8e_{1}^2[\mm_{0,2}\mm_{1,1}]$, which vanishes by the relation \eqref{eq:rel_D_4,1}.
Hence there is no additional relation, and thus ends the proof.
\end{proof}
Fully determining the kernel of $\Ker(\ddd_{4})\cong H_{\st}^{\even}(\Lambda^{4}\tilde{H}_{\bQ})$ seems to be a difficult task. However, we recall from Corollary~\ref{coro:low_dim_paired} that $H_{\st}^{0}(\Lambda^{4}\tilde{H}_{\bQ})=H_{\st}^{2}(\Lambda^{4}\tilde{H}_{\bQ})=0$.
Also, we compute the $\Tor$-groups of $H_{\st}^{*}(\Lambda^{4}\tilde{H}_{\bQ})$ as follows.
\begin{prop}\label{prop:4th-ext_Tor}
For any $j > 0$, we have 
$$H_{j}(\Sym_{\bQ}(\cE);H_{\st}^{*}(\Lambda^{4}\tilde{H}_{\bQ}))\cong\Lambda^{j-1}\cE_{\geq 2}\oplus\Lambda^{j}\cE_{\geq 2}\oplus\Lambda^{j+1}\cE_{\geq 2}\oplus\Lambda^{j+2}\cE_{\geq 2}.$$
Moreover, we have 
\begin{align*}
H_{0}(\Sym_{\bQ}(\cE);H_{\st}^{*}(\Lambda^{4}\tilde{H}_{\bQ}))  \;\cong\;& \bQ\{[\mm_{a,2}\mm_{b,2}-\mm_{a-1,3}\mm_{b+1,1}-\mm_{b-1,3}\mm_{a+1,1}];2 \leq a \leq b\} \\  
& \oplus\bQ\{[\mm_{\gamma-1,2}\mm_{\alpha,1}\mm_{\beta,1} - \mm_{\alpha-1,2}\mm_{\beta,1}\mm_{\gamma,1}]
; 1\leq \alpha < \beta < \gamma\}
\\ 
& \oplus\bQ\{
[\mm_{\gamma-1,2}\mm_{\alpha,1}\mm_{\beta,1} - \mm_{\beta-1,2}\mm_{\alpha,1}\mm_{\gamma,1}]
; 1\leq \alpha < \beta < \gamma\}
\\ 
& \oplus\bQ\{[\mm_{\gamma-1,2}\mm_{\alpha,1}\mm_{\alpha,1} - \mm_{\alpha-1,2}\mm_{\alpha,1}\mm_{\gamma,1}]
; 1\leq \alpha < \gamma\geq 3\}\\
& \oplus\bQ\{ 
[\mm_{\gamma-1,2}\mm_{\alpha,1}\mm_{\gamma,1} - \mm_{\alpha-1,2}\mm_{\gamma,1}\mm_{\gamma,1}]
; 1\leq \alpha < \gamma\}\\
&\oplus\bQ\{[\mm_{\alpha,1}\mm_{\beta,1}\mm_{\gamma,1}\mm_{\delta,1}];1\leq \alpha\leq \beta\leq \gamma\leq \delta \} \oplus \bQ\oplus \cE_{\geq 2}\oplus\Lambda^{2}\cE_{\geq 2}.
\end{align*}
\end{prop}
\begin{proof}
By the computations \eqref{eq:computation_Tor_general} and Theorem~\ref{prop:4th-ext_odd}, we deduce that $H_{j}(\Sym_{\bQ}(\cE);H_{\st}^{\odd}(\Lambda^{4}\tilde{H}_{\bQ})) \cong\Lambda^{j-1}\cE_{\geq 2}\oplus\Lambda^{j}\cE_{\geq 2}$
for all $j \geq1$, and $H_{0}(\Sym_{\bQ}(\cE);H_{\st}^{\odd}(\Lambda^{4}\tilde{H}_{\bQ}))\cong \bQ$. Combining this computation with the isomorphism \eqref{eq:exterior_Tor} for $d=4$, we deduce that $H_{j}(\Sym_{\bQ}(\cE);H_{\st}^{\even}(\Lambda^{4}\tilde{H}_{\bQ}))\cong \Lambda^{j+2}\cE_{\geq 2}\oplus\Lambda^{j+1}\cE_{\geq 2}$, and thus ends the proof of the first part of the proposition.

For the second part of the proposition, we first have the contribution of $H_{0}(\Sym_{\bQ}(\cE);H_{\st}^{\odd}(\Lambda^{4}\tilde{H}_{\bQ}))$ computed above.
By Lemma~\ref{lem:Tor_0_lemma}, we compute $H_{0}(\Sym_{\bQ}(\cE);H_{\st}^{\even}(\Lambda^{4}\tilde{H}_{\bQ}))$ thanks to the exact sequence \eqref{eq:Tor_0_ES}. We recall that $H_{2}(\Sym_{\bQ}(\cE);H_{\st}^{\odd}(\Lambda^{4}\tilde{H}_{\bQ}))$ is computed above.
Following the decomposition \eqref{eq:decomposition_filtration_Ker_H_{0}}, we finish the proof by progressively determining the generators of $\Ker(\Delta_{4})$ associated to each summand $\Ker(\ddd_{4,k/k-1})$ for $1\leq k\leq 3$, via the chain complex \eqref{eq:complex_SES} and Proposition~\ref{prop:Ker_Delta_method}.

\underline{For the summand $\Ker(\ddd_{4, 1/0})$:} for each $1 \leq a \leq b$, the element $\ddd_{4}(\mm_{a-1,3}\mm_{b+1,1}+\mm_{b-1,3}\mm_{a+1,1}-\mm_{a,2}\mm_{b,2})$ is equal to $\pa_{1}(\mm_{a-1,3}\otimes e_{b+1} + \mm_{b-1,3}\otimes e_{a+1}) \in H_{\st}^{\odd}(\Lambda^{3}H_{\bQ})$. By Theorem~\ref{prop:4th-ext_odd}, we have $\Pi(\mm_{a-1,3})\otimes e_{b+1} + \Pi(\mm_{b-1,3})\otimes e_{a+1}= 0$ for all $2 \leq a \leq b$, while $\Pi(\mm_{0,3})\otimes e_{b+1} + \Pi(\mm_{b-1,3})\otimes e_{2}$ is equal to $2\Pi(\mm_{0,3})\otimes e_{2}$ for $b=1$, and $\Pi(\mm_{0,3})\otimes e_{b+1}$ for $b \geq 2$.
These elements are linearly independent in $H_{1}(\Sym_{\bQ}(\cE);H_{\st}^{\odd}(\Lambda^4\tilde{H}_Q))$.
Hence, by Proposition~\ref{prop:Ker_Delta_method}, the restriction of $\Ker(\Delta_{4})$ to $H_{0}({\Sym_{\bQ}(\cE)};\Ker(\ddd_{4,1/0}))$ is isomorphic to the free $\bQ$-module generated by $\{[\mm_{a,2}\mm_{b,2}-\mm_{a-1,3}\mm_{b+1,1}-\mm_{b-1,3}\mm_{a+1,1}];2 \leq a \leq b \}$.

\underline{For the summand $\Ker(\ddd_{4, 2/1})$:} for each $1\leq \alpha < \beta < \gamma$, we compute similarly to relations \eqref{eq:rel_D_4,1} that $\ddd_{4}([\mm_{\gamma-1,2}\mm_{\alpha,1}\mm_{\beta,1} - \mm_{\alpha-1,2}\mm_{\beta,1}\mm_{\gamma,1}])=\pa_{1}(-\mm_{\alpha-1,2}\mm_{\beta,1}\otimes e_{\gamma})$ and $\ddd_{4}([\mm_{\gamma-1,2}\mm_{\alpha,1}\mm_{\beta,1} - \mm_{\beta-1,2}\mm_{\alpha,1}\mm_{\gamma,1}])=\pa_{1}(-\mm_{\alpha-1,2}\mm_{\beta,1}\otimes e_{\gamma})$.
In $H_{\st}^{\odd}(\Lambda^4\tilde{H}_Q)$, we note that $\Pi(\mm_{\alpha-1,2}\mm_{\beta,1})= - \Pi(\mm_{\beta-1,2}\mm_{\alpha,1})$, we recall from \eqref{eq:rel_D_4,1} that we have $\Pi(\mm_{\alpha-1,2}\mm_{\beta,1})=0$ if $\beta \geq 3$ and $\Pi(\mm_{0,2}\mm_{2,1}) = -\Pi(e_{2}\mm_{0,3}) = 0$ for the remaining case $(\alpha, \beta) = (1, 2)$. Hence, by Proposition~\ref{prop:Ker_Delta_method}, all the classes $[\mm_{\gamma-1,2}\mm_{\alpha,1}\mm_{\beta,1} - \mm_{\alpha-1,2}\mm_{\beta,1}\mm_{\gamma,1}]$ and $[\mm_{\gamma-1,2}\mm_{\alpha,1}\mm_{\beta,1} - \mm_{\beta-1,2}\mm_{\alpha,1}\mm_{\gamma,1}]$ belong to $\Ker(\Delta_{4})$.

We now consider some $1\leq \alpha < \gamma$. For $\gamma \geq 3$, we compute similarly to relations \eqref{eq:rel_D_4,1} that $\ddd_{4}([\mm_{\gamma-1,2}\mm_{\alpha,1}\mm_{\alpha,1} - \mm_{\alpha-1,2}\mm_{\alpha,1}\mm_{\gamma,1}])=\pa_{1}(-\mm_{\alpha-1,2}\mm_{\alpha,1}\otimes e_{\gamma})$ and $\ddd_{4}([\mm_{\gamma-1,2}\mm_{\alpha,1}\mm_{\gamma,1} - \mm_{\alpha-1,2}\mm_{\gamma,1}\mm_{\gamma,1}])=0=\pa_{1}(0)$. Now $\Pi(\mm_{\alpha-1,2}\mm_{\alpha,1}) = 0$ for all $\alpha\geq 1$ by Theorem~\ref{prop:4th-ext_odd} and the relations \eqref{eq:rel_D_4,1}, so all these classes for $\gamma \geq 3$ belong to $\Ker(\Delta_{4})$ by Proposition~\ref{prop:Ker_Delta_method}.

The only remaining case is that of $(\alpha, \gamma) = (1, 2)$. We compute similarly to relations \eqref{eq:rel_D_4,1} that $\ddd_{4}(
[\mm_{1,2}\mm_{1,1}\mm_{1,1} - \mm_{0,2}\mm_{1,1}\mm_{2,1}])=\pa_{1}((2\mm_{1,2}\mm_{1,1}-\mm_{0,2}\mm_{2,1})\otimes e_{1}
- \mm_{0,2}\mm_{1,1}\otimes e_{2})$ and  
$\ddd_{4}([\mm_{1,2}\mm_{1,1}\mm_{2,1} - \mm_{0,2}\mm_{2,1}\mm_{2,1}])=\pa_{1}(\mm_{1,2}\mm_{2,1}\otimes e_{1} + (\mm_{1,2}\mm_{1,1} - 2\mm_{0,2}\mm_{2,1})\otimes e_{2})$. 
Here, similarly to relations \eqref{eq:rel_D_4,1} in $H_{\st}^{\odd}(\Lambda^{4}\tilde{H}_{\bQ})$, we compute that in $H_{1}(\Sym_{\bQ}(\cE);H_{\st}^{\odd}(\Lambda^4\tilde{H}_Q))$
\begin{align*}
&\Pi(2\mm_{1,2}\mm_{1,1}-\mm_{0,2}\mm_{2,1})\otimes e_{1} - \Pi(\mm_{0,2}\mm_{1,1})\otimes e_{2} \\ \;=\;& \Pi(-2e_{1}\mm_{0,3}+e_{2}\mm_{0,3})\otimes e_{1} - \Pi(\frac{1}{2}\ddd_{4}(\mm_{0,2}\mm_{0,2}))\otimes e_{2} \\
\;=\;& -2\Pi(e_{1}\mm_{0,3})\otimes e_{1}
\end{align*}
which is non-null because $\Pi(e_{1}\mm_{0,3})\otimes e_{1}$ cannot belong to the image of $\pa_{2}$ (thanks to the formal definition of that differential, see \eqref{eq:iD}), while
\begin{align*}
&\Pi(\mm_{1,2}\mm_{2,1})\otimes e_{1} + \Pi(\mm_{1,2}\mm_{1,1} - 2\mm_{0,2}\mm_{2,1})\otimes e_{2} \\
\;=\;& -\Pi(e_{2}\mm_{0,3})\otimes e_{1} + \Pi(-e_{1}\mm_{0,3} +  2e_{1}\mm_{0,3})\otimes e_{2} \\
\;=\;& \Pi(e_{1}\mm_{0,3})\otimes e_{2} \\
\;=\;& -\Pi(e_{2}\mm_{0,3})\otimes e_{1} = 0.
\end{align*}
Therefore all the classes introduced above of the type $[\mm_{\gamma-1,2}\mm_{\alpha,1}\mm_{\alpha,1} - \mm_{\alpha-1,2}\mm_{\alpha,1}\mm_{\gamma,1}]$ and  
$[\mm_{\gamma-1,2}\mm_{\alpha,1}\mm_{\gamma,1} - \mm_{\alpha-1,2}\mm_{\gamma,1}\mm_{\gamma,1}]$,
except $[\mm_{1,2}\mm_{1,1}\mm_{1,1} - \mm_{0,2}\mm_{1,1}\mm_{2,1}]$, are in the kernel 
$\Ker(\Delta_{4})$. Hence, by Proposition~\ref{prop:Ker_Delta_method}, the restriction of $\Ker(\Delta_{4})$ to $H_{0}({\Sym_{\bQ}(\cE)};\Ker(\ddd_{4,2/1}))$ is isomorphic to the free $\bQ$-module generated by all the above classes.

\underline{For the summand $\Ker(\ddd_{4,3/2})$:} for all $1\leq\alpha\leq\beta\leq\gamma\leq\delta$, 
the element $\ddd_{4}(\mm_{\alpha,1}\mm_{\beta,1}\mm_{\gamma,1}\mm_{\delta,1})$ is equal to $\pa_{1}(\mm_{\alpha,1}\mm_{\beta,1}\mm_{\gamma,1}\otimes e_{\delta} +\mm_{\alpha,1}\mm_{\beta,1}\mm_{\delta,1}\otimes e_{\gamma}
+\mm_{\alpha,1}\mm_{\gamma,1}\mm_{\delta,1}\otimes e_{\beta}
+\mm_{\beta,1}\mm_{\gamma,1}\mm_{\delta,1}\otimes e_{\alpha}\in H_{\st}^{\odd}(\Lambda^{3}H_{\bQ})\otimes\cE)$.

Our further computations are based on the relations proved in the paragraph on the computation of $\Cok(\ddd_{4,3})$ in the proof of Theorem~\ref{prop:4th-ext_odd}.
If $\gamma \geq 3$, then the image by $\Pi\otimes \cE$ of this class vanishes in $H_{1}(\Sym_{\bQ}(\cE);H_{\st}^{\odd}(\Lambda^4\tilde{H}_Q)$. Since we have $\Pi(\mm_{\alpha,1}\mm_{\beta,1}\mm_{2,1}) = \Pi(2e_{\alpha} e_{\beta}\mm_{0,3}) = 0$, the image by $\Pi\otimes \cE$ of this class also vanishes when $\gamma = 2$. 
Finally if $\gamma = 1$, then  $\alpha = \beta = 1$. Hence the image by $\Pi\otimes \cE$ of this class is equal to $4\Pi(\mm_{1,1}\mm_{1,1}\mm_{1,1})\otimes e_{1}$ for $\delta= 1$, and $\Pi(\mm_{1,1}\mm_{1,1}\mm_{1,1})\otimes e_\delta$ for $\delta \geq 2$. Since $\Pi(\mm_{1,1}\mm_{1,1}\mm_{1,1}) = \Pi(2e_{1}^{2}\mm_{0,3}) = 0$, it follows from Proposition~\ref{prop:Ker_Delta_method} that the restriction of $\Ker(\Delta_{4})$ to $H_{0}({\Sym_{\bQ}(\cE)};\Ker(\ddd_{4,3/2}))$ is isomorphic to the free $\bQ$-module generated by $\{[\mm_{\alpha,1}\mm_{\beta,1}\mm_{\gamma,1}\mm_{\delta,1}];1\leq\alpha\leq\beta\leq\gamma\leq\delta \}$.
\end{proof}

\subsection{Fifth exterior power}\label{ss:fifth_exterior}

We consider the exact sequence \eqref{ES_cohomology_exterior} for $d=5$ and apply the contraction formulas of Proposition~\ref{prop:Contraction_Formula} to compute the images of the derivation $\ddd_{5}$. We introduce the following notations in order to compute $H_{\st}^{\even}(\Lambda^{5}\tilde{H}_{\bQ})$.

\begin{notation}\label{not:Lambda_5}
We denote by $\fL$ the torsion $\Sym_{\bQ}(\cE)$-module
$$\begin{array}{c}
(\Sym_{\bQ}(\cE)/(e_{1}^{3}, e_{\alpha}e_{\beta},e_{\gamma};\alpha,\beta\geq1\,\, \text{except $\alpha = \beta = 1$}, \gamma\geq4))\{\mm_{0,4};\mm_{0,2}\mm_{0,2}\}\\
\oplus(\Sym_{\bQ}(\cE)/(
e_{\alpha}, e_{\beta} e_{\gamma} ; \alpha\geq3, \beta, \gamma \geq 1))\{\mm_{0,3}\mm_{1,1}\}
\\\oplus(\Sym_{\bQ}(\cE)/(e_{\alpha}, e_{\beta} e_{\gamma} ;\alpha\geq2, \beta, \gamma \geq 1))\{\mm_{0,3}\mm_{2,1}\}.
\end{array}$$
Also, we denote by $\fK$ the submodule of $\fL$ defined by the direct sum of the trivial $\Sym_{\bQ}(\cE)$-modules $\{2e_{2}\mm_{0,3}\mm_{1,1}+3e_{1}\mm_{0,3}\mm_{2,1}\}$, $\{e_{2}\mm_{0,2}\mm_{0,2}+6e_{1}\mm_{0,3}\mm_{1,1}\}$ and $\{e_{3}\mm_{0,2}\mm_{0,2}-6e_{1}^{2}\mm_{0,4}\}$.
\end{notation}

\begin{thm}\label{thm:5th-ext-even}
The $\Sym_{\bQ}(\cE)$-module $H_{\st}^{\even}(\Lambda^{5}\tilde{H}_{\bQ})$ is isomorphic to the quotient $\fL/\fK$.
\end{thm}

\begin{proof}
We progressively compute $\Cok(\ddd_{5,k})$ for $0\leq k\leq 4$ using the exact sequence \eqref{eq:snake}.

\underline{Computation of $\Cok(\ddd_{5,1})$:}
we compute that for $a,b\geq 0$ and $\alpha\geq1$:
\begin{eqnarray*}
   &&\ddd_{5,1/0}(\mm_{a,4}\mm_{\alpha,1})=\mm_{a+1,3}\mm_{\alpha,1}\,;\\&&\ddd_{5,1/0}(\mm_{0,3}\mm_{b,2})=\mm_{1,2}\mm_{b,2}+\mm_{0,3}\mm_{b+1,1}\,;\\&&\ddd_{5,1/0}(\mm_{a+1,3}\mm_{b,2} - \mm_{a,4}\mm_{b+1,1}) = \mm_{a+2,2}\mm_{b,2}.
\end{eqnarray*}
Hence $\{[\mm_{0,2}\mm_{0,2}];
[\mm_{0,2}\mm_{1,2}] = -[\mm_{0,3}\mm_{1,1}]; 
[\mm_{1,2}\mm_{1,2}] = -[\mm_{0,3}\mm_{2,1}]\}$ 
is a $\Sym_{\bQ}(\cE)$-free basis of 
$\Cok(\ddd_{5,1/0})$. We compute that for $a, b\geq 2$:
$$\ddd_{5,1/0}(\mm_{a-1,3}\mm_{b,2}-\mm_{a-2,4}\mm_{b+1,1}) 
= \mm_{a,2}\mm_{b,2} = 
\ddd_{5,1/0}(\mm_{b-1,3}\mm_{a,2}-\mm_{b-2,4}\mm_{a+1,1}).$$
It follows from the structure of the filtration $F_{1}H_{\st}^{*}(\Lambda^{5}{H}_{\bQ})$ and the formulas for $\ddd_{5,1/0}$ that the above equality is the only possibility to get an element of $\Ker(\ddd_{5,1/0})$. Hence $\{\mm_{a-1,3}\mm_{b,2} - \mm_{a-2,4}\mm_{b+1,1}
- \mm_{b-1,3}\mm_{a,2} + \mm_{b-2,4}\mm_{a+1,1};
2\leq a < b\}$ is a $\Sym_{\bQ}(\cE)$-free basis of 
$\Ker(\ddd_{5,1/0})$. By the formal definition of the connecting homomorphism defined by the snake lemma, each one of these generators is mapped by $\delta_{5,0}$ to $[e_{a+1}\mm_{b-2,4} - e_{b+1}\mm_{a-2,4}]$ in $\Cok(\ddd_{5,0})$. Also, we recall from \ref{(2)} in Proposition~\ref{prop:general_result_filtration} that $\Cok(\ddd_{5,0})\cong\Sym_{\bQ}(\cE)[\mm_{0,4}]$. Hence we have $\mathrm{Im}(\delta_{5,0}) = (e_{\alpha}; \alpha\geq4)[\mm_{0,4}]$ and thus $\Cok(\ddd_{5,1})
\cong (\Sym_{\bQ}(\cE))^{\oplus 3}\oplus \Sym_{\bQ}(\cE)[\mm_{0,4}]/(e_{\alpha}; \alpha\geq4)$ where the first three summands correspond to $\Cok(\ddd_{5,1/0})$.

\underline{Computation of $\Cok(\ddd_{5,2})$:}
we compute that for $0\leq a\leq b$ and $\alpha,\beta\geq1$:
\begin{eqnarray*}
   &&\ddd_{5,2/1}(\mm_{a,3}\mm_{\alpha,1}\mm_{\beta,1}) = \mm_{a+1,2}\mm_{\alpha,1}\mm_{\beta,1}\,;\\&&\ddd_{5,2/1}(\mm_{0,2}\mm_{0,2}\mm_{\alpha,1}) = 2\mm_{0,2}\mm_{1,1}\mm_{\alpha,1}\,;\\
   &&\ddd_{5,2/1}(\mm_{0,2}\mm_{b+1,2}\mm_{\alpha,1}-\mm_{b,3}\mm_{1,1}\mm_{\alpha,1}) = \mm_{0,2}\mm_{b+2,1}\mm_{\alpha,1}\,;\\
   &&\ddd_{5,2/1}(\mm_{a+1,2}\mm_{b+1,2}\mm_{\alpha,1}-\mm_{a,3}\mm_{b+2,1}\mm_{\alpha,1}-\mm_{b,3}\mm_{a+2,1}\mm_{\alpha,1} ) = 0.
\end{eqnarray*}
It follows from these formulas shows that $\Cok(\ddd_{5,2/1}) = 0$, and that $\Ker(\ddd_{5,2/1})$ is a $\Sym_{\bQ}(\cE)$-free module with basis
\begin{eqnarray*}
&&\{\mm_{a+1,2}\mm_{b+1,2}\mm_{\alpha,1} 
   -\mm_{a,3}\mm_{b+2,1}\mm_{\alpha,1}
   -\mm_{b,3}\mm_{a+2,1}\mm_{\alpha,1};\,\, 0 \leq a\leq b, \alpha \geq 1\}\\
&&\sqcup \{\mm_{0,2}\mm_{b+1,2}\mm_{c+2,1}-\mm_{b,3}\mm_{1,1}\mm_{c+2,1}
-\mm_{0,2}\mm_{c+1,2}\mm_{b+2,1}
+\mm_{c,3}\mm_{1,1}\mm_{b+2,1}; \,\, 0 \leq b < c\}
\\
&&\sqcup \{\mm_{0,2}\mm_{0,2}\mm_{b+2,1}
-2\mm_{0,2}\mm_{b+1,2}\mm_{1,1}+2\mm_{b,3}\mm_{1,1}\mm_{1,1}; \,\, b \geq 0\}.
\end{eqnarray*}
Then the connecting homomorphism $\delta_{5,1}$ respectively maps each one of these generators to the classes in $\Cok(\ddd_{5,1})$ of following elements
\begin{eqnarray*}
   &&e_{\alpha}(\mm_{a+1,2}\mm_{b+1,2}- \mm_{a,3}\mm_{b+2,1} - \mm_{b,3}\mm_{a+2,1}) - e_{b+2}\mm_{a,3}\mm_{\alpha,1} - e_{a+2}\mm_{b,3}\mm_{\alpha,1};
   \\&& e_{1}(\mm_{c,3}\mm_{b+2,1}-\mm_{b,3}\mm_{c+2,1}) 
   - e_{b+2}(\mm_{0,2}\mm_{c+1,2} -  \mm_{c,3}\mm_{1,1}) 
   + e_{c+2}(\mm_{0,2}\mm_{b+1,2} - \mm_{b,3}\mm_{1,1});
   \\&& e_{b+2}\mm_{0,2}\mm_{0,2} + 2e_{1}(2\mm_{b,3}\mm_{1,1} - \mm_{0,2}\mm_{b+1,2}).
\end{eqnarray*}
We denote these elements by $M_{a,b,\alpha}$, $N_{b,c}$ and $P_{b}$ respectively. Using \ref{(3)} and \ref{(4)} in Proposition~\ref{prop:general_result_filtration}, we compute them explicitly in $\Cok(\ddd_{5,1})$ to determine further relations in $\Cok(\ddd_{5,2})$ as follows.
\begin{equation}\label{eq:M_a_b_alpha}
M_{a,b,\alpha} =
\begin{cases}
0 & \text{if $2 \leq a (\leq b)$},\\
(4+\delta_{a,b})e_{\alpha} e_{b+2}\mm_{0,4} & \text{if $1 = a (\leq b)$},\\
-e_{b+2}\mm_{0,3}\mm_{\alpha,1} & \text{if $a = 0$ and $b \geq 2$},\\
4e_{\alpha} e_{2}\mm_{0,4}-e_{3}\mm_{0,3}\mm_{\alpha,1} & \text{if $a = 0$ and $b = 1$},\\
-3e_{\alpha}\mm_{0,3}\mm_{2,1} -2e_{2}\mm_{0,3}\mm_{\alpha,1} & \text{if $a = b = 0$},
\end{cases}
\end{equation}
for $0 \leq a\leq b, \alpha \geq 1$, 
\begin{equation}\label{eq:N_b_c}
N_{b,c}=
\begin{cases}
0 & \text{if $2\leq b(<c)$},\\
2e_{1}e_{c+2}\mm_{0,4} & \text{if $1=b (<c)$},\\
-2e_{c+2}\mm_{0,3}\mm_{1,1} & \text{if $b=0$, $c \geq 2$},\\
-2e_{1}e_{2}\mm_{0,4} - 2e_{3}\mm_{0,3}\mm_{1,1} & \text{if $b=0$, $c=1$},
\end{cases}
\end{equation}
for $0 \leq b < c$, and finally for $b\geq 0$: 
\begin{equation}\label{eq:P_b}
P_b=
\begin{cases}
e_{b+2}\mm_{0,2}\mm_{0,2} & \text{if $b\geq 2$},\\
e_{3}\mm_{0,2}\mm_{0,2} - 6e_{1}^2\mm_{0,4} & \text{if $b=1$},\\
e_{2}\mm_{0,2}\mm_{0,2} + 6e_{1}\mm_{0,3}\mm_{1,1} & \text{if $b = 0$}.
\end{cases}
\end{equation}
Furthermore, we deduce from \ref{(3)} and \ref{(4)} in Proposition~\ref{prop:general_result_filtration} that:
$$[\mm_{0,3}\mm_{\alpha,1}] = 
\begin{cases}
0 & \text{if $\alpha \geq 4$},\\
-e_{2}[\mm_{0,4}] & \text{if $\alpha = 3$},\\
-[\mm_{1,2}\mm_{1,2}] & \text{if $\alpha = 2$},\\
-[\mm_{0,2}\mm_{1,2}] & \text{if $\alpha = 1$}.
\end{cases}$$
Combining all the above relations, we deduce that $\Cok(\ddd_{5,2})$ is isomorphic to the quotient of $\Cok(\ddd_{5,1})$ by the following relations where we consider integers $\ell\geq 4$ and $\alpha\geq 1$:
\begin{equation}\label{eq:D53relation}
\aligned
&e_{\ell}[\mm_{0,3}\mm_{2,1}]=e_{\ell}[\mm_{0,3}\mm_{1,1}]=e_{\ell}[\mm_{0,2}\mm_{0,2}]=0\,;\\
&e_{2}^{2}[\mm_{0,4}]=e_{1}e_{2}[\mm_{0,4}]=e_{\alpha}e_{3}[\mm_{0,4}]=0\,;\\
&e_{3}[\mm_{0,3}\mm_{2,1}]=e_{3}[\mm_{0,3}\mm_{1,1}]=e_{2}[\mm_{0,3}\mm_{2,1}]=0\,;\\
&3e_{1}[\mm_{0,3}\mm_{2,1}]=-2e_{2}[\mm_{0,3}\mm_{1,1}]\,;\\
&e_{3}[\mm_{0,2}\mm_{0,2}] = 6e_{1}^2[\mm_{0,4}]\,;\\
&e_{2}[\mm_{0,2}\mm_{0,2}] = - 6e_{1}[\mm_{0,3}\mm_{1,1}].
\endaligned
\end{equation}
In addition, we deduce from the combinations of the relations \eqref{eq:D53relation} that for all $\alpha\geq2$ and $\beta\geq3$:
\begin{equation}\label{eq:D53relation2}
\aligned
& e_{2}e_{\alpha}[\mm_{0,3}\mm_{1,1}]= -\tfrac32 e_{\alpha} e_{1}[\mm_{0,3}\mm_{2,1}]= 0,\\
& e_{3}e_{\alpha}[\mm_{0,2}\mm_{0,2}] = 6e_{1}^2e_{\alpha} [\mm_{0,4}] = 0,\\
& e_{2}e_{\beta}[\mm_{0,2}\mm_{0,2}] = -6e_{1}e_{\beta}[\mm_{0,3}\mm_{1,1}] = 0.
\endaligned
\end{equation}

\underline{Computation of $\Cok(\ddd_{5,3})$:} 
we compute that for $a\geq 0$ and $\alpha,\beta,\gamma\geq1$: $$\ddd_{5,3/2}(\mm_{a,2} \mm_{\alpha,1} \mm_{\beta,1}\mm_{\gamma,1})=\mm_{a+1,1}\mm_{\alpha,1}\mm_{\beta,1}\mm_{\gamma,1}.$$
Hence the derivation $\ddd_{5,3/2}$ has a section which maps each element $\mm_{\alpha,1}\mm_{\beta,1}\mm_{\gamma,1}\mm_{\delta,1}$ where $\alpha,\beta,\gamma,\delta\geq1$
to the element $\tfrac14(\mm_{\alpha-1,2}\mm_{\beta,1}\mm_{\gamma,1}\mm_{\delta,1}
+\mm_{\alpha,1}\mm_{\beta-1,2}\mm_{\gamma,1}\mm_{\delta,1}
+\mm_{\alpha,1}\mm_{\beta,1}\mm_{\gamma-1,2}\mm_{\delta,1}
+\mm_{\alpha,1}\mm_{\beta,1}\mm_{\gamma,1}\mm_{\delta-1,2})$.
Therefore, we have $\Cok(\ddd_{5,3/2})= 0$ and $\Ker(\ddd_{5,3/2})$ is a free $\Sym_{\bQ}(\cE)$-module with basis
$$\left\{3\mm_{a,2}\mm_{\alpha,1}\mm_{\beta,1}\mm_{\gamma,1}-\mm_{a+1,1}Q_{\alpha,\beta,\gamma};a\geq0, 1\leq \alpha \leq \beta  \leq \gamma \right\}$$
where $Q_{\alpha,\beta,\gamma}:=\mm_{\alpha-1,2}\mm_{\beta,1}\mm_{\gamma,1}+\mm_{\alpha,1}\mm_{\beta-1,2}\mm_{\gamma,1}+\mm_{\alpha,1}\mm_{\beta,1}\mm_{\gamma-1,2}$.

From now on, we fix some $1\leq\alpha \leq\beta \leq\gamma$.
From the relations \ref{(3)}, \ref{(4)} and \eqref{eq:key_rel_general_result_filtration} of Proposition~\ref{prop:general_result_filtration} and from the above computation of $\Cok(\ddd_{5,2})$, we have 
\begin{equation}\label{eqn:Qabc}
Q_{\alpha,\beta,\gamma}=\begin{cases}
0 & \text{if $\gamma \geq 4$},\\
-4e_{\alpha} e_{\beta}\mm_{0,4} & \text{if $\gamma = 3$},\\
2\mm_{0,3}(e_{\alpha}\mm_{\beta,1}+ e_{\beta}\mm_{\alpha,1}) & \text{if $\gamma = 2$},\\
-\frac{3}{2}e_{1}\mm_{0,2}\mm_{0,2}  & \text{if $\alpha = \beta = \gamma = 1$},
\end{cases}
\end{equation}

Then, the connecting homomorphism $\delta_{5,2}$ maps each $3\mm_{a,2}\mm_{\alpha,1}\mm_{\beta,1}\mm_{\gamma,1}-\mm_{a+1,1}Q_{\alpha,\beta,\gamma}$ to the class in $\Cok(\ddd_{5,2})$ of the following element:
\begin{eqnarray*}
   &R_{a,\alpha,\beta,\gamma}:=&3\mm_{a,2}(e_{\alpha}\mm_{\beta,1}\mm_{\gamma,1}+e_{\beta}\mm_{\gamma,1}\mm_{\alpha,1}+e_{\gamma}\mm_{\alpha,1}\mm_{\beta,1})\\
   &&-\mm_{a+1,1}(e_{\beta}\mm_{\alpha-1,2}\mm_{\gamma,1}+ e_{\gamma}\mm_{\alpha-1,2}\mm_{\beta,1}+ e_{\gamma}\mm_{\beta-1,2}\mm_{\alpha,1}\\
   &&\,\,\,\,\,\,\,\,\,\,\,\,\,\,\,\,\,\,\,\,\,\,\,\,\,+e_{\alpha}\mm_{\beta-1,2}\mm_{\gamma,1}+ e_{\alpha}\mm_{\gamma-1,2}\mm_{\beta,1}+ e_{\beta}\mm_{\gamma-1,2}\mm_{\alpha,1})\\
   &&-e_{a+1}Q_{\alpha,\beta,\gamma}.
\end{eqnarray*}
Since we work in $\Cok(\ddd_{5,2})$ and as $\mm_{0,2}\ddd_{4}(e_{\alpha}\mm_{\beta-1,2}\mm_{\gamma-1,2}
+ e_{\beta}\mm_{\gamma-1,2}\mm_{\alpha-1,2}
+ e_{\gamma}\mm_{\alpha-1,2}\mm_{\beta-1,2})$ belongs to $F_{2}H^{*}_{st}(\Lambda^5H_{\bQ})$, we have:
$$\aligned
&R_{a, \alpha, \beta, \gamma} + e_{a+1}Q_{\alpha, \beta, \gamma}\\
= & ~ 3\mm_{a,2}(e_{\alpha}\mm_{\beta,1}\mm_{\gamma,1}+
e_{\beta}\mm_{\gamma,1}\mm_{\alpha,1}+
e_{\gamma}\mm_{\alpha,1}\mm_{\beta,1})\\
& 
- \ddd_{2}(\mm_{a,2})\ddd_{4}(e_{\alpha}\mm_{\beta-1,2}\mm_{\gamma-1,2}
+ e_{\beta}\mm_{\gamma-1,2}\mm_{\alpha-1,2}
+ e_{\gamma}\mm_{\alpha-1,2}\mm_{\beta-1,2})\\
= & ~ 3\mm_{a,2}(e_{\alpha}\mm_{\beta,1}\mm_{\gamma,1}+
e_{\beta}\mm_{\gamma,1}\mm_{\alpha,1}+
e_{\gamma}\mm_{\alpha,1}\mm_{\beta,1})\\
& 
+ \mm_{a,2}(\ddd_{3}\circ\ddd_{4})(e_{\alpha}\mm_{\beta-1,2}\mm_{\gamma-1,2}
+ e_{\beta}\mm_{\gamma-1,2}\mm_{\alpha-1,2}
+ e_{\gamma}\mm_{\alpha-1,2}\mm_{\beta-1,2})\\
= & 5\mm_{a,2}(e_{\alpha}\mm_{\beta,1}\mm_{\gamma,1}+
e_{\beta}\mm_{\gamma,1}\mm_{\alpha,1}+
e_{\gamma}\mm_{\alpha,1}\mm_{\beta,1})\\
& 
+ 2\mm_{a,2}(e_{\alpha} e_{\beta}\mm_{\gamma-1,2}
+ e_{\beta} e_{\gamma}\mm_{\alpha-1,2}
+ e_{\gamma} e_{\alpha}\mm_{\beta-1,2})\\
= & ~ \mm_{a,2}S_{\alpha, \beta, \gamma},
\endaligned
$$
where $S_{\alpha, \beta, \gamma} := 5(e_{\alpha}\mm_{\beta,1}\mm_{\gamma,1}+
e_{\beta}\mm_{\gamma,1}\mm_{\alpha,1}+
e_{\gamma}\mm_{\alpha,1}\mm_{\beta,1})
+ 2(e_{\alpha} e_{\beta}\mm_{\gamma-1,2}
+ e_{\beta} e_{\gamma}\mm_{\alpha-1,2}
+ e_{\gamma} e_{\alpha}\mm_{\beta-1,2})$.
Hence, recalling that $\ddd_{2-a+1}\circ\cdots \circ \ddd_{2}$ is denoted by $\ddd^{a}$ for $a \geq 1$, we deduce from the general formula \eqref{eq:key_rel_general_result_filtration} that $R_{a, \alpha, \beta, \gamma} + e_{a+1}Q_{\alpha, \beta, \gamma}
= (-1)^a\mm_{0,a+2}\ddd^a(S_{\alpha, \beta, \gamma})$ and thus by the relations \eqref{eq:D53relation} that:
$$R_{a, \alpha, \beta, \gamma} + e_{a+1}Q_{\alpha, \beta, \gamma}
= \begin{cases}
0 & \text{if $a \geq 3$}, \\
36e_{\alpha} e_{\beta} e_{\gamma} [\mm_{0,4}] & \text{if $a = 2$}, \\
-12\mm_{0,3}(e_{\alpha} e_{\beta} \mm_{\gamma,1}
+ e_{\beta} e_{\gamma} \mm_{\alpha,1}
+ e_{\gamma} e_{\alpha} \mm_{\beta,1})
&  \text{if $a = 1$.}
\end{cases}$$
We recall that we assume $1 \leq \alpha \leq \beta \leq \gamma$. All the following computations for $R_{a,\alpha,\beta,\gamma}$ are done thanks to the above computation of $\Cok(\ddd_{5,2})$, in particular using the relations in \eqref{eq:D53relation} and \eqref{eq:D53relation2}, and the equalities \ref{(3)}, \ref{(4)} and \eqref{eq:key_rel_general_result_filtration} of Proposition~\ref{prop:general_result_filtration}. For $a \geq 3$, we have $e_{a+1}Q_{\alpha,\beta, \gamma} = 0$, and therefore $R_{a, \alpha, \beta, \gamma} = 0$.
For $a=2$, we compute that
\begin{equation}\label{eq:R_2_alpha_beta_gamma}
R_{2,\alpha,\beta,\gamma} = 36e_{\alpha} e_{\beta} e_{\gamma}\mm_{0,4} - e_{3}Q_{\alpha,\beta, \gamma} = 
\begin{cases}
45 e_{1}^3 [\mm_{0,4}] & \text{if $\alpha = \beta = \gamma = 1$, }\\
0 & \text{otherwise.}
\end{cases}
\end{equation}
For $a = 1$, we first compute that:
$$-12\mm_{0,3}(e_{\alpha} e_{\beta} \mm_{\gamma,1}
+ e_{\beta} e_{\gamma} \mm_{\alpha,1}
+ e_{\gamma} e_{\alpha} \mm_{\beta,1})
= \begin{cases}
-16e_{1}e_{2}[\mm_{0,3}\mm_{1,1}] & \text{if $\alpha =\beta=1$, $\gamma = 2$,}\\
-36e_{1}^2[\mm_{0,3}\mm_{1,1}] & \text{if $\alpha = \beta = \gamma = 1$,}\\
0 & \text{otherwise,}
\end{cases}$$
Then, since $R_{1,\alpha,\beta,\gamma} = -12\mm_{0,3}(e_{\alpha} e_{\beta} \mm_{\gamma,1}
+ e_{\beta} e_{\gamma} \mm_{\alpha,1}
+ e_{\gamma} e_{\alpha} \mm_{\beta,1}) - e_{2}Q_{\alpha,\beta,\gamma}$, we deduce that:
\begin{equation}\label{eq:R_1_alpha_beta_gamma}
R_{1,\alpha,\beta,\gamma} =\begin{cases}
-20e_{1}e_{2}[\mm_{0,3}\mm_{1,1}] =30e_{1}^2[\mm_{0,3}\mm_{2,1}]
 & \text{if $\alpha =\beta=1$, $\gamma = 2$, }\\
- 45 e_{1}^2[\mm_{0,3}\mm_{1,1}] & \text{if $\alpha = \beta = \gamma = 1$,}\\
0 & \text{otherwise.}
\end{cases}
\end{equation}

For $a=0$, we first compute that:
$$\mm_{0,2}S_{\alpha, \beta, \gamma}
= \begin{cases}
-19e_{1}^3[\mm_{0,4}] & \text{if $\alpha=\beta=1$, $\gamma = 3$,}\\
\frac{62}3e_{1}e_{2}[\mm_{0,3}\mm_{1,1}] & \text{if $\alpha =1$, $\beta = \gamma = 2$, }\\
19 e_{1}^2[\mm_{0,3}\mm_{1,1}] & \text{if $\alpha =\beta=1$, $\gamma = 2$, }\\
-\frac{3}{2} e_{1}^{2}[\mm_{0,2}\mm_{0,2}] & \text{if $\alpha = \beta = \gamma = 1$}\\
0 & \text{otherwise.}
\end{cases}$$
Therefore, we obtain that:
\begin{equation}\label{eq:R_0_alpha_beta_gamma}
R_{0,\alpha,\beta,\gamma} 
=\mm_{0,2}S_{\alpha, \beta, \gamma}- e_{1}Q_{\alpha, \beta,\gamma}
=\begin{cases}
-15e_{1}^3[\mm_{0,4}] & \text{if $\alpha=\beta=1$, $\gamma = 3$,}\\
20 e_{1}e_{2}[\mm_{0,3}\mm_{1,1}]
& \text{if $\alpha =1$, $\beta = \gamma = 2$,}\\
15 e_{1}^2[\mm_{0,3}\mm_{1,1}]
& \text{if $\alpha =\beta=1$, $\gamma = 2$, }\\
0 & \text{otherwise.}
\end{cases}
\end{equation}

Furthermore, by \ref{(5)} in Proposition~\ref{prop:general_result_filtration}, we have $e_{\alpha} e_{\beta} e_{\gamma} [\mm_{0,4}] = 0 \in 
\Cok(\ddd_{5,3})$ for any $\alpha, \beta, \gamma \geq 1$, 
and 
$e_{\alpha} e_{\beta} u = 0\in \Cok(\ddd_{5,3})$ except for $\alpha = \beta = 1$.
In particular, we have $3e_{1}^2[\mm_{0,3}\mm_{2,1}] = -2e_{1}e_{2}[\mm_{0,3}\mm_{1,1}]= 0$.

Therefore, combining all these new relations to \eqref{eq:D53relation} and \eqref{eq:D53relation2}, $\Cok(\ddd_{5,3})$ is isomorphic to the quotient of $\Cok(\ddd_{5,1})$ by the following relations where we consider integers $\ell\geq 4$ and $\alpha,\beta\geq 1$:
\begin{eqnarray*}
&& e_{\ell}[\mm_{0,3}\mm_{2,1}]=e_{\ell}[\mm_{0,3}\mm_{1,1}]=e_{\ell}[\mm_{0,2}\mm_{0,2}]=0\,;\\
&& e_{1}^{3}[\mm_{0,4}]=e_{\alpha}e_{\beta}[\mm_{0,4}]=0 \,\,{\text{except $\alpha = \beta=1$}}\,;\\
&& e_{3}[\mm_{0,3}\mm_{2,1}]=e_{2}[\mm_{0,3}\mm_{2,1}]=e_{\alpha} e_{\beta}[\mm_{0,3}\mm_{2,1}]=0\,;\\
&& e_{3}[\mm_{0,3}\mm_{1,1}]
=e_{\alpha} e_{\beta}[\mm_{0,3}\mm_{1,1}]=0\,;\\
&& e_{\alpha}e_{\beta}[\mm_{0,2}\mm_{0,2}]=0 \,\, {\text{except $\alpha = \beta=1$}}\,;\\&&3e_{1}[\mm_{0,3}\mm_{2,1}]=-2e_{2}[\mm_{0,3}\mm_{1,1}]\,;\\&&e_{3}[\mm_{0,2}\mm_{0,2}] = 6e_{1}^2[\mm_{0,4}]\,;\\&&e_{2}[\mm_{0,2}\mm_{0,2}] = - 6e_{1}[\mm_{0,3}\mm_{1,1}].
\end{eqnarray*}

\underline{Computation of $\Cok(\ddd_{5,4})$:}
since the target of $\ddd_{5,4/3}$ is zero, we have $\Cok(\ddd_{5,4/3}) = 0$ and 
$\Ker(\ddd_{5,4/3})$ is generated by the elements of type
$\mm_{\alpha_1,1}\mm_{\beta,1}\mm_{\gamma,1}
\mm_{\epsilon,1}\mm_{\xi,1}$.
The connecting homomorphism $\delta_{5,3}$ maps them to 
$\frac1{24}\sum_{\sigma\in\mathfrak{S}_{5}}e_{\sigma(\alpha)}\mm_{\sigma(\beta),1}\mm_{\sigma(\gamma),1}\mm_{\sigma(\epsilon),1}\mm_{\sigma(\xi),1} \in \Cok(\ddd_{5,3})$.
If $\epsilon = \max(\alpha, \beta, \gamma,\epsilon)$, then we compute from \ref{(4)} in Proposition~\ref{prop:general_result_filtration} that 
\begin{equation}\label{eq:cok_5/4}
[\mm_{\alpha,1}\mm_{\beta,1}\mm_{\gamma,1}\mm_{\epsilon,1}] = \begin{cases}
0 & \text{if $\epsilon \geq 3$, }\\
2e_{\alpha} e_{\beta}[\mm_{0,3}\mm_{\gamma,1}]+ 2 e_{\beta} e_{\gamma} [\mm_{0,3}\mm_{\alpha,1}] +
2e_{\gamma} e_{\alpha} [\mm_{0,3}\mm_{\beta,1}] & \text{if $\epsilon = 2$, }\\
\tfrac32e_{1}^2[\mm_{0,2}\mm_{0,2}] & \text{if $\epsilon = 1$.}\\
\end{cases}    
\end{equation}
In particular, we have $e_{1}^3\mm_{0,2}\mm_{0,2}
= 0 \in \Cok(\ddd_{5,4})$.
This provides the last additional relations in $\Cok(\ddd_{5,4})=H^{\mathrm{even}}(\Lambda^{5}\tilde{H}_{\bQ})$, which ends the proof.
\end{proof}

Computing $\Ker(\ddd_{5})\cong H_{\st}^{\odd}(\Lambda^{5}\tilde{H}_{\bQ})$ seems very difficult and out of reach with our methods.
However, we recall from Corollary~\ref{coro:low_dim_paired} that $H_{\st}^{1}(\Lambda^{5}\tilde{H}_{\bQ})=H_{\st}^{3}(\Lambda^{5}\tilde{H}_{\bQ})=0$.

\subsubsection{$\Tor$-group computations}\label{s:Tor_groups_lambda_5}
We now compute the $\Tor$-goups of the $\Sym_{\bQ}(\cE)$-module $H_{\st}^{\even}(\Lambda^{5}\tilde{H}_{\bQ})$.
Namely, our goal is to show the following result, whose proof is decomposed in several computations and occupies the rest of the present section.
\begin{thm}\label{thm:5th-ext_Tor}
For the $\Sym_{\bQ}(\cE)$-module $H_{\st}^{*}(\Lambda^{5}\tilde{H}_{\bQ})$, with the convention that $\Lambda^{l}\cE_{\geq 4}=0$ for $l<0$, we have for any $j \geq1$
\begin{eqnarray*}
   H_{j}(\Sym_{\bQ}(\cE);H_{\st}^{*}(\Lambda^{5}\tilde{H}_{\bQ}))&\cong& (\Lambda^{j-3}\cE_{\geq 4})^{\oplus 6}\oplus(\Lambda^{j-2}\cE_{\geq 4})^{\oplus 17}\oplus (\Lambda^{j-1}\cE_{\geq 4})^{\oplus 21}\\
   &&\oplus(\Lambda^{j}\cE_{\geq 4})^{\oplus 21}\oplus (\Lambda^{j+1}\cE_{\geq 4})^{\oplus 15}\oplus(\Lambda^{j+2}\cE_{\geq 4})^{\oplus 4}.
\end{eqnarray*}
Moreover, $H_{0}(\Sym_{\bQ}(\cE);H_{\st}^{*}(\Lambda^{5}\tilde{H}_{\bQ}))\cong\bQ^{\oplus 21}\oplus \cE_{\geq 4}^{\oplus 15}\oplus(\Lambda^{2}\cE_{\geq 4})^{\oplus 4} \oplus\mathscr{S}_{5}$ with
\begin{align*}
\mathscr{S}_{5} \;:=\;& \bQ\{[\mm_{a-1,3}\mm_{b,2} - \mm_{a-2,4}\mm_{b+1,1} - \mm_{b-1,3}\mm_{a,2} + \mm_{b-2,4}\mm_{a+1,1}];b  > a \geq 3 \}
\\  & \oplus\bQ\{[\hat{M}_{a,b,\alpha}]\mid 2 \leq a \leq b, \alpha\geq 1 \} \oplus\bQ\{[\hat{M}_{1,b,\alpha}]\mid b \geq3, \alpha\geq 1 \, ;\, b\in\{1,2\},\alpha\geq 4 \} \\
& \oplus\bQ\{[\hat{M}_{0,b,\alpha}]\mid b\geq 2, \alpha \geq 4 \} \oplus \bQ\{[3\hat{M}_{0,b,2}-\hat{M}_{0,1,b+1}]\mid b\geq 3 \}\\
& \oplus \bQ\{[\hat{M}_{1,2,2}-\hat{M}_{1,1,3}] ; [\hat{M}_{1,2,2}-\hat{M}_{0,2,3}];[\hat{M}_{0,1,3}+20M_{1,1,2}-15\hat{M}_{0,2,2}-8M_{1,2,2}]\}\\
& \oplus\bQ\{[\hat{N}_{b,c}]\mid c > b\geq1 \}\oplus \bQ\{[2\hat{M}_{0,b,1}-\hat{N}_{0,b-1}]\mid b\geq 3 \}
\\  & \oplus\bQ\{[\hat{R}_{a, \alpha, \beta, \gamma}]\mid a\geq 0, 1\leq \alpha\leq \beta\leq \gamma,\\
&\hskip 25mm (a, \alpha, \beta, \gamma) \neq (2,1,1,1), (1,1,2,2), (1,1,1,2), (1,1,1,1), \\
& \hskip 45.5mm (0,1,1,3), (0,2,2,2), (0,1,2,2), (0,1,1,2)\}\\
& \oplus \bQ\{[\hat{R}_{2,1,1,1}+3\hat{R}_{0,1,1,3}]; [12\hat{R}_{1,1,2,2}+7\hat{R}_{0,2,2,2}]\}\\
& \oplus \bQ\{[5\hat{R}_{1,1,1,2}+2\hat{R}_{0,1,2,2}]; [\hat{R}_{1,1,1,1}+3\hat{R}_{0,1,1,2}]\}\\
& \oplus\bQ\{[\mm_{\alpha,1}\mm_{\beta,1}\mm_{\gamma,1}
\mm_{\epsilon,1}\mm_{\xi,1}]\mid 1\leq  \alpha \leq \beta \leq \gamma \leq \epsilon \leq \xi\,\,\text{ such that }\,\,\xi\geq2\}\\
& \oplus \bQ^{\oplus 21}\oplus \cE_{\geq 4}^{\oplus 15}\oplus(\Lambda^{2}\cE_{\geq 4})^{\oplus 4},
\end{align*}
where $\hat{M}_{a,b,\alpha} := \mm_{a+1,2}\mm_{b,2}\mm_{\alpha,1}-\mm_{b-1,3}\mm_{a+2,1}\mm_{\alpha,1} - \mm_{a,3} \mm_{b+1,1} \mm_{\alpha,1}$, 
$\hat{N}_{b,c} := \mm_{0,2}\mm_{b+1,2}\mm_{c+2,1}-\mm_{b,3}\mm_{1,1}\mm_{c+2,1} -\mm_{0,2}\mm_{c+1,2}\mm_{b+2,1}+\mm_{c,3}\mm_{1,1}\mm_{b+2,1}$
and $\hat{R}_{a, \alpha, \beta, \gamma} := 3\mm_{a,2}\mm_{\alpha,1}\mm_{\beta,1}\mm_{\gamma,1}-\mm_{a+1,1}Q_{\alpha,\beta,\gamma}$.
\end{thm}

\begin{notation}
Following Notation~\ref{not:Lambda_5}, we respectively denote by $\fL'$ and $\fK'$ the canonical torsion $\Sym_{\bQ}(\cE_{\leq 3})$-modules such that $\fL\cong \fL'\otimes \bQ_{\cE_{\geq 4}}$ and $\fK\cong\fK'\otimes \bQ_{\cE_{\geq 4}}$, where $\bQ_{\cE_{\geq 4}}$ denotes the trivial $\Sym_{\bQ}(\cE_{\geq 4})$-module. Then, following the decomposition of $\fL$ in Notation~\ref{not:Lambda_5}, we denote by $\fL'_{m}$ the summand of $\fL'$ corresponding to $m\in\{\mm_{0,4}; \mm_{0,3}\mm_{1,1}; \mm_{0,3}\mm_{2,1}; \mm_{0,2}\mm_{0,2}\}$.
\end{notation}

The proof of Theorem~\ref{thm:5th-ext_Tor} begins with the observation from Theorem~\ref{thm:5th-ext-even} that the $\Sym_{\bQ}(\cE)$-module $H_{\st}^{\even}(\Lambda^{5}\tilde{H}_{\bQ})$ is isomorphic to $(\fL'/\fK')\otimes_{\bQ}\bQ_{\cE_{\geq 4}}$. Hence it is enough to compute the $\Tor$-groups of $\fL'/\fK'$ and we will be done using the formulas \eqref{eq:computation_Tor_general}.

\paragraph*{$\Tor$-groups for $\fL'_{\mm_{0,3}\mm_{2,1}}$.}
We consider the torsion $\Sym_{\bQ}(\cE_{\leq 3})$-module
$$\fL'_{\mm_{0,3}\mm_{2,1}}:=\Sym_{\bQ}(\cE_{\leq 3})/(e_{1}^2, e_{2}, e_{3}).$$
\begin{lem}\label{lem:Tor_computations_L'_2}
The group $H_{j}(\Sym_{\bQ}(\cE_{\leq 3}); \fL'_{\mm_{0,3}\mm_{2,1}})$ is equal to
$$\begin{cases}
\bQ[1] \cong \bQ & \text{if $j = 0$}, \\
\langle [de_{2}], [de_{3}], [e_{1}de_{1}]\rangle\cong \bQ^3 & \text{if $j = 1$}, \\
\langle [e_{1}de_{1}\wedge de_{2}], [e_{1}de_{1}\wedge de_{3}], 
[de_{2}\wedge de_{3}]\rangle\cong \bQ^3 & \text{if $j = 2$}, \\
\langle [e_{1}de_{1}\wedge de_{2}\wedge de_{3}]\rangle \cong \bQ & \text{if $j = 3$}, \\
0 & \text{otherwise}.
\end{cases}$$
\end{lem}
\begin{proof}
Using the finite chain complex \eqref{eq:chain_complex_Tor} and formulas \eqref{eq:iD}, the result follows from Lemma~\ref{lem:refined_Tor_method} via the following computations for all $c, c', b_{k}, b'_{k}, a_{k}, a'_{k} \in \bQ$ with $1\leq k\leq 3$:
\begin{eqnarray*}
&& \pa_{3}((c+c'e_{1})(de_{1}\wedge de_{2}\wedge de_{3}))
= ce_{1}(de_{2}\wedge de_{3}),\\
&& \pa_{2}((b_1+b'_1e_{1})(de_{2}\wedge de_{3}))
= 0,\\
&& \pa_{2}((b_2+b'_2e_{1})(de_{3}\wedge de_{1}))
= -b_2e_{1}de_{3},\\
&& \pa_{2}((b_3+b'_3e_{1})(de_{1}\wedge de_{2}))
= b_3 e_{1} de_{2},\\
&& \pa_{1}(\sum^3_{k=1}(a_{k}+a'_{k}e_{1})de_{k}) 
= a_1e_{1}.
\end{eqnarray*}
\end{proof}

\paragraph*{$\Tor$-groups for $\fL'_{\mm_{0,3}\mm_{1,1}}$.}
We consider the torsion $\Sym_{\bQ}(\cE_{\leq 3})$-module $$\fL'_{\mm_{0,3}\mm_{1,1}}:=\Sym_{\bQ}(\cE_{\leq 3})/(e_{1}^2, e_{2}^2, e_{1}e_{2}, e_{3}).$$
\begin{lem}\label{lem:Tor_computations_L'_1}
The group $H_{j}(\Sym_{\bQ}(\cE_{\leq 3}); \fL'_{\mm_{0,3}\mm_{1,1}})$ is equal to
$$\begin{cases}
\bQ[1] \cong \bQ & \text{if $j = 0$}, \\
\langle [e_{1}de_{1}], [e_{2}de_{2}], [de_{3}], 
\tfrac{1}{2}[e_{1}de_{2}+e_{2}de_{1}]\rangle\cong \bQ^4 & \text{if $j = 1$}, \\
\langle [e_{1} de_{1}\wedge de_{2}], [e_{2} de_{2}\wedge de_{1}], 
[e_{1} de_{1}\wedge de_{3}], [e_{2} de_{2}\wedge de_{3}], 
[(e_{1}de_{2}+e_{2}de_{1})\wedge de_{3}]
\rangle\cong \bQ^5 & \text{if $j = 2$}, \\
\langle [e_{1}de_{1}\wedge de_{2}\wedge de_{3}], 
[e_{2}de_{1}\wedge de_{2}\wedge de_{3}]\rangle \cong \bQ^2 & \text{if $j = 3$}, \\
0 & \text{otherwise}.
\end{cases}$$
\end{lem}
\begin{proof}
Using the finite chain complex \eqref{eq:chain_complex_Tor} and formulas \eqref{eq:iD}, the result follows from Lemma~\ref{lem:refined_Tor_method} by computing that for all $c, c', c'', b_{k}, b'_{k}, b''_{k}, a_{k}, a'_{k}, a''_{k} \in \bQ$ with $1\leq k\leq 3$:
\begin{eqnarray*}
&& \pa_{3}((c+c'e_{1}+c''e_{2})(de_{1}\wedge de_{2}\wedge de_{3}))
= c e_{1}(de_{2}\wedge de_{3})+c e_{2}(de_{3}\wedge de_{1}),\\
&& \pa_{2}((b_1+b'_1e_{1}+b''_1e_{2})(de_{2}\wedge de_{3}))
= b_1e_{2}de_{3},\\
&& \pa_{2}((b_2+b'_2e_{1}+b''_2e_{2})(de_{3}\wedge de_{1}))
= -b_2e_{1}de_{3},\\
&& \pa_{2}((b_3+b'_3e_{1}+b''_3e_{2})(de_{1}\wedge de_{2}))
= b_3(e_{1}de_{2} - e_{2}de_{1}),\\
&& \pa_{1}(\sum^3_{k=1}(a_{k}+a'_{k}e_{1}+a''_{k}e_{2})de_{k}) 
= a_1e_{1}+a_2e_{2}.
\end{eqnarray*}
\end{proof}

\paragraph*{$\Tor$-groups for $\fL'_{\mm_{0,4};\mm_{0,2}\mm_{0,2}}$.}
The $\Sym_{\bQ}(\cE_{\leq 3})$-module $\fL'_{\mm_{0,4};\mm_{0,2}\mm_{0,2}}$ is isomorphic to the torsion $\Sym_{\bQ}(\cE_{\leq 3})$-module $\Sym_{\bQ}(\cE_{\leq 3})/(e_{1}^3, e_{2}^2, e_{3}^2, e_{1}e_{2}, e_{1}e_{3}, e_{2}e_{3})$.
\begin{lem}\label{lem:Tor_computations_L'_3}
The group $H_{j}(\Sym_{\bQ}(\cE_{\leq 3}); \fL'_{\mm_{0,4};\mm_{0,2}\mm_{0,2}})$ is equal to
$$\begin{cases}
\bQ[1] \cong \bQ & \text{if $j = 0$}, \\
\langle [e_{1}^2de_{1}], [e_{2}de_{2}], [e_{3}de_{3}], 
\tfrac{1}{2}[e_{1}de_{2}+e_{2}de_{1}], 
\tfrac{1}{2}[e_{2}de_{3}+e_{3}de_{2}], 
\tfrac{1}{2}[e_{3}de_{1}+e_{1}de_{3}]\rangle\cong \bQ^6, 
& \text{if $j = 1$}, \\
\langle [e_{1}^2 de_{1}\wedge de_{2}],[e_{1}^2 de_{1}\wedge de_{3}],[e_{2} de_{2}\wedge de_{1}],[e_{2} de_{2}\wedge de_{3}],[e_{3} de_{3}\wedge de_{1}],[e_{3} de_{3}\wedge de_{2}]\rangle\\ 
\oplus \langle [e_{2}de_{3}\wedge de_{1}], [e_{3}de_{1}\wedge de_{2}]\rangle\cong \bQ^8, & \text{if $j = 2$}, \\
\langle [e_{1}^2de_{1}\wedge de_{2}\wedge de_{3}], 
[e_{2}de_{1}\wedge de_{2}\wedge de_{3}],
[e_{3}de_{1}\wedge de_{2}\wedge de_{3}]\rangle \cong \bQ^3 & \text{if $j = 3$}, \\
0 & \text{otherwise}.
\end{cases}$$

\end{lem}
\begin{proof}
Using the finite chain complex \eqref{eq:chain_complex_Tor} and formulas \eqref{eq:iD}, the result follows from Lemma~\ref{lem:refined_Tor_method} by computing that for all $c, \tilde{c}, c', c'', c''', b_{k}, \tilde{b}_{k}, b'_{k}, b''_{k}, b'''_{k}, a_{k}, \tilde{a}_{k}, a'_{k}, a''_{k}, a'''_{k} \in \bQ$ with $1\leq k\leq 3$:
\begin{eqnarray*}
&& \pa_{3}((c+\tilde{c} e_{1}^2+c'e_{1}+c''e_{2}+c'''e_{3})(de_{1}\wedge de_{2}\wedge de_{3}))\\
&& = c(e_{1}de_{2}\wedge de_{3}+e_{2}de_{3}\wedge de_{1}+e_{3}de_{1}\wedge de_{2})
+ c'e_{1}^2de_{2}\wedge de_{3},\\
&& \pa_{2}((b_1+\tilde{b}_1e_{1}^2+b'_1e_{1}+b''_1e_{2}+b'''_1e_{3})(de_{2}\wedge de_{3}))
= b_1(e_{2}de_{3} - e_{3}de_{2}),\\
&& \pa_{2}((b_2+\tilde{b}_2e_{1}^2+b'_2e_{1}+b''_2e_{2}+b'''_2e_{3})(de_{3}\wedge de_{1}))
= b_2(e_{3}de_{1} - e_{1}de_{3}) - b'_2e_{1}^2de_{3},\\
&& \pa_{2}((b_3+\tilde{b}_3e_{1}^2+b'_3e_{1}+b''_3e_{2}+b'''_3e_{3})(de_{1}\wedge de_{2}))
= b_3(e_{1}de_{2} - e_{2}de_{1}) + b'_3e_{1}^2de_{2},\\
&& \pa_{1}(\sum^3_{k=1}(a_{k}
+\tilde{a}_{k}e_{1}^2+a'_{k}e_{1}+a''_{k}e_{2}+a'''_{k}e_{3})de_{k}) 
= a'_1e_{1}^2+a_1e_{1}+a_2e_{2}+a_3e_{3}.
\end{eqnarray*}
\end{proof}

\paragraph*{$\Tor$-groups for $\fL'/\fK'$.}
Using Lemmas~\ref{lem:Tor_computations_L'_1}--\ref{lem:Tor_computations_L'_3}, we can now compute the $\Tor$-groups for the $\Sym_{\bQ}(\cE_{\leq 3})$-module $\fL'/\fK'$.
\begin{prop}\label{lem:Tor_computations_L'_final}
We have
$$H_{j}(\Sym_{\bQ}(\cE_{\leq 3}); \fL'/\fK')
\cong \begin{cases}
\bQ^4 & \text{if $j = 0$}, \\
\bQ^{15} & \text{if $j = 1$}, \\
\bQ^{17} & \text{if $j = 2$}, \\
\bQ^6 & \text{if $j = 3$}, \\
0 & \text{otherwise}.
\end{cases}$$
\end{prop}
\begin{proof}
We consider the three elements $\rho_1 := 2 e_{2}\mm_{0,3}\mm_{1,1} + 3 e_{1}\mm_{0,3}\mm_{2,1}$, $\rho_2 :=  e_{2}\mm_{0,2}\mm_{0,2} + 6 e_{1}\mm_{0,3}\mm_{1,1}$ and $\rho_3 :=  e_{3}\mm_{0,2}\mm_{0,2} {-6}e_{1}^2\mm_{0,4}$ in $\fL'$. For each $1\leq k\leq 3$, we denote by $\bQ_{\rho_{k}}$ the trivial $\Sym_{\bQ}(\cE_{\leq 3})$-submodule of $\fL'$ generated by $\rho_{k}$, by $\iota_{k}: \bQ_{\rho_{k}} \to \fL'$ the map sending $1$ to $\rho_{k}$, by $(\iota_{k})_{j}$ the induced map $H_{j}(\Sym_{\bQ}(\cE_{\leq 3}); \iota_{k})$, and by $\iota$ the direct sum $\iota_{1}\oplus\iota_{2}\oplus\iota_{3}$. In particular, we note that $\fK'\cong \bQ_{\rho_{1}}\oplus \bQ_{\rho_{2}}\oplus \bQ_{\rho_{3}}\cong \langle\rho_1,\rho_2,\rho_3\rangle$ and that $H_{j}(\Sym_{\bQ}(\cE_{\leq 3}); \bQ_{\rho_{k}})\cong \Lambda^{j}\langle \{de_{1},de_{2},de_{3} \}\rangle$ by Lemma~\ref{lem:refined_Tor_method}. 

We recall that the $\Tor$-groups $H_{*}(\Sym_{\bQ}(\cE_{\leq 3}); \fL')$ are computed from Lemmas~\ref{lem:Tor_computations_L'_2}, \ref{lem:Tor_computations_L'_1} and \ref{lem:Tor_computations_L'_3}.
So we compute the $\Tor$-groups $H_{*}(\Sym_{\bQ}(\cE_{\leq 3}); \fL'/\fK')$ by using the $H_{*}(\Sym_{\bQ}(\cE_{\leq 3}); -)$-long exact sequence associated with short exact sequences of $\Sym_{\bQ}(\cE_{\leq 3})$-modules $\fK'\hookrightarrow\fL'\twoheadrightarrow\fL'/\fK'$ as follows. Each map $\iota_{k}$ induces a chain map from the \eqref{eq:differential_forms_resolution}-type chain complex for $\bQ_{\rho_{k}}$ to the \eqref{eq:chain_complex_Tor}-type chain complex for $\fL'$, that we then use to make the calculations below via formal straightforward computations on formulas \eqref{eq:iD}.
We will also make use of the following filtration on $\fL'$ to compute the cokernel of each map $(\iota)_{k}$:
\begin{equation}\label{eq:filtration_L'}
\langle \mm_{0,3}\mm_{2,1}\rangle \subset 
\langle \mm_{0,3}\mm_{2,1}, \mm_{0,3}\mm_{1,1}\rangle\subset
\langle \mm_{0,3}\mm_{2,1}, \mm_{0,3}\mm_{1,1}, \mm_{0,2}\mm_{0,2}\rangle \subset
\fL'.
\end{equation}

\underline{For $H_{0}(\Sym_{\bQ}(\cE_{\leq 3}); \fL'/\fK')$:} It is clear from the computations in the proofs of Lemmas~\ref{lem:refined_Tor_method} that the image of each $\rho_{k}\in \fK'$ is hit by $\pa_{1}$. Therefore, we have $(\iota_{1})_{0}(1) = (\iota_{2})_{0}(1) = (\iota_{3})_{0}(1) = 0$.
Hence we obtain the $\bQ$-basis $\{[\mm_{0,3}\mm_{2,1}], [\mm_{0,3}\mm_{1,1}], [\mm_{0,2}\mm_{0,2}], [\mm_{0,4}]\}$ for $H_{0}(\Sym_{\bQ}(\cE_{\leq 3});\fL'/\fK')$, and that the connecting homomorphism $H_{1}(\Sym_{\bQ}(\cE_{\leq 3});\fL'/\fK') \to H_{0}(\Sym_{\bQ}(\cE_{\leq 3});\fK'))$ is surjective.

\underline{For $H_{1}(\Sym_{\bQ}(\cE_{\leq 3});\fL'/\fK')$:} using the results and relations for $H_{1}(\Sym_{\bQ}(\cE_{\leq 3}); \fL')$ deduced Lemmas~\ref{lem:Tor_computations_L'_2}, \ref{lem:Tor_computations_L'_1} and \ref{lem:Tor_computations_L'_3}, we compute that:
\begin{eqnarray*}
&& (\iota_{1})_{1}(de_{1}) =  3 [e_{1}de_{1}]\mm_{0,3}\mm_{2,1}+ [e_{1}de_{2}+e_{2}de_{1}]\mm_{0,3}\mm_{1,1};\\
&& (\iota_{1})_{1}(de_{2}) =  2 [e_{2}de_{2}]\mm_{0,3}\mm_{1,1};\\
&& (\iota_{1})_{1}(de_{3}) = 0;\\
&& (\iota_{2})_{1}(de_{1}) =  6 [e_{1}de_{1}]\mm_{0,3}\mm_{1,1}+\tfrac{1}{2} [e_{1}de_{2}+e_{2}de_{1}]\mm_{0,2}\mm_{0,2};\\
&& (\iota_{2})_{1}(de_{2}) = 3 [e_{1}de_{2}+e_{2}de_{1}]\mm_{0,3}\mm_{1,1}+ [e_{2}de_{2}]\mm_{0,2}\mm_{0,2};\\
&& (\iota_{2})_{1}(de_{3}) =  \tfrac{1}{2} [e_{2}de_{3}+e_{3}de_{2}]\mm_{0,2}\mm_{0,2};\\
&& (\iota_{3})_{1}(de_{1}) = \tfrac{1}{2} [e_{1}de_{3}+e_{3}de_{1}]\mm_{0,2}\mm_{0,2} -6 [e_{1}^2de_{1}]\mm_{0,4};\\
&& (\iota_{3})_{1}(de_{2}) = \tfrac{1}{2} [e_{2}de_{3}+e_{3}de_{2}]\mm_{0,2}\mm_{0,2}=(\iota_{2})_{1}(de_{3});\\
&& (\iota_{3})_{1}(de_{3}) =   [e_{3}de_{3}]\mm_{0,2}\mm_{0,2}.
\end{eqnarray*}
We deduce that the kernel of the map $(\iota)_{1}\colon H_{1}(\Sym_{\bQ}(\cE_{\leq 3});\fK') \to H_{1}(\Sym_{\bQ}(\cE_{\leq 3});\fL'))$ is isomorphic to $\bQ^{2}$.
Furthermore, by using the filtration \eqref{eq:filtration_L'}, we deduce that the kernel of the connecting homomorphism $H_{1}(\Sym_{\bQ}(\cE_{\leq 3});\fL'/\fK') \twoheadrightarrow H_{0}(\Sym_{\bQ}(\cE_{\leq 3});\fK'))\cong \bQ^{3}$ has the following basis:
\begin{eqnarray*}
&& [e_{1}de_{1}]\mm_{0,3}\mm_{2,1};\,\,[de_{2}]\mm_{0,3}\mm_{2,1};\,\, [de_{3}]\mm_{0,3}\mm_{2,1};\\
&& [e_{1}de_{1}]\mm_{0,3}\mm_{1,1};\,\, [de_{3}]\mm_{0,3}\mm_{1,1};\\
&&[e_{1}^2de_{1}]\mm_{0,2}\mm_{0,2};\,\, \tfrac{1}{2}[e_{3}de_{1}+e_{1}de_{3}]\mm_{0,2}\mm_{0,2}; \\
&& [e_{2}de_{2}]\mm_{0,4};\,\, [e_{3}de_{3}]\mm_{0,4};\,\,  \tfrac{1}{2}[e_{1}de_{2}+e_{2}de_{1}]\mm_{0,4};\,\,  \tfrac{1}{2}[e_{2}de_{3}+e_{3}de_{2}]\mm_{0,4};\,\, \tfrac{1}{2}[e_{3}de_{1}+e_{1}de_{3}]\mm_{0,4}.
\end{eqnarray*}
This provides the computation of $H_{1}(\Sym_{\bQ}(\cE_{\leq 3}); \fL'/\fK')$.

\underline{For $H_{2}(\Sym_{\bQ}(\cE_{\leq 3});\fL'/\fK')$:} using the results and relations for $H_{2}(\Sym_{\bQ}(\cE_{\leq 3}); \fL')$ deduced Lemmas~\ref{lem:Tor_computations_L'_2}, \ref{lem:Tor_computations_L'_1} and \ref{lem:Tor_computations_L'_3}, we compute that:
\begin{eqnarray*}
&& (\iota_{1})_{2}(de_{2}\wedge de_{3}) =  2 [e_{2}de_{2}\wedge de_{3}]\mm_{0,3}\mm_{1,1};\\
&& (\iota_{1})_{2}(de_{3}\wedge de_{1}) =  3 [e_{1}de_{3}\wedge de_{1}]\mm_{0,3}\mm_{2,1} -  [(e_{1}de_{2}+e_{2}de_{1})\wedge de_{3}]\mm_{0,3}\mm_{1,1};\\
&& (\iota_{1})_{2}(de_{1}\wedge de_{2}) =  3 [e_{1}de_{1}\wedge de_{2}]\mm_{0,3}\mm_{2,1} + 2 [e_{2}de_{1}\wedge de_{2}]\mm_{0,3}\mm_{1,1};\\
&& (\iota_{2})_{2}(de_{2}\wedge de_{3}) =  3 [(e_{1}de_{2}+e_{2}de_{1})\wedge de_{3}]\mm_{0,3}\mm_{1,1}+ [e_{2}de_{2}\wedge de_{3}]\mm_{0,2}\mm_{0,2};\\
&& (\iota_{2})_{2}(de_{3}\wedge de_{1}) = 6 [e_{1}de_{3}\wedge de_{1}]\mm_{0,3}\mm_{1,1} + [e_{2}de_{3}\wedge de_{1}]\mm_{0,2}\mm_{0,2}, \\
&& (\iota_{2})_{2}(de_{1}\wedge de_{2}) =   6 [e_{1}de_{1}\wedge de_{2}]\mm_{0,3}\mm_{1,1} + [e_{2}de_{1}\wedge de_{2}]\mm_{0,2}\mm_{0,2};\\
&& (\iota_{3})_{2}(de_{2}\wedge de_{3}) =  [e_{3}de_{2}\wedge de_{3}]\mm_{0,2}\mm_{0,2};\\
&& (\iota_{3})_{2}(de_{3}\wedge de_{1}) =  [e_{3}de_{3}\wedge de_{1}]\mm_{0,2}\mm_{0,2} - 6 [e_{1}^2de_{3}\wedge de_{1}]\mm_{0,4};\\
&& (\iota_{3})_{2}(de_{1}\wedge de_{2}) =  [e_{3}de_{1}\wedge de_{2}]\mm_{0,2}\mm_{0,2} - 6 [e_{1}^2de_{1}\wedge de_{2}]\mm_{0,4}.
\end{eqnarray*}
Hence the kernel of the map $(\iota)_{2}\colon H_{2}(\Sym_{\bQ}(\cE_{\leq 3});\fK') \to H_{2}(\Sym_{\bQ}(\cE_{\leq 3});\fL'))$ vanishes.
Also, using the filtration \eqref{eq:filtration_L'}, we obtain the following basis for the kernel of the connecting homomorphism $H_{2}(\Sym_{\bQ}(\cE_{\leq 3});\fL'/\fK') \twoheadrightarrow \bQ^{2}$:
\begin{eqnarray*}
&& [e_{1} de_{1}\wedge de_{2}]m_{0,3}m_{2,1}; \,\, [e_{1} de_{1}\wedge de_{3}]m_{0,3}m_{2,1}; \,\, [de_{2}\wedge de_{3}]m_{0,3}m_{2,1};\\
&& [e_{1} de_{1}\wedge de_{2}]m_{0,3}m_{1,1}; \,\, [e_{1} de_{1}\wedge de_{3}]m_{0,3}m_{1,1}; \\
&& [e_{1}^2 de_{1}\wedge de_{2}]m_{0,2}m_{0,2}; \,\, [e_{1}^2 de_{1}\wedge de_{3}]m_{0,2}m_{0,2}; \,\, [e_{3} de_{3}\wedge de_{1}]m_{0,2}m_{0,2}; [e_{3}de_{1}\wedge de_{2}]m_{0,2}m_{0,2}; \,\,\\
&& [e_{2} de_{2}\wedge de_{3}]m_{0,4}; \,\, [e_{2} de_{2}\wedge de_{1}]m_{0,4}; \,\, [e_{3} de_{3}\wedge de_{1}]m_{0,4}; \\
&& [e_{3} de_{3}\wedge de_{2}]m_{0,4}; \,\, [e_{2}de_{3}\wedge de_{1}]m_{0,4};\,\, [e_{3}de_{1}\wedge de_{2}]m_{0,4};\,\,
\end{eqnarray*}

\underline{For $H_{3}(\Sym_{\bQ}(\cE_{\leq 3}); \fL'/\fK')$:} using the results and relations for $H_{3}(\Sym_{\bQ}(\cE_{\leq 3}); \fL')$ deduced Lemmas~\ref{lem:Tor_computations_L'_2}, \ref{lem:Tor_computations_L'_1} and \ref{lem:Tor_computations_L'_3}, we compute that:
\begin{eqnarray*}
&& (\iota_{1})_{3}(de_{1}\wedge de_{2}\wedge de_{3}) =  3 [e_{1}de_{1}\wedge de_{2}\wedge de_{3}]\mm_{0,3}\mm_{2,1} +  2 [e_{2}de_{1}\wedge de_{2}\wedge de_{3}]\mm_{0,3}\mm_{1,1};\\
&& (\iota_{2})_{3}(de_{1}\wedge de_{2}\wedge de_{3}) = 6 [e_{1}de_{1}\wedge de_{2}\wedge de_{3}]\mm_{0,3}\mm_{1,1} + [e_{2}de_{1}\wedge de_{2}\wedge de_{3}]\mm_{0,2}\mm_{0,2};\\
&& (\iota_{3})_{3}(de_{1}\wedge de_{2}\wedge de_{3}) = [e_{3}de_{1}\wedge de_{2}\wedge de_{3}]\mm_{0,2}\mm_{0,2} -6[e_{1}^2de_{1}\wedge de_{2}\wedge de_{3}]\mm_{0,4}.
\end{eqnarray*}
In particular, the kernel of the map $(\iota)_{3}\colon H_{3}(\Sym_{\bQ}(\cE_{\leq 3});\fK') \to H_{3}(\Sym_{\bQ}(\cE_{\leq 3});\fL'))$ vanishes.
Furthermore, by using the filtration \eqref{eq:filtration_L'}, we deduce that the kernel of the connecting homomorphism $H_{3}(\Sym_{\bQ}(\cE_{\leq 3});\fL'/\fK') \to H_{2}(\Sym_{\bQ}(\cE_{\leq 3});\fK'))\cong \bQ^{3}$ has the following basis:
\begin{eqnarray*}
&& [e_{1}de_{1}\wedge de_{2}\wedge de_{3}]\mm_{0,3}\mm_{2,1}; \\
&& [e_{1}de_{1}\wedge de_{2}\wedge de_{3}]\mm_{0,3}\mm_{1,1}; \\
&& [e_{1}^2de_{1}\wedge de_{2}\wedge de_{3}]\mm_{0,2}\mm_{0,2};\,\,
[e_{3}de_{1}\wedge de_{2}\wedge de_{3}]\mm_{0,2}\mm_{0,2}; \\
&& [e_{2}de_{1}\wedge de_{2}\wedge de_{3}]\mm_{0,4};\,\,
[e_{3}de_{1}\wedge de_{2}\wedge de_{3}]\mm_{0,4}.
\end{eqnarray*}

\underline{For $H_{j}(\Sym_{\bQ}(\cE_{\leq 3}); \fL'/\fK')$ with $j\geq4$:} since $H_{j}(\Sym_{\bQ}(\cE_{\leq 3});\fK'))=H_{j}(\Sym_{\bQ}(\cE_{\leq 3});\fL'))=0$ for $j\geq4$ while the kernel of the map $(\iota)_{3}\colon H_{3}(\Sym_{\bQ}(\cE_{\leq 3});\fK') \to H_{3}(\Sym_{\bQ}(\cE_{\leq 3});\fL'))$ is trivial, we deduce from the long exact sequnce for the $\Tor$-groups that $H_{j}(\Sym_{\bQ}(\cE_{\leq 3});\fL'/\fK'))=0$ which ends the proof.
\end{proof}

\paragraph*{Proof of Theorem~\ref{thm:5th-ext_Tor}.}
We are now ready to compute the $\Tor$-groups for the $\Sym_{\bQ}(\cE)$-module $H_{\st}^{*}(\Lambda^{5}\tilde{H}_{\bQ})$.
Using the computations \eqref{eq:computation_Tor_general}, it follows from Proposition~\ref{lem:Tor_computations_L'_final} that for all $j \geq0$
$$H_{j}(\Sym_{\bQ}(\cE);H_{\st}^{\even}(\Lambda^{5}\tilde{H}_{\bQ}))\cong (\Lambda^{j-3}\cE_{\geq 4})^{\oplus 6}\oplus(\Lambda^{j-2}\cE_{\geq 4})^{\oplus 17}\oplus (\Lambda^{j-1}\cE_{\geq 4})^{\oplus 15}\oplus(\Lambda^{j}\cE_{\geq 4})^{\oplus 4}.$$
In particular, we have
$H_{0}(\Sym_{\bQ}(\cE);H_{\st}^{\even}(\Lambda^{5}\tilde{H}_{\bQ}))\cong \bQ^{\oplus 4}$.
Combining this computation with the isomorphism \eqref{eq:exterior_Tor}, we deduce that for $i>0$
$$H_{i}(\Sym_{\bQ}(\cE);H_{\st}^{\odd}(\Lambda^{5}\tilde{H}_{\bQ}))\cong (\Lambda^{i-1}\cE_{\geq 4})^{\oplus 6}\oplus(\Lambda^{i}\cE_{\geq 4})^{\oplus 17}\oplus (\Lambda^{i+1}\cE_{\geq 4})^{\oplus 15}\oplus(\Lambda^{i+2}\cE_{\geq 4})^{\oplus 4}.$$
This ends the first part of the theorem.

For the computation of $H_{0}(\Sym_{\bQ}(\cE);H_{\st}^{*}(\Lambda^{5}\tilde{H}_{\bQ}))$, we first have the contribution of the $\Tor$-group $H_{0}(\Sym_{\bQ}(\cE);H_{\st}^{\even}(\Lambda^{5}\tilde{H}_{\bQ}))$ computed above.
By Lemma~\ref{lem:Tor_0_lemma}, we compute the $\Tor_{0}$-group $H_{0}(\Sym_{\bQ}(\cE);H_{\st}^{\odd}(\Lambda^{5}\tilde{H}_{\bQ}))$ thanks to the exact sequence \eqref{eq:Tor_0_ES}. We have already computed $H_{2}(\Sym_{\bQ}(\cE);H_{\st}^{\even}(\Lambda^{5}\tilde{H}_{\bQ})$ above.
Following the decomposition \eqref{eq:decomposition_filtration_Ker_H_{0}}, we now progressively determine the generators of $\Ker(\Delta_{4})$ associated to each summand $\Ker(\ddd_{4,k/k-1})$ for $1\leq k\leq 4$ for $1\leq k\leq 2$ to finish the proof. For this aim, we use the chain complex \eqref{eq:complex_SES} and Proposition~\ref{prop:Ker_Delta_method}.

\underline{For the summand $\Ker(\ddd_{5, 1/0})$:} for each $2 \leq a < b$, the image by $\ddd_{5}$ of the element $\mm_{a-1,3}\mm_{b,2} - \mm_{a-2,4}\mm_{b+1,1} - \mm_{b-1,3}\mm_{a,2} + \mm_{b-2,4}\mm_{a+1,1}$ is equal to $\pa_{1}(\mm_{b-2,4}\otimes e_{a+1} - \mm_{a-2,4}\otimes e_{b+1})$ in $H_{\st}^{\even}(\Lambda^{4}\tilde{H}_{\bQ})$. We deduce from Theorem~\ref{thm:5th-ext-even}, we have $\Pi(\mm_{b-2,4})\otimes e_{a+1} - \Pi(\mm_{a-2,4})\otimes e_{b+1}\neq 0$ if and only if $a=2$ and $b>2$. In that case, they are equal to $\Pi(\mm_{0,4})\otimes e_{b+1}$, and so they are linearly independent.
Hence, by Proposition~\ref{prop:Ker_Delta_method}, the kernel $\Ker(\ddd_{5, 1/0})$ associated to that summand is isomorphic to the free $\bQ$-module generated by $\{[\mm_{a-1,3}\mm_{b,2} - \mm_{a-2,4}\mm_{b+1,1} - \mm_{b-1,3}\mm_{a,2} + \mm_{b-2,4}\mm_{a+1,1}];b  > a \geq 3 \}$.

\underline{For the summand $\Ker(\ddd_{5, 2/1})$:}
Firs of all, we stress that all the following computations are made by using Theorem~\ref{thm:5th-ext-even} and the fact that $\pa_{2}(m\otimes de_{k}\wedge de_l) = e_{k}m\otimes de_{l} - e_{l}m\otimes de_{k}$.

For all integers $0 \leq a \leq b \geq 1$ and $\alpha\geq 1$, we begin with setting $\hat{M}_{a,b,\alpha} := \mm_{a+1,2}\mm_{b,2}\mm_{\alpha,1} - \mm_{a,3} \mm_{b+1,1} \mm_{\alpha,1} -\mm_{b-1,3}\mm_{a+2,1}\mm_{\alpha,1}$. We compute that $\ddd_{5}(\hat{M}_{a,b,\alpha})=\pa_{1}(M_{a,b,\alpha}^{\otimes})$ in $H_{\st}^{\even}(\Lambda^{4}\tilde{H}_{\bQ})$, where
$$M_{a,b,\alpha}^{\otimes}:=(\mm_{a+1,2}\mm_{b,2} -\mm_{a,3}\mm_{b+1,1}  - \mm_{b-1,3}\mm_{a+2,1})\otimes e_{\alpha} -\mm_{a,3}\mm_{\alpha,1}\otimes e_{b+1}  -\mm_{b-1,3}\mm_{\alpha,1}\otimes e_{a+2}.$$
We compute in a similar way to relations \eqref{eq:M_a_b_alpha} that:
$$(\Pi\otimes \cE)(M_{a,b,\alpha}^{\otimes}) =
\begin{cases}
5e_{3}\mm_{0,4}\otimes e_{\alpha}  & \text{if $(a, b, \alpha) = (1,2,\alpha)$ with $\alpha \leq 3$},\\
5e_{2}\mm_{0,4}\otimes e_{3} & \text{if $(a, b, \alpha) = (1,1,3)$},\\
(3e_{2}+2e_{3})\mm_{0,4}\otimes e_{\alpha} & \text{if $(a, b, \alpha) = (1,1,\alpha)$ with $\alpha\leq2$},\\
-\mm_{0,3}\mm_{\alpha,1}\otimes e_{b+1} & \text{if $(a, b, \alpha) = (0,b,\alpha)$ with $b \geq 3$, $\alpha\leq2$},\\
5e_{2}\mm_{0,4}\otimes e_{3}  & \text{if $(a, b, \alpha) = (0,2,3)$},\\
4e_{\alpha}\mm_{0,4}\otimes e_{\alpha} - \mm_{0,3}\mm_{\alpha,1}\otimes e_{3} & \text{if $(a, b, \alpha) = (0,2,\alpha)$ with $\alpha\leq2$},\\
-3\mm_{0,3}\mm_{2,1}\otimes e_{\alpha} & \text{if $(a, b, \alpha) = (0,1,\alpha)$ with $\alpha \geq 3$},\\
-5\mm_{0,3}\mm_{2,1}\otimes e_{2} & \text{if $(a, b, \alpha) = (0,1,2)$},\\
-3\mm_{0,3}\mm_{2,1}\otimes e_{1}-2\mm_{0,3}\mm_{1,1}\otimes e_{2} & \text{if $(a, b, \alpha) = (0,1,1)$},\\
0 & \text{otherwise}.
\end{cases}$$
Now, we set $\hat{N}_{b,c} := \mm_{0,2}\mm_{b+1,2}\mm_{c+2,1}-\mm_{b,3}\mm_{1,1}\mm_{c+2,1} -\mm_{0,2}\mm_{c+1,2}\mm_{b+2,1}+\mm_{c,3}\mm_{1,1}\mm_{b+2,1}$ for all integers $0 \leq b < c$. Then we compute that $\ddd_{5}(\hat{N}_{b,c})=\pa_{1}(N_{b,c}^{\otimes})$ in $H_{\st}^{\even}(\Lambda^{4}\tilde{H}_{\bQ})$, where
$$N_{b,c}^{\otimes}:=(\mm_{c,3}\mm_{1,1}-\mm_{0,2}\mm_{c+1,2})\otimes e_{b+2}-(\mm_{b,3}\mm_{1,1}-\mm_{0,2}\mm_{b+1,2})\otimes e_{c+2}-(\mm_{b,3}\mm_{c+2,1}-\mm_{c,3}\mm_{b+2,1})\otimes e_{1}.$$
We compute in a similar way to relations \eqref{eq:N_b_c} that
$$(\Pi\otimes \cE)(N_{b,c}^{\otimes})=
\begin{cases}
-2\mm_{0,3}\mm_{1,1}\otimes e_{c+2} & \text{if $b = 0$ and $c \geq 2$},\\
-2e_{2}\mm_{0,4}\otimes e_{1} - 2\mm_{0,3}\mm_{1,1}\otimes e_{3} & \text{if $b=0$, $c=1$},\\
0 & \text{otherwise.}
\end{cases}$$
Hence, by Proposition~\ref{prop:Ker_Delta_method}, the contribution of the generators $\hat{M}_{a,b,\alpha}$ and $\hat{N}_{b,c}$ to the kernel $\Ker(\ddd_{5, 2/1})$ is isomorphic to the free $\bQ$-module with basis
$$\aligned
&\{\hat{M}_{a,b,\alpha}\mid 2 \leq a \leq b, \alpha\geq 1\}\cup \{\hat{M}_{1,b,\alpha}\mid b \geq3, \alpha\geq 1\}\cup \{\hat{M}_{1,b,\alpha}\mid 1\leq b\leq2,\alpha\geq 4\}\\
& \cup \{\hat{M}_{0,b,\alpha}\mid b \geq2,\alpha\geq 4\} \cup\{3\hat{M}_{0,b,2}-\hat{M}_{0,1,b+1} ; 2\hat{M}_{0,b,1}-\hat{N}_{0,b-1} \mid b\geq 3\}\\
&\cup\{\hat{M}_{1,2,2}-\hat{M}_{1,1,3} ; \hat{M}_{1,2,2}-\hat{M}_{0,2,3} ; \hat{M}_{0,1,3}+20M_{1,1,2}-15\hat{M}_{0,2,2}-8M_{1,2,2}\}\\
&\cup \{\hat{N}_{b,c}\mid b\geq 1, c\geq2\}.
\endaligned$$

Finally, we set $\hat{P}_{b}:=\mm_{0,2}\mm_{0,2}\mm_{b+2,1}-2\mm_{0,2}\mm_{b+1,2}\mm_{1,1}+2\mm_{b,3}\mm_{1,1}\mm_{1,1}$ for each $b \geq 0$. We have $\ddd_{5}(\hat{P}_{b})=\pa_{1}(P_{b}^{\otimes})$, where
$$P_{b}^{\otimes}:=\mm_{0,2}\mm_{0,2}\otimes e_{b+2} + (4\mm_{b,3}\mm_{1,1} - 2\mm_{0,2}\mm_{b+1,2})\otimes e_{1}.$$
Then, we compute in a similar way to relations \eqref{eq:P_b} that
$$(\Pi\otimes \cE)(P_{b}^{\otimes}) =
\begin{cases}
\mm_{0,2}\mm_{0,2}\otimes e_{b+2} & \text{if $b \geq 2$,}\\
\mm_{0,2}\mm_{0,2}\otimes e_{3} - 6e_{1}\mm_{0,4}\otimes e_{1} & \text{if $b = 1$,}\\
\mm_{0,2}\mm_{0,2}\otimes e_{2} + 6 \mm_{0,3}\mm_{1,1}\otimes e_{1} & \text{if $b = 0$.}
\end{cases}$$
These elements are linearly independent modulo the generators $\hat{M}_{a,b,\alpha}$ and $\hat{N}_{b,c}$ stated above.
Hence, by Proposition~\ref{prop:Ker_Delta_method}, for any $b\geq 0$, the generator $\hat{P}_{b}$ does not contribute to the kernel $\Ker(\ddd_{5, 2/1})$.

\underline{For the summand $\Ker(\ddd_{5, 3/2})$:} we set $\hat{R}_{a, \alpha, \beta, \gamma} := 3\mm_{a,2}\mm_{\alpha,1}\mm_{\beta,1}\mm_{\gamma,1}-\mm_{a+1,1}Q_{\alpha,\beta,\gamma}$  where $Q_{\alpha,\beta,\gamma}$ is defined in \eqref{eqn:Qabc}, for all $a\geq0$ and $1 \leq \alpha \leq \beta \leq \gamma$.
We compute that $\ddd_{5}(\hat{R}_{a, \alpha,\beta, \gamma} ))=\pa_{1}(R_{a,\alpha,\beta,\gamma}^{\otimes})$ in $H_{\st}^{\even}(\Lambda^{4}\tilde{H}_{\bQ})$, where
\begin{eqnarray*}
   &R_{a,\alpha,\beta,\gamma}^{\otimes}:=&3\mm_{a,2}(\mm_{\beta,1}\mm_{\gamma,1}\otimes e_{\alpha}+\mm_{\gamma,1}\mm_{\alpha,1}\otimes e_{\beta}+\mm_{\alpha,1}\mm_{\beta,1}\otimes e_{\gamma})\\
   &&-\mm_{a+1,1}(\mm_{\alpha-1,2}\mm_{\gamma,1}\otimes e_{\beta}+ \mm_{\alpha-1,2}\mm_{\beta,1}\otimes e_{\gamma}+ \mm_{\beta-1,2}\mm_{\alpha,1}\otimes e_{\gamma}\\
   &&\,\,\,\,\,\,\,\,\,\,\,\,\,\,\,\,\,\,\,\,\,\,\,\,\,+\mm_{\beta-1,2}\mm_{\gamma,1}\otimes e_{\alpha})+ \mm_{\gamma-1,2}\mm_{\beta,1}\otimes e_{\alpha}+ \mm_{\gamma-1,2}\mm_{\alpha,1}\otimes e_{\beta})\\
   &&-Q_{\alpha,\beta,\gamma}\otimes e_{a+1}.
\end{eqnarray*}
We recall that we denote by $\ddd_{2-a+1}\circ\cdots \circ \ddd_{2}$ by $\ddd^{a}$ for $a \geq 1$. By similar computations to that of $R_{a, \alpha, \beta, \gamma}$ in the proof of Theorem~\ref{thm:5th-ext-even}, we have:
$$\aligned
& (\Pi\otimes \cE)(R_{a,\alpha,\beta,\gamma}^{\otimes}
+Q_{\alpha, \beta, \gamma}\otimes e_{a+1})\\
& = (-1)^a\mm_{0,a+2}\ddd^a(5\mm_{\beta,1}\mm_{\gamma,1}
+e_{\beta}\mm_{\gamma-1,2} +e_{\gamma}\mm_{\beta-1,2})\otimes e_{\alpha} \\
& \qquad +  (-1)^a\mm_{0,a+2}\ddd^a(5\mm_{\gamma,1}\mm_{\alpha,1}
+e_{\gamma}\mm_{\alpha-1,2} +e_{\alpha}\mm_{\gamma-1,2})\otimes e_{\beta} \\
& \qquad +  (-1)^a\mm_{0,a+2}\ddd^a(5\mm_{\alpha,1}\mm_{\beta,1}
+e_{\alpha}\mm_{\beta-1,2} +e_{\beta}\mm_{\alpha-1,2})\otimes e_{\gamma}
\endaligned$$
All the following computations are made by using Theorem~\ref{thm:5th-ext-even}, the equality \eqref{eqn:Qabc} and the fact that $\pa_{2}(m\otimes de_{k}\wedge de_l) = e_{k}m\otimes de_{l} - e_{l}m\otimes de_{k}$.

If $a \geq 3$, we deduce that $(\Pi\otimes \cE)(R_{a,\alpha,\beta,\gamma}^{\otimes})
= - Q_{\alpha, \beta, \gamma}\otimes e_{a+1}$ by Theorem~\ref{thm:5th-ext-even}. Then, using the computations \eqref{eqn:Qabc}, we may re-write $Q_{\alpha, \beta, \gamma}\otimes e_{a+1}$ so that $e_{a+1}$ is on the left-hand side of the tensor product, multiplying some $\mm_{i,j}$. But for $a\geq3$, the element $e_{a+1}$ annihilates any element of $H_{\st}^{\even}(\Lambda^{5}\tilde{H}_{\bQ})$, so $Q_{\alpha, \beta, \gamma}\otimes e_{a+1}=0$.
Therefore, we have $(\Pi\otimes \cE)(R_{a,\alpha,\beta,\gamma}^{\otimes}) = 0$ if $a \geq 3$.

If $a = 2$, we compute similarly to relations \eqref{eq:R_2_alpha_beta_gamma} that:
$$(\Pi\otimes \cE)(R_{2,\alpha,\beta,\gamma}^{\otimes})= 36e_{\beta} e_{\gamma} \mm_{0,4}\otimes e_{\alpha} - Q_{\alpha, \beta, \gamma}\otimes e_{3}=\begin{cases}
45e_{1}^2\mm_{0,4}\otimes e_{1} & \text{if $(\alpha, \beta, \gamma) = (1,1,1)$}, \\
0 & \text{otherwise}.
\end{cases}$$

If $a = 1$, we compute similarly to relations \eqref{eq:R_1_alpha_beta_gamma} that:
$$(\Pi\otimes \cE)(R_{1,\alpha,\beta,\gamma}^{\otimes})=
\begin{cases}
-14e_{2}\mm_{0,3}\mm_{1,1}\otimes e_{2} & \text{if $(\alpha, \beta, \gamma) = (1,2,2)$}, \\
-16e_{2}\mm_{0,3}\mm_{1,1}\otimes e_{1} & \text{if $(\alpha, \beta, \gamma) = (1,1,2)$}, \\
-45e_{1}\mm_{0,3}\mm_{1,1}\otimes e_{1} & \text{if $(\alpha, \beta, \gamma) = (1,1,1)$}, \\
0 & \text{otherwise}.
\end{cases}$$

If $a = 0$, we compute similarly to relations \eqref{eq:R_0_alpha_beta_gamma} that:

$$(\Pi\otimes \cE)(R_{0,\alpha,\beta,\gamma}^{\otimes})=\begin{cases}
-15e_{1}^2\mm_{0,4}\otimes e_{1} & \text{if $(\alpha, \beta, \gamma) = (1,1,3)$}, \\
24e_{2}\mm_{0,3}\mm_{1,1}\otimes e_{2} & \text{if $(\alpha, \beta, \gamma) = (2,2,2)$}, \\
20e_{2}\mm_{0,3}\mm_{1,1}\otimes e_{1} & \text{if $(\alpha, \beta, \gamma) = (1,2,2)$}, \\
15e_{1}\mm_{0,3}\mm_{1,1}\otimes e_{1} & \text{if $(\alpha, \beta, \gamma) = (1,1,2)$}, \\
0 & \text{otherwise}.
\end{cases}$$
Hence, by Proposition~\ref{prop:Ker_Delta_method}, the kernel $\Ker(\ddd_{5, 3/2})$ associated to that summand is isomorphic to the free $\bQ$-module with basis:
$$\aligned
& \{\hat{R}_{a, \alpha, \beta, \gamma}\mid a\geq 0, 1\leq \alpha\leq \beta\leq \gamma, (a, \alpha, \beta, \gamma) \neq 
(2,1,1,1), (1,1,2,2), (1,1,1,2), (1,1,1,1), \\
& \hskip 70.9mm (0,1,1,3), (0,2,2,2), (0,1,2,2), (0,1,1,2)\}\\
& \cup \{\hat{R}_{2,1,1,1}+3\hat{R}_{0,1,1,3}; 12\hat{R}_{1,1,2,2}+7\hat{R}_{0,2,2,2};
5\hat{R}_{1,1,1,2}+2\hat{R}_{0,1,2,2}; \hat{R}_{1,1,1,1}+3\hat{R}_{0,1,1,2}\}.
\endaligned$$

\underline{For the summand $\Ker(\ddd_{5, 4/3})$:} for each $1\leq \alpha\leq \beta\leq \gamma\leq \epsilon\leq\xi$, we compute that
$$\ddd_{5}(\mm_{\alpha,1}\mm_{\beta,1}\mm_{\gamma,1}\mm_{\epsilon,1}\mm_{\xi,1})=\pa_{1}(C^{\otimes}_{\alpha,\beta,\gamma,\epsilon,\xi})$$
in $H_{\st}^{\even}(\Lambda^{4}H_{\bQ})\otimes\cE$, where
\begin{eqnarray*}    &C^{\otimes}_{\alpha,\beta,\gamma,\epsilon,\xi}:=&\mm_{\alpha,1}\mm_{\beta,1}\mm_{\gamma,1}\mm_{\epsilon,1}\otimes e_{\xi}+\mm_{\alpha,1}\mm_{\beta,1}\mm_{\gamma,1}\mm_{\xi,1}\otimes e_{\epsilon}
+\mm_{\alpha,1}\mm_{\beta,1}\mm_{\epsilon,1}\mm_{\xi,1}\otimes e_{\gamma}\\
&&+\mm_{\alpha,1}\mm_{\gamma,1}\mm_{\epsilon,1}\mm_{\xi,1}\otimes e_{\beta}
+\mm_{\beta,1}\mm_{\gamma,1}\mm_{\epsilon,1}\mm_{\xi,1}\otimes e_{\alpha}.
\end{eqnarray*}

Then, we make the following computations in a similar way to those of the relations \eqref{eq:cok_5/4}, by using Theorem~\ref{thm:5th-ext-even} and the fact that $\pa_{2}(m\otimes de_{k}\wedge de_l) = e_{k}m\otimes de_{l} - e_{l}m\otimes de_{k}$.

If $\xi \geq 4$, since $e_{\xi}$ annihilates any element of $H_{\st}^{\even}(\Lambda^{5}\tilde{H}_{\bQ})$, we have:
\begin{eqnarray*}
(\Pi\otimes \cE)(C^{\otimes}_{\alpha,\beta,\gamma,\epsilon,\xi})&= & \mm_{\alpha,1}\mm_{\beta,1}\mm_{\gamma,1}\mm_{\epsilon,1}\otimes e_{\xi} \\
&=&-e_{\xi}\mm_{\beta,1}\mm_{\gamma,1}\mm_{\epsilon-1,2}\otimes e_{\alpha}
-e_{\xi}\mm_{\gamma,1}\mm_{\alpha,1}\mm_{\epsilon-1,2}\otimes e_{\beta}
-e_{\xi}\mm_{\alpha,1}\mm_{\beta,1}\mm_{\epsilon-1,2}\otimes e_{\gamma}\\
&=&0.
\end{eqnarray*}
If $\xi = 3$, since $e_{3}$ annihilates any element of $H_{\st}^{\even}(\Lambda^{5}\tilde{H}_{\bQ})$, then we have:
\begin{eqnarray*}
(\Pi\otimes \cE)(C^{\otimes}_{\alpha,\beta,\gamma,\epsilon,3})&= & \mm_{\alpha,1}\mm_{\beta,1}\mm_{\gamma,1}\mm_{\epsilon,1}\otimes e_{3}
+12e_{\alpha} e_{\beta} e_{\gamma} \mm_{0,4}\otimes e_{\epsilon} \\
&=&-e_{3}\mm_{\beta,1}\mm_{\gamma,1}\mm_{\epsilon-1,2}\otimes e_{\alpha}
-e_{3}\mm_{\gamma,1}\mm_{\alpha,1}\mm_{\epsilon-1,2}\otimes e_{\beta}
-e_{3}\mm_{\alpha,1}\mm_{\beta,1}\mm_{\epsilon-1,2}\otimes e_{\gamma}\\
&=&0.
\end{eqnarray*}
If $\xi = 2$, we note that each product $e_{a}e_{b}$ for any $a,b\geq1$ annihilates the generators $\mm_{0,3}\mm_{1,1}$ and  $\mm_{0,3}\mm_{2,1}$ of $H_{\st}^{\even}(\Lambda^{5}\tilde{H}_{\bQ})$. Then it is an easy routine to check by analogous computations to those of \eqref{eq:cok_5/4} that $(\Pi\otimes \cE)(C^{\otimes}_{\alpha,\beta,\gamma,\epsilon,2})=0$.

Finally, for $\xi=1$, we compute that $(\Pi\otimes \cE)(C^{\otimes}_{1,1,1,1,1})= - \frac{15}{2}e_{1}^2\mm_{0,2}\mm_{0,2}\otimes e_{1}\neq0$.

Then it follows from Proposition~\ref{prop:Ker_Delta_method} that the kernel $\Ker(\ddd_{5, 3/2})$ associated to the summand $\Ker(\ddd_{5,3/2})$ is isomorphic to the free $\bQ$-module with basis $\{[\mm_{\alpha,1}\mm_{\beta,1}\mm_{\gamma,1}
\mm_{\epsilon,1}\mm_{\xi,1}]\mid 1\leq  \alpha \leq \beta \leq \gamma \leq \epsilon \leq \xi\,\,\text{ such that }\,\,\xi\geq2\}$.

\phantomsection
\addcontentsline{toc}{section}{References}
\renewcommand{\bibfont}{\normalfont\small}
\setlength{\bibitemsep}{0pt}
\printbibliography

\noindent {Nariya Kawazumi, \itshape Department of Mathematical Sciences, 
University of Tokyo, 3-8-1 Komaba, Meguro-ku, Tokyo 153-
8914, Japan}. \noindent {Email address: kawazuminariya@g.ecc.u-tokyo.ac.jp}

\noindent {Arthur Souli{\'e}, \itshape Normandie Univ., UNICAEN, CNRS, LMNO, 14000 Caen, France.}. \noindent {Email address: \tt artsou@hotmail.fr, arthur.soulie@unicaen.fr}

\end{document}